\documentclass[a4paper,12pt]{article}
\usepackage{amsfonts,amssymb,amsmath,amsthm, dsfont,relsize,accents, enumerate, color, tikz}
\usepackage{bm}
\usepackage{mathrsfs}  
\usepackage{upgreek}
\usepackage[english]{babel}
\usepackage{graphicx}
\usepackage{microtype}
\newtheorem  {theorem}       {Theorem}[section]
\newtheorem  {lemma}         [theorem]{Lemma}
\newtheorem  {corollary}     [theorem]{Corollary}
\newtheorem  {proposition}   [theorem]{Proposition}
\newtheorem  {definition}    [theorem]{Definition}
\newtheorem  {remark}        [theorem]{Remark}

\newtheorem  {claim}        [theorem]{Claim}
\newtheorem  {question}        [theorem]{Question}

\newcommand{\Z}{\mathbb{Z}}

\newcommand{\sstar}{* \hspace{-.05cm} *}
\newcommand{\ddiamond}{\diamond \hspace{-.05cm} \diamond}

\newlength{\dhatheight}

\newcommand{\doublehat}[1]{%
    \settoheight{\dhatheight}{\ensuremath{\hat{#1}}}%
    \addtolength{\dhatheight}{-0.35ex}%
    \hat{\vphantom{\rule{1pt}{\dhatheight}}%
    \smash{\hat{#1}}}}

\begin{document}

\title{On the  threshold of spread-out \\ contact process percolation}

\author{Bal\'azs R\'ath\textsuperscript{1}, Daniel Valesin\textsuperscript{2}}
\footnotetext[1]{MTA-BME Stochastics Research Group, Budapest University of Technology and Economics, Egry J\'ozsef u. 1, 1111 Budapest, Hungary.  rathb@math.bme.hu}
\footnotetext[2]{\noindent University of Groningen, Nijenborgh 9, 9747 AG Groningen, The Netherlands. \\ d.rodrigues.valesin@rug.nl}
\date{\today}
\maketitle

\begin{abstract}
\noindent We study the stationary distribution of the (spread-out) $d$-dimensional contact process from the point of view of site percolation. In this process, vertices of~$\mathbb{Z}^d$ can be healthy (state 0) or infected (state 1). With rate one infected sites recover, and with rate~$\lambda$ they transmit the infection to some other vertex chosen uniformly within a ball of radius~$R$. The classical phase transition result for this process states that there is a critical value~$\lambda_c(R)$ such that the process has a non-trivial stationary distribution if and only if~$\lambda > \lambda_c(R)$. In configurations sampled from this stationary distribution, we study nearest-neighbor site percolation of the set of infected sites; the associated percolation threshold is denoted~$\lambda_p(R)$. We prove that~$\lambda_p(R)$ converges to~$1/(1-p_c)$ as~$R$ tends to infinity, where~$p_c$ is the threshold for Bernoulli site percolation on~$\mathbb{Z}^d$. As a consequence, we prove that~$\lambda_p(R) > \lambda_c(R)$ for large enough~$R$, answering an open question of \cite{LS06} in the spread-out case.
\end{abstract}


\newpage

\tableofcontents

\newpage

\section{Introduction}

\subsection{Nearest-neighbour contact process}\label{subsect_intro_near_neigh_cont_proc}
The (nearest-neighbor) \textit{contact process} on~$\mathbb{Z}^d$ with infection rate~$\lambda > 0$ is the continuous-time Markov process~$(\xi_t)_{t\geq 0}$ with state space~$\{0,1\}^{\mathbb{Z}^d}$ and infinitesimal pregenerator given by
\begin{equation}\label{eq:generator_nn}
(\mathcal{L}f)(\xi) = \sum_{\substack{x \in \mathbb{Z}^d:\\\xi(x) = 1}} \left( f(\xi^{0 \to x}) - f(\xi) + \frac{\lambda}{  2d}  \sum_{y
	\sim x} (f(\xi^{1 \to y}) - f(\xi))\right),
\end{equation}
where~$\xi \in \{0,1\}^{\mathbb{Z}^d}$,~$f: \{0,1\}^{\mathbb{Z}^d} \to \mathbb{R}$ is a local function,~$y \sim x$ indicates that the~$\ell^1$-norm of~$y-x$ is one,  and~$\xi^{i \to z}$ is the configuration obtained by changing~$\xi$ so that the state of vertex~$z$ is set to~$i$ and the states of other vertices are left unchanged. 

\smallskip

The common interpretation of the dynamics is that sites of~$\mathbb{Z}^d$ are individuals, which can be healthy (state~0) or infected (state~1). Infected individuals recover with rate one, and with rate~$\lambda$ they transmit the  infection; the target of the transmission is chosen uniformly among the~$2d$ neighbors. This process was introduced in~\cite{Har74} and is treated in the expository text~\cite{Li99}. Here we list some definitions and statements that will be important in order to explain our results.

\smallskip

We denote by~$\underline{0}$ and~$\underline{1}$ the identically-zero and identically-one element of~$\{0,1\}^{\mathbb{Z}^d}$, respectively. Note that~$\underline{0}$ is an absorbing state for the contact process. The \textit{survival probability} of the infection is
$$\rho(\lambda):= \mathbb{P}_\lambda( \, \xi^{\{0\}}_t \neq \underline{0} \text{ for all } t \, ),$$
where~$\mathbb{P}_\lambda$ is a probability measure under which the contact process on~$\mathbb{Z}^d$  with infection rate~$\lambda$  is defined and~$(\xi^{\{0\}}_t)_{t \geq 0}$ is the contact process started with a single infected site at the origin.
In our notation here and in what follows, we omit the dependence on the dimension~$d$. Having in mind the simple observation that~$\lambda \mapsto \rho(\lambda)$ is non-decreasing, one then defines the critical infection rate as
$$\lambda_c := \sup\{ \, \lambda > 0: \rho(\lambda) = 0 \, \}.$$
The contact process exhibits a phase transition, manifested by the fact that~$\lambda_c \in (0,\infty)$, see Corollary~4.4, page~308 in \cite{Li85}. It is also known that
\begin{equation}\label{eq:bez_grimm}
\rho(\lambda_c) = 0\quad \text{and}\quad \lim_{\lambda \searrow \lambda_c} \rho(\lambda) = 0.
\end{equation}
The first equality is the celebrated result by Bezuidenhout and Grimmett~\cite{BG90}. The second equality is a consequence of the first and the fact (whose proof preceded~\cite{BG90}) that~$\lambda \mapsto \rho(\lambda)$ is right continuous on~$[0,\infty)$ (see the proof of Theorem~1.6 in~\cite{Li85}; the assumption there is~$d=1$, but the proof is the same for any dimension).

\smallskip

Finally, the \textit{upper invariant measure} of the contact process, denoted~$\mu_\lambda$, is defined  as
$$\mu_\lambda:= \lim_{\substack{t\to \infty\\ \text{(d)}}} \xi^{\underline{1}}_t,$$
where~$(\xi^{\underline{1}}_t)_{t\geq 0}$ is the contact process started from the identically-one configuration, and the limit in distribution can be shown to exist. This distribution is invariant under the contact process dynamics, and also invariant and ergodic with respect to translations of~$\Z^d$. Moreover,
\begin{equation}\label{eq:density_nn}
\mu_\lambda(\, \xi(0)=1 \, ) = \rho(\lambda).
\end{equation}
 In particular,~$\mu_{\lambda}$ is the Dirac measure concentrated on the identically-zero configuration when~$\lambda \le \lambda_c$, and is a non-trivial measure supported on configurations with infinitely many 1's when~$\lambda > \lambda_c$.

\subsection{Spread-out contact process}\label{subsect_spread_out_cont_proc_def}
The main process of interest in this work is the \textit{spread-out contact process}, studied in~\cite{bds89}. Apart from the infection rate~$\lambda$, this process has as an additional parameter, the range~$R \in \mathbb{N}$; its infinitesimal pregenerator is
\begin{equation}\label{eq:generator}
(\mathcal{L}f)(\xi) = \sum_{\substack{x \in \Z^d:\\\xi(x) = 1}} \left( f(\xi^{0 \to x}) - f(\xi) + \frac{\lambda}{  |B(R)|} \sum_{\substack{y \in B(x,R)}} (f(\xi^{1 \to y}) - f(\xi))\right),
\end{equation}
where we repeat the notation of~\eqref{eq:generator_nn},~$B(x,R)$ is the~$\ell^\infty$-ball with center~$x$ and radius~$R$ in~$\Z^d$, and~$|B(R)|$ is the cardinality of the~$\ell^\infty$-ball with center~0 and radius~$R$ in~$\Z^d$, see~\eqref{ball_def} below. Hence, in this modified version of the contact process, infected vertices  again recover with rate one and transmit the infection with rate~$\lambda$; however, the target of the transmission is chosen uniformly at random in the translate of~$B(R)$ centered at the position of the vertex that transmits the infection.


\smallskip

Making the dependence on~$R$ explicit, we again define the survival probability
\begin{equation*}
\rho(\lambda,R):=\mathbb{P}_{\lambda,R}( \, \xi^{\{0\}}_t \neq \underline{0} \text{ for all }t \, ),
\end{equation*}
where~$\mathbb{P}_{\lambda,R}$ is a probability measure under which the spread-out contact process on~$\Z^d$ with infection rate~$\lambda$ and range~$R$ is defined. The critical infection rate is then
$$\lambda_c(R):= \sup\{ \, \lambda: \rho(\lambda, R) = 0 \, \},$$
we again have~$\lambda_c(R) \in (0,\infty)$, and we believe that
$$\rho(\lambda_c(R),R) = 0\quad \text{and}\quad \lim_{\lambda \searrow \lambda_c(R)}\rho(\lambda,R)=0$$
holds, but we will neither prove nor use this spread-out analogue of \eqref{eq:bez_grimm}.

\smallskip

We obtain the upper invariant measure, now denoted~$\mu_{\lambda,R}$, in the same way as before, and have
\begin{equation}\label{eq:density}
\mu_{\lambda,R}( \, \xi(0)=1   \, ) = \rho(\lambda,R),
\end{equation}
see Claim \ref{claim_graphical_upper_invariant} below.
 The main results of~\cite{bds89} imply that as~$\lambda$ is kept fixed and~$R$ is taken to infinity, the spread-out contact process started from a single infected site becomes in some respects similar to a continuous-time branching process in which individuals die with rate one and give birth with rate~$\lambda$.
This similarity is captured by the convergences (see Theorems~1 and~2 in~\cite{bds89}):
\begin{equation}
\label{eq:first_asymp}
\lim_{R \to \infty} \lambda_c(R) = 1
\end{equation}
and
\begin{equation}
\label{eq:sec_asymp}
\lim_{R \to \infty} \rho(\lambda,R) = \rho(\lambda,\infty) := 1 - \frac{1}{\lambda} \text{ for }\lambda > 1.
\end{equation}
 Indeed, for the aforementioned branching process, the critical value of the birth rate~$\lambda$ is one, and, in case~$\lambda > 1$, then the probability of survival of a population started from one individual is equal to~$\rho(\lambda,\infty)$. Let us  note that \cite[Theorem 1]{bds89} also identifies the rate of convergence in \eqref{eq:first_asymp} (which depends on $d$), but our current interest lies elsewhere.

\subsection{Percolation under the upper invariant measure}
In this paper, we will investigate \textit{site percolation} on~$\Z^d$, with~$d \ge 2$, where the underlying configurations  are sampled from the upper stationary distribution of the (spread-out) contact process.

\smallskip

Define $\text{Perc} \subset \{0,1\}^{\Z^d}$ as the set of configurations~$\xi \in \{0,1\}^{\Z^d}$ for which the subgraph of the nearest-neighbour lattice~$\Z^d$ induced by~$\{x: \xi(x) = 1\}$ has an infinite connected component.  See \eqref{perc_def_eq} below for a formal definition. Then let
$$
 \lambda_p := \sup\{ \, \lambda: \mu_\lambda(\text{Perc}) = 0 \, \},
 \quad
  \lambda_p(R):= \sup\{ \, \lambda: \mu_{\lambda,R}(\text{Perc}) = 0 \, \}.
 $$
It follows from the definition of~$\lambda_c$ and~\eqref{eq:density_nn} (respectively, from the definition of~$\lambda_c(R)$ and~\eqref{eq:density}) that~$\lambda_p \ge \lambda_c$ and~$\lambda_p(R) \ge \lambda_c(R)$.

It is known that
\begin{equation*}
\lambda_p < \infty \quad \text{and}\quad \lambda_p(R) < \infty \text{ for all }R.
\end{equation*}
The first inequality follows from the fact, given in~\cite[Theorem~2.1]{LS06}, that
\begin{equation}\label{liggett_steif_stoch_dom_result}
\text{ $\forall\, p \in (0,1)\; \; \exists\,  \lambda > 0 \; : \; \mu_\lambda$ stochastically dominates $\pi_p$,}
  \end{equation}
where $\pi_p$ denotes  the Bernoulli product measure $\pi_p$ on $\{0,1\}^{\mathbb{Z}^d}$ with parameter~$p$.
 Note that the stochastic domination result is given there for the one-dimensional contact process, but the result for~$\Z^d$ readily follows, since the contact process on~$\Z^d$ dominates independent one-dimensional contact processes on lines parallel to one of the coordinate axes.

 It also follows from~$\lambda_p < \infty$ that~$\lambda_p(R) < \infty$ for any~$R$. Indeed, by using the graphical construction of the process (which we review in Section~\ref{subsection_graphical_constr_of_cont_proc} below), it is possible to give a coupling~$(\xi_t,\xi'_t)_{t \ge 0}$, where~$(\xi_t)$ is a nearest-neighbor contact process (as defined from the generator~\eqref{eq:generator_nn}) with infection rate~$\lambda$ and~$(\xi'_t)$ is a spread-out contact process (as defined from the generator~\eqref{eq:generator}) with range~$R$ and infection rate~$\frac{|B(R)|}{2d}\lambda$, both processes started with the identically-one configuration, and so that~$\xi_t'(x) \ge \xi_t(x)$ for all~$t \ge 0$ and~$x \in \Z^d$. Since~$\mu_\lambda$ is the limiting distribution of~$\xi_t$ as~$t \to \infty$ and~$\mu_{\frac{|B(R)|}{2d}\lambda,R}$ is the limiting distribution of~$\xi'_t$ as~$t \to \infty$, we readily obtain that~$\mu_\lambda$ is stochastically dominated by~$\mu_{\frac{|B(R)|}{2d}\lambda,R}$. In particular, if for some~$\lambda > 0$ we have~$\mu_\lambda(\text{Perc}) > 0$, then~$\mu_{\frac{|B(R)|}{2d}\lambda,R}(\text{Perc}) > 0$ also holds.

\smallskip

We address the following question.
\begin{question}\textup{\cite[Section 6, Question 2]{LS06} \label{conj:LS}}
	For the nearest-neighbour contact process on~$\Z^d$,~$d \ge 2$, do we have~$\lambda_c  < \lambda_p$?
\end{question}
In light of~\eqref{eq:bez_grimm} and the fact that the set of infected sites under~$\mu_\lambda$ has density~$\rho(\lambda)$, it is natural to expect that we indeed have the strict inequality. For example,  \cite{rob_sharp} showed that in the two-dimensional case the cluster size distribution of infected sites under $\mu_\lambda$ has exponential decay if the value of the infection parameter $\lambda$  belongs to the (presumably non-empty) interval of parameters~$(\lambda_c,\lambda_p)$.

\smallskip

 Although we could not settle Question~\ref{conj:LS} (which, to the best of our knowledge, is currently still open), we have made progress on the corresponding statement for the spread-out contact process:

\begin{theorem}\label{thm_main} For the spread-out contact process on~$\Z^d$,~$d \ge 2$, we have
	\begin{equation}\label{limit_lambda_p_R}
	\lim_{R \to \infty} \lambda_{p}(R) = \frac{1}{1-p_c},
	\end{equation}
where~$p_c=p_c(d)$ is the critical parameter for Bernoulli site percolation on~$\Z^d$.
\end{theorem}
\begin{corollary}\label{corollary_l_c_l_p_different}
By~\eqref{eq:first_asymp} and~\eqref{limit_lambda_p_R}, we have~$\lambda_c(R) < \lambda_p(R)$ if~$R$ is large enough.
\end{corollary}
Note that if $d=2$ then the statement of Corollary \ref{corollary_l_c_l_p_different} also follows from known results
and an elementary observation, see Remark \ref{remark_stein_ideas} below.

\smallskip

The limiting value~$\lambda = \frac{1}{1-p_c}$ in~\eqref{limit_lambda_p_R} arises as the solution to the equation~$\rho(\lambda,\infty) = p_c$, cf.\ \eqref{eq:sec_asymp}. To make sense of this, we observe that as the range $R$ of the spread-out contact process is taken to infinity (with~$\lambda$ fixed), one can expect the measure~$\mu_{\lambda,R}$ to converge weakly to a Bernoulli product measure $\pi_{\rho(\lambda, \infty)}$ with density~$\rho(\lambda, \infty)$. This guess comes from~\eqref{eq:sec_asymp} and by analogy with the asymptotic behavior of the (nearest-neighbor) contact process as the dimension is taken to infinity; see the theorem in~\cite[page~388]{SV86}. Although our methods could also prove the weak convergence to Bernoulli product measure in our context, we abstain from giving a full proof for the sake of brevity, and only provide a sketch, see Remark~\ref{rem_sketch_measure} below.

\smallskip

Let us stress that weak convergence of $\mu_{\lambda,R}$ to $\pi_{\rho(\lambda, \infty)}$ does not automatically imply the convergence \eqref{limit_lambda_p_R} of the corresponding percolation thresholds.
It is believed that in many  circumstances the probability of the percolation event changes continuously with the underlying measure on configurations (and is thus ``decided locally'', since weak convergence of measures is a local phenomenon). This is related to the so-called Schramm locality conjecture; see for instance~\cite{bnp11, MT17,sxz14}. However, there are also known exceptions and pathological situations. For instance,~\cite[Theorem~19]{bgp12} states that for any~$p \in (0,1)$ and any~$K \in \mathbb{N}$, there exists a probability measure~$\mu$ on~$\{0,1\}^{\Z^d}$ satisfying
\begin{equation}\label{eq_bgpeq}
	\mu\big( (\xi(x_1),\ldots, \xi(x_K)) = (i_1,\ldots,i_K)\big) = p^{\sum_k i_k}\cdot (1-p)^{K-\sum_k i_k}
\end{equation}
for any~$(i_1,\ldots,i_K)\in \{0,1\}^K$ and any~$K$-tuple~$(x_1,\ldots,x_K)$ of distinct vertices of~$\Z^d$ and so that~$\mu(\text{Perc}) = 0$, and there exists another measure~$\nu$ on~$\{0,1\}^{\Z^d}$ also satisfying the property expressed in~\eqref{eq_bgpeq}, but so that~$\nu(\text{Perc}) = 1$. In words: both of the measures $\mu$ and $\nu$  ``locally'' look like  $\pi_p$, yet $\mu$ percolates while $\nu$ one doesn't.

\smallskip

In~\cite[Section 6, Question~1]{LS06}, it is asked if for any~$p \in (0,1)$ there exists~$\lambda > \lambda_c$ such that~$\mu_\lambda$ is stochastically dominated by the Bernoulli product measure $\pi_p$ with parameter~$p$ (and the same question can be asked for the spread-out process). Note that an affirmative answer would immediately imply that~$\lambda_p > \lambda_c$ (resp.\ $\lambda_p(R)> \lambda_c(R)$). As far as we know, this question remains open, and we do not provide an answer here. We do not know if taking~$R$ large makes this question any easier. However, we pose the following question about the spread-out model.
\begin{question}\label{question_sandwich}
 Given any $\lambda>1$ and $\varepsilon>0$, can one find $R_0$ such that
\begin{enumerate}[(i)]
\item\label{q_lb_sandwich}
 $\mu_{\lambda,R}$ stochastically dominates  $\pi_{\rho(\lambda,\infty)-\varepsilon}$ if $R \geq R_0$?
 \item\label{q_ub_sandwich} $\mu_{\lambda,R}$ is stochastically dominated by $\pi_{\rho(\lambda,\infty)+\varepsilon}$ if $R \geq R_0$?
\end{enumerate}
\end{question}
Note that an affirmative answer to Question \ref{question_sandwich}\eqref{q_lb_sandwich}\&\eqref{q_ub_sandwich}  would immediately imply our Theorem \ref{thm_main}.

\medskip

Let us now mention some related results (other then the already mentioned~\cite{LS06} and~\cite{rob_sharp}). In~\cite{vB20} it is proved that, if~$d \ge 2$ and~$\lambda > \lambda_p$, then almost surely under~$\mu_\lambda$, there is a unique infinite percolation cluster. (Note that the value defined by us as~$\lambda_p$ is denoted in~\cite{vB20} by~$\lambda_c$). This result is proved by means of a standard Burton-Keane approach \cite{bk89}, which would also work for the spread-out contact process. In \cite[Theorem 1.5]{rob_stein} it is proved that the stationary
 supercritical contact process on $\mathbb{Z}^d$ observed on certain space-time slabs stochastically dominates an i.i.d.\ Bernoulli product measure.
 Sharpness of the percolation phase transition for the stationary configuration of a variant of the two-dimensional contact process with three states is proved in \cite{markus_jakob_rob_sharp}. The percolation phase transition of the invariant distributions of the (nearest-neighbor and spread-out) voter model is studied in~\cite{RV15} and~\cite{RV17}. The latter reference has as main theorem an asymptotic result for the relevant percolation threshold, as the range of interaction of the model is taken to infinity; it is analogous to our Theorem~\ref{thm_main} (with a much shorter proof).\medskip

The following argument was pointed out to us by Stein Andreas Bethuelsen.
\begin{remark}\label{remark_stein_ideas}
Let us note that it is easy to prove that
\begin{equation}\label{stein_dom}
\pi_{\frac{\lambda}{1+\lambda}} \text{ stochastically dominates } \mu_{\lambda,R}
\end{equation}
 (see below Claim \ref{claim_graphical_upper_invariant} for the proof).
From \eqref{stein_dom} one obtains that $\frac{\lambda}{1+\lambda} < p_c$ implies $\lambda \leq \lambda_p(R)$ for any $\lambda >0$, from which $\frac{p_c}{1-p_c} \leq \lambda_p(R)$ follows.
Thus if $d=2$ then the statement of Corollary \ref{corollary_l_c_l_p_different} also follows using \eqref{eq:first_asymp} and the fact that $p_c(2)>\frac{1}{2}$, cf.\ \cite{H82}.
However, this elementary argument breaks down for any $d\geq 3$, since $p_c(d) \leq p_c(3) < \frac{1}{2}$, cf.\ \cite{CR85}.

\smallskip

One may also consider the percolation threshold $\lambda^0_p(R)$ of the set of healthy sites: given $\xi \in \{0,1\}^{\Z^d}$, let us define $\overline{\xi} \in \{0,1\}^{\Z^d}$
by letting $\overline{\xi}(x)=1-\xi(x)$. Let $\lambda^0_p(R):= \inf\{ \, \lambda: \mu_{\lambda,R}( \overline{\xi} \in \text{Perc}) = 0 \, \}$. Now $\frac{1-p_c}{p_c} \leq \lambda^0_p(R)$ follows from \eqref{stein_dom}, thus if $d \geq 3$ then $\lambda_c(R) < \lambda^0_p(R)$ holds using \eqref{eq:first_asymp} and $p_c(d) < \frac{1}{2}$.
Note that
$\lambda^0_p(R)<+\infty$ follows from \eqref{liggett_steif_stoch_dom_result} (which also holds for $\mu_{\lambda,R}$). Also note that the method of proof of Theorem \ref{thm_main}
could be easily modified to give $\lim_{R \to \infty} \lambda^0_p(R)=1/p_c$.

\smallskip

If we analogously define $\lambda^0_p$ to be the percolation threshold of the set of healthy sites under  the upper invariant measure $\mu_\lambda$ of the nearest-neighbour contact process (cf.\ \eqref{eq:generator_nn}), then
$ \lambda_c < \lambda^0_p < +\infty  $ holds if $d$ is large enough using \eqref{stein_dom} (which also holds in the nearest-neighbour case), $ \lim_{d \to \infty} \lambda_c=1$
(cf.\ \cite[(7)]{SV86}), $p_c(d) < \frac{1}{2}$ and \eqref{liggett_steif_stoch_dom_result}.
 \end{remark}

Our overall method of proof of Theorem \ref{thm_main} is very similar to the ones in~\cite{RV15} and~\cite{RV17}. However, while in~\cite{RV15} we could establish a non-trivial percolation phase transition for the nearest-neighbor voter model in dimension~5 and higher, here we were not able to settle Question~\ref{conj:LS} (concerning the nearest-neighbor contact process) even in high dimensions. The technical difficulty that explains this disparity is that the renormalization scheme we employ involves a competition between a combinatorial complexity factor (pertaining to the number of so-called proper embeddings of binary trees into~$\Z^d$), on the one hand, and the rate of site correlations under the invariant measure, on the other hand. In the voter model, the correlation structure is governed by coalescing random walks, for which good estimates are available, and become better for our purposes as the dimension increases. In the contact process (especially when~$\lambda$ is only slightly larger than~$\lambda_c$), one would expect the correlations to be much more complicated to analyze.

\begin{remark}
	Although up to this point we have distinguished the nearest-neighbour contact process from the spread-out contact process, from now on we only deal with the latter, since our results only refer to it. We also frequently drop the qualification `spread-out', and write simply `contact process'.
	\end{remark}

\subsection{Ideas and structure of proof}
Our proof of Theorem~\ref{thm_main} deals separately with the two statements
\begin{align}
\label{eq:want_main1}
	\lambda < \frac{1}{1-p_c} \quad &\Longrightarrow \quad \mu_{\lambda, R}(\mathrm{Perc}) = 0 \; \text{ for $R$ large enough,}\\
\label{eq:want_main2}
	\lambda >  \frac{1}{1-p_c} \quad &\Longrightarrow \quad \mu_{\lambda, R}(\mathrm{Perc}) > 0 \; \text{ for $R$ large enough.}
\end{align}
In our proofs of both statements, we use a renormalization scheme introduced in~\cite{Ra15} and employed in our works concerning percolation under the invariant distribution of the voter model,~\cite{RV15} and~\cite{RV17}. In very rough terms, the scheme depends on obtaining an upper bound for the probability of the event that inside many disjoint translates of a large box of~$\Z^d$, a configuration~$\xi \sim \mu_{\lambda,R}$ misbehaves in some prescribed way. In order to yield useful bounds, the renormalization scheme needs to be combined with some argument to guarantee that the restrictions of~$\xi \sim \mu_{\lambda,R}$ to boxes that are distant from each other are not too strongly correlated.

\smallskip

This kind of decorrelation result is easier to achieve if instead of considering~$\mu_{\lambda,R}$, we could consider (for some~$t > 0$) the measure~$\mu_{\lambda,R,t}$, defined as the distribution at time~$t$ of the contact process with parameters~$\lambda,R$ and started from all vertices infected at time zero (recall that~$\mu_{\lambda,R}$ is the weak limit of~$\mu_{\lambda,R,t}$ as~$t\to \infty$). To argue that this truncation to a finite-time horizon exhibits the desired spatial independence (as~$R$ gets larger, for~$t$ fixed), we rely on a coupling between the contact process and the related model of \textit{branching random walks}. While a coupling between these two processes is already given in~\cite[page 34]{Li99}, here we need a more careful construction in order to guarantee that a certain comparison property is satisfied -- see Section \ref{section_graphical_construction}.

\smallskip

Replacing~$\mu_{\lambda,R}$ by~$\mu_{\lambda,R,t}$ in the proof of~\eqref{eq:want_main1} is justified by the well-known fact that~$\mu_{\lambda,R,t}$ stochastically dominates~$\mu_{\lambda,R}$, see the last paragraph in~\cite[page 34]{Li99}. With this at hand, in order to prove~\eqref{eq:want_main1} (which already implies Corollary \ref{corollary_l_c_l_p_different}) it is sufficient to prove
$$\lambda < \frac{1}{1-p_c} \quad \Longrightarrow \quad \mu_{\lambda, R,t}(\mathrm{Perc}) = 0 \; \text{ for some $t$ and $R$ large enough.}$$

However, for~\eqref{eq:want_main2} this stochastic domination does not help, as it goes in the opposite direction as the one desired. Our treatment of~\eqref{eq:want_main2} involves a result allowing us to (stochastically) bound a coarse-graining of the upper invariant configuration from below. As already mentioned,~\cite[Theorem~2.1]{LS06} shows that~$\mu_{\lambda,R}$ stochastically dominates a Bernoulli product measure, but that result is only applicable when the infection rate is sufficiently large, so it does not fit our purposes. For this reason, we establish a result in the same spirit of the stochastic domination of Liggett and Steif, valid for all~$\lambda > 1$ and~$R$ large enough (depending on~$\lambda$), except that the domination is valid for a coarse grained version of~$\mu_{\lambda,R}$ instead of the measure itself.

Let us briefly describe this domination result (see Theorem~\ref{thm:density_good_boxes} below and the definitions that precede it). Under the assumption that~$\lambda > 1$, we prove that there are constants~$\kappa \in \mathbb{N}$ and~$\alpha \in (0,1)$ (depending on~$d$ and~$\lambda$, but not on~$R$) such that the following holds. We divide space into boxes of the form~$\kappa R \cdot z + [0,\kappa R)^d$, for~$z \in \Z^d$, and declare that each such box is \textit{good} for a configuration~$\xi \in \{0,1\}^{\Z^d}$ if it is ``sufficiently infected'', meaning that many sub-boxes of side length~$R$ have more than~$\alpha R^d$ infected vertices; see Definition~\ref{def:good_boxes} below for the precise condition. We then show that if~$\xi \sim \mu_{\lambda,R}$, then the distribution of the indicator of the set~$\{z: \text{the box~$\kappa R\cdot z + [0,\kappa R)^d$ is good} \}$ stochastically dominates a Bernoulli product measure. As~$R \to \infty$, the density of the product measure converges to one, and this convergence is exponentially fast in~$R^d$. We believe this result to be of independent interest. A similar coarse-graining technique has been employed in~\cite{bds89}; see the explanation in the second paragraph in page~448 of that paper. However, the emphasis of~\cite{bds89} was on proving that the $R$-spread out contact process survives when $\lambda=1+\varepsilon(R)$ (if $\varepsilon(R)$ goes to zero slowly enough as $R \to \infty$) using coarse-graining, while our emphasis is on proving that the $R$-spread-out contact process with a fixed $\lambda>1$ and $R \gg 1$ looks massively supercritical if we view it through the lens of coarse-graining.

\medskip

The rest of the paper is organized as follows. Section~\ref{s:notation} introduces notation that is used throughout the paper as well as the graphical construction of the contact process. In Section~\ref{s:domination}, we establish Theorem~\ref{thm:density_good_boxes}, which, as already mentioned, shows that a coarse graining of the upper invariant measure of the supercritical contact process stochastically dominates a Bernoulli product measure. Part of the proof of this result involves adapting to the context of oriented percolation a result of~\cite{LS06} pertaining to the contact process; this is done in the Appendix.
 Section~\ref{section_graphical_construction} develops our coupling between the contact process and branching random walks, and also contains some estimates concerning the collision probabilities of these branching random walks. Section~\ref{section_renorm_intro} reviews the renormalization technique from~\cite{Ra15}. Finally, Section~\ref{section_proof_of_main} contains the proof of Theorem~\ref{thm_main}.

\section{Basic notation}
\label{s:notation}

We denote $\mathbb{N}=\{1,2,\dots,\}$ and $\mathbb{N}_0=\{0,1,2,\dots\}$. For any set $A$, the cardinality of $A$ is denoted by $|A|$.

\subsection{Geometry of the lattice}

For a vector $x=(x_1,\dots,x_d) \in \mathbb{Z}^d$, the $\ell^\infty$-norm of $x$ is defined by $|x|=\max_{1 \leq i \leq d} |x_i| $ and the $\ell^1$-norm of $x$ is defined by $|x|_1=\sum_{i=1}^d |x_i|$. Two vertices $x,y$ are \textit{ nearest neighbors} if $|x-y|_1 = 1$; we denote this by $x \sim y$. Vertices $x$ and $y$ are $*$-neighbors if $|x-y| = 1$.

\begin{definition}\label{def_paths}
A \textit{nearest-neighbor} path in $\Z^d$ is a finite or infinite sequence $\gamma = (\gamma(0),\gamma(1),\ldots)$ such that $\gamma(i)\sim \gamma(i+1)$ for each $i$. A \textit{$*$-connected path} is a sequence $\gamma = (\gamma(0), \gamma(1),\ldots)$ such that $\gamma(i)$ and $\gamma(i+1)$ are $*$-neighbors for each $i$.
\end{definition}
Observe that any nearest-neighbor path is also a $*$-connected path.

\begin{definition}\label{def_connected_in_a_config}
Given disjoint sets $A, B \subset \Z^d$ and a configuration $\xi \in \{0,1\}^{\Z^d}$, we say that $A$ and $B$ are connected by an \emph{open} path in $\xi$ (and write $A \stackrel{\xi}{\longleftrightarrow} B$) if there exists a nearest-neighbor path $\gamma = (\gamma(0), \ldots, \gamma(n))$ such that $\gamma(0) \in A$, $\gamma(n) \in B$ and $\xi(\gamma(i)) = 1$ for all $i$. Similarly, we write $A\stackrel{*\xi}{\longleftrightarrow} B$ if there exists a $*$-connected path from a vertex in $A$ to a vertex in $B$ and $\xi$ is equal to 1 at all points in this path.
\end{definition}

\begin{definition}\label{def_connected_to_infty_in_a_config}
Given a set $A \subset \Z^d$ and a configuration $\xi \in \{0,1\}^{\Z^d}$, we say that $A$ is connected to infinity by an \emph{open} path in $\xi$ (and write $A \stackrel{\xi}{\longleftrightarrow} \infty$) if there exists an infinite nearest-neighbor simple path from a vertex in $A$ and $\xi$ is equal to 1 at all points in this path.
\end{definition}

The balls and spheres with respect to the $\ell^\infty$-norm are given by
\begin{align}
\label{ball_def}
&B(L) = \{x \in \Z^d: |x| \leq L\}, &B(x,L) = \{y \in \Z^d: |x-y| \leq L\},\\
\label{sphere_def}
&S(L) = \{x \in \Z^d: |x|= L\}, &S(x,L) = \{y \in \Z^d: |x-y|= L\}.
\end{align}

\subsection{Set of configurations }
\label{subsetion_set_of_config}

The indicator of an event  $B$ is denoted by $\mathds{1}[B]$.

\smallskip


We endow $\{0,1\}^{\Z^d}$ with the $\sigma$-algebra $\mathcal{F}$ generated by all the cylinder sets.

\smallskip

Let us denote by $\underline{1}$ the element of $\{0,1\}^{\Z^d}$ for which $\underline{1}(x)=1$ for all $x \in \mathbb{Z}^d$.

\smallskip

We adopt the convention of associating a configuration~$\xi \in \{0,1\}^\Z$ with the set~$\{x: \xi(x)=1\}$. For example, if
$H \subseteq \mathbb{Z}^d$ then $|\xi \cap H |:=\sum_{x \in H} \xi(x)$.

We endow~$\{0,1\}^{\Z^d}$ with the partial order under which~$\xi \leq \xi'$ if~$\xi(x) \leq \xi'(x)$ for all~$x \in \mathbb{Z}^d$. An event~$A \in \mathcal{F}$ is called increasing if, for any~$\xi \in A$ and~$\xi' \geq \xi$, we have~$\xi' \in A$. If $\mu$ and $\mu'$ are both probability measures on $(\{0,1\}^{\Z^d},\mathcal{F})$ then we
say that $\mu'$ stochastically dominates $\mu$ if for any increasing event $A$ we have $\mu'(A)\geq \mu(A)$.

\smallskip

Let us define the $\mathcal{F}$-measurable percolation event $\mathrm{Perc}$ by
\begin{equation}\label{perc_def_eq}
 \{\, \xi \in \mathrm{Perc} \, \} := \cup_{x \in \mathbb{Z}^d } \{ x \, \stackrel{\xi}{\longleftrightarrow} \, \infty \}.
\end{equation}
Note that $\mathrm{Perc}$ is an increasing event. In the terminology of percolation theory,~$\{x: \xi(x)=1\}$ is the set of $\xi$-\emph{open} sites, while~$\{x: \xi(x)=0\}$ is the set of $\xi$-\emph{closed} sites.



\smallskip

For $p \in [0,1]$, we denote by $\pi_p$ the Bernoulli($p$) product measure on $(\{0,1\}^{\Z^d},\mathcal{F})$.

\subsection{Graphical construction of the contact process}
\label{subsection_graphical_constr_of_cont_proc}

We recall the graphical construction of the contact process from \cite[Part I., Section 1]{Li99}. 

\smallskip

For each $x \in \mathbb{Z}^d$, let $N_x$
denote a Poisson process of rate $1$ on $[0,\infty)$,  and for each $x,y \in \mathbb{Z}^d$ satisfying $y \in B(x,R)$ let $N_{x,y}$
denote a Poisson process of rate $\lambda/|B(R)|$ on $[0,\infty)$.
 All Poisson processes are independent.

\smallskip

  We decorate the space-time picture $\mathbb{Z}^d \times [0,+\infty)$ by placing a recovery symbol~$\times$ at $(x,t)$ if
 $t$ is an arrival time of $N_x$ and placing an infection arrow pointing from $(x,t)$ to $(y,t)$ if $t$ is an arrival time of $N_{x,y}$.
 An \emph{infection path} in  $\mathbb{Z}^d \times [0,+\infty)$ is a connected oriented path which moves along
the time lines in the increasing $t$ direction without passing through a recovery
symbol, and along infection arrows in the direction of the arrow.

 If $x,y \in \mathbb{Z}^d$ and $0\leq s \leq t$ then we denote by $(x,s)\rightsquigarrow (y,t)$
the event that there is an infection path connecting $(x,s)$ to $(y,t)$.

\begin{definition}[Infection path indicators]\label{Xi_graph_def}
  Let us define the  $\{0,1 \}$-valued random variables
\begin{equation}\label{Xi_graphical_def_eq}
 \Xi_t(x,y) := \mathds{1}\left[(x,0)\rightsquigarrow (y,t) \right],\qquad x, y \in \Z^d,\; t \ge 0.
 \end{equation}
 \end{definition}
 For any  $\xi_0 \in \{0,1\}^{\mathbb{Z}^d}$, the contact process $(\xi_t)_{t \geq 0}$ with infection rate $\lambda$, range $R$ and initial  state $\xi_0$ can be constructed by letting
\begin{equation}
\xi_t(y)= \max_{x \in \mathbb{Z}^d} \left\{ \xi_0(x) \cdot \Xi_t(x,y) \right\} , \qquad y \in \mathbb{Z}^d, \quad t \geq 0.
\end{equation}

We will also need the following claim.

\begin{claim}[Graphical construction of $\mu_{\lambda,R}$] \label{claim_graphical_upper_invariant}
 If we define
\begin{equation}\label{upper_invariant_active_path}
  \xi(x):=  \lim_{t \to \infty} \max_{y \in \mathbb{Z}^d} \Xi_t(x,y) , \qquad x \in \mathbb{Z}^d,
\end{equation}
then $(\xi(x))_{x \in \mathbb{Z}^d}$ has the law $\mu_{\lambda,R}$ of the upper invariant measure of the contact process with infection rate $\lambda$ and range $R$.
\end{claim}
The proof of this claim follows from the self-duality of the contact process (see \cite[Part I., Section 1.]{Li99})
and the fact that $\xi(x)$ in \eqref{upper_invariant_active_path}  is the indicator of the event that there is an infection path from $(x,0)$ to infinity.

\begin{proof}[Proof of \eqref{stein_dom}]
If we define $\xi^\circ(x)$ to be the indicator of the event that the first infection arrow pointing out of $x$ arrives earlier than the first recovery mark at $x$ then
$\xi^\circ \sim \pi_{\frac{\lambda}{1+\lambda}}$ and $\xi \leq \xi^{\circ}$, from which
 \eqref{stein_dom} readily follows.
\end{proof}

\section{Coarse-grained upper invariant measure dominates Bernoulli}\label{section_domination_liggett_steif}
\label{s:domination}

The goal of this section is to establish a stochastic domination result pertaining to the upper invariant distribution~${\mu}_{\lambda, R}$ of the contact process when~$\lambda > 1$ and~$R$ is large enough.
The main result we will obtain is Theorem~\ref{thm:density_good_boxes} below; we will need many preliminary results in order to prove it. The following diagram depicts all of them, and an arrow means that a result is used in the proof of the result to which it points.

\begin{tikzpicture}
	\node at (-14.7,1.3) {\small Theorem \ref{thm:lss} };
	\node at (-13.7,0.75) {\small (Liggett-Schonmann-Stacey) };
	\node at (-22.5,-1.25) {\small Theorem 3.9};
	\node at (-22,-1.75) {\small (proved in Appendix)};
	\node at (-18.6,0) {\small Theorem \ref{thm:density_good_boxes} $\xleftarrow{\makebox[.4cm]{}}$ Proposition \ref{prop:second_domination}$\xleftarrow{\makebox[.4cm]{}}$ Lemma \ref{lem:domination_time_one} $\xleftarrow{\makebox[.4cm]{}}$ Proposition \ref{prop:growth_to_adjacent_box}};
	\node at (-15,-1.5) {\small Lemma \ref{lem:growth_to_neighbor} $\xrightarrow{\makebox[.4cm]{}}$ Lemma \ref{lem:infect_small_boxes} \hspace{.3cm} Lemma \ref{lem:growth_in_box}};
	\node at (-18,-3) {\small Lemma \ref{lem:calculus}$ \xrightarrow{\makebox[.2cm]{}}$ Lemma \ref{lem:first_bound_density}$\xrightarrow{\makebox[.2cm]{}}$ Lemma \ref{lem:low_density} $\xrightarrow{\makebox[.2cm]{}}$ Lemma \ref{lem:for_supermartingale}$\xrightarrow{\makebox[.2cm]{}}$   Proposition \ref{prop:supermartingales}};
	\draw [->] (-15,-1.2) -- (-14,-0.3);
	\draw [->] (-12.2,-2.7) -- (-12.2,-1.8);
	\draw [->] (-12.2,-1.2) -- (-13.2,-0.3);
	\draw [->] (-23,-1.05) -- (-23,-.35);
	\draw [->] (-16,1.2) -- (-17,0.3);
\end{tikzpicture}\\

 We will consider large boxes of the form~$ [0,\kappa R)^d + \kappa R z$, where~$z \in \Z^d$ and~$\kappa$ is a large constant. In a configuration~$\xi$ sampled from~${\mu}_{\lambda,R}$, we will want to argue that inside most boxes of this form, the infected set~$\{x:\xi(x) = 1\}$ satisfies a certain high density condition, see Definition~\ref{def:good_boxes}.  This is achieved in Theorem~\ref{thm:density_good_boxes} below, which states that the set of such good boxes stochastically dominates a Bernoulli product measure with very high density.
The results of this section  will only be used for the proof of $\limsup_{R \to \infty } \lambda_p(R) \leq \frac{1}{1-p_c}$ in Section \ref{subsection_upper_bound}.

\smallskip

Recalling the notation used in \eqref{eq:generator}, we define the contact process with range $R$ and infection rate $\lambda$ on a finite subset $B$ of $\mathbb{Z}^d$
as the Markov process with state space $\{0,1\}^B$ and infinitesimal generator
\begin{equation}\label{eq:generator_on_B}
 (\mathcal{L}f)(\xi) = \sum_{\substack{x \in B:\\\xi(x) = 1}} \left( f(\xi^{0 \to x}) - f(\xi) + \frac{\lambda}{  |B(R)|} \cdot \sum_{\substack{y \in B(x,R)\cap B  }} (f(\xi^{1 \to y}) - f(\xi))\right).
\end{equation}

\begin{proposition}[Infection spreads to adjacent box]
\label{prop:growth_to_adjacent_box}
Let us fix~$d \geq 1$ and~$\lambda > 1$. There exist constants
$$\kappa \in \mathbb{N},\quad 0<\alpha<1,\quad  R_{*} > 0, \quad 0<T <\infty,\quad 0< c <\infty$$
(that only depend on $d$ and $\lambda$) such that the following holds  for any $R \geq R_*$.
Let
\begin{equation}
\label{eq:boxes12}
B = [0,\kappa R)^d,\quad B'= B+ \kappa R \cdot \vec{e}_1, \quad \widehat{B} = B \cup B',
\end{equation}
where~$\vec{e}_1=(1,0,\dots,0)$ is the first canonical vector of~$\Z^d$.

Let~$({\xi}_t)_{t \geq 0}$ be the contact process with infection rate $\lambda$ and range~$R$ on~$\widehat{B}$. If
\begin{equation}
\label{eq:initial_conf_good}
|\xi_0 \cap ([0,R)^d + R \cdot z)| \geq \alpha R^d \text{ for all $z \in \mathbb{Z}^d$ such that } [0,R)^d +R \cdot z \subset B,
\end{equation}
then
$$
\mathbb{P}\left(\begin{array}{c}
|\xi_{T} \cap ([0,R)^d + R \cdot z) | \geq \alpha R^d\text{ for all }
\\
 z \in \mathbb{Z}^d \text { such that } [0,R)^d +R \cdot z \subset \widehat{B}
\end{array}\right) > 1-\exp\left( -c R^d \right).
$$
\end{proposition}
The above proposition suggests the following definition.
\begin{definition}[Good box]\label{def:good_boxes}
Given~$d \geq 1$ and~$\lambda > 1$, let~$\kappa$ and~$\alpha$ be as in Proposition~\ref{prop:growth_to_adjacent_box}, and fix~$R \in \mathbb{N}$.
\begin{enumerate}
\item Given~$z \in \Z^d$, we say that the box~$B= [0,\kappa R)^d  + \kappa R \cdot z$ is \emph{good} for configuration~$\xi \in \{0,1\}^{\Z^d}$ (or~$\xi \in \{0,1\}^B$) if
$$
|\xi \cap ([0,R)^d+ R  \cdot z' )| \geq \alpha R^d \text{ for all } z' \in \mathbb{Z}^d \text{ such that } [0,R)^d +R \cdot z' \subset B.
$$
\item Given~$\xi \in \{0,1\}^{\Z^d}$, define~$\omega_\xi \in \{0,1\}^{\Z^d}$ by
\begin{equation} \label{eq:def_of_omega_xi}
\omega_\xi(z) = \mathds{1}\left[\text{ the box }[0,\kappa R)^d + \kappa R\cdot  z \text{ is good for } \xi \,  \right],\quad z \in \mathbb{Z}^d.
\end{equation}
\end{enumerate}
\end{definition}
\begin{theorem}[Good boxes under $\mu_{\lambda,R}$ dominate product Bernoulli]\label{thm:density_good_boxes}
Let us fix~$d \geq 1$ and~$\lambda > 1$. There exist~$R_{\sstar} >0$ and~$\gamma > 0$ (that only depend on $d$ and $\lambda$) such that the following holds for any~$R \ge R_{\sstar}$. Let~$\xi$ be a random configuration sampled from~${\mu}_{\lambda,R}$, and let~$\omega_\xi$ be the corresponding configuration of good boxes, as in~\eqref{eq:def_of_omega_xi}. Then the law of~$\omega_\xi$ stochastically dominates a Bernoulli product  measure with density~$1 - \exp(-\gamma R^d)$.
\end{theorem}

\begin{remark} It follows from \eqref{stein_dom} that there exists $\gamma'>0$ that only depends on $d$ and $\lambda$ such that
 $\mu_{\lambda,R}( |\xi \cap [0,\kappa R)^d |=0 ) \geq \exp(-\gamma' R^d)$ for all $R \in \mathbb{N}$, which certainly implies
	$\mu_{\lambda,R}( \omega_\xi(0)=1 ) \leq 1- \exp(-\gamma' R^d)$ (for any~$\alpha \in (0,1)$),
  thus Theorem \ref{thm:density_good_boxes} is sharp in this sense. However,  the values of  $\alpha$ and $\gamma$ that we produce in our proofs
   are far from being optimal. Note that an affirmative answer to Question \ref{question_sandwich}\eqref{q_lb_sandwich} would easily imply Theorem \ref{thm:density_good_boxes}.
       Let us also note that in our proof of
       $\limsup_{R \to \infty } \lambda_p(R) \leq \frac{1}{1-p_c}$ in Section \ref{subsection_upper_bound}, we will only use that~$\omega_\xi$ stochastically dominates a Bernoulli product
    measure with density~$1 - \exp(-R^{\delta})$, where $\delta >\frac{\ln(2)}{\ln(6)}$.
\end{remark}
The rest of Section \ref{section_domination_liggett_steif} is devoted to the proof of the above stated results.
We encourage the reader to skip to Section \ref{section_graphical_construction} at first reading.

The rest of Section \ref{section_domination_liggett_steif} is organized as follows: we first deduce Theorem \ref{thm:density_good_boxes} from Proposition \ref{prop:growth_to_adjacent_box} in Section \ref{ss:domination}, then we prove Proposition \ref{prop:growth_to_adjacent_box} in Section \ref{ss:propagation}.

\subsection{Good boxes under $\mu_{\lambda,R}$ dominate product Bernoulli}
\label{ss:domination}

In this section, we prove Theorem~\ref{thm:density_good_boxes} using Proposition~\ref{prop:growth_to_adjacent_box}. We fix~$d \geq 1$ and~$\lambda > 1$ throughout.

We start by defining an auxiliary discrete-time Markov process~$(\eta_n)_{n \geq 0}$ taking values in~$\{0,1\}^{\Z^d}$.

\begin{definition}[Oriented percolation]\label{def_cellular_automaton}
Given~$p \in [0,1]$, let~$P_p$ be a probability measure under which we have defined random variables~$Z(z,n)$, for~$z \in \mathbb{Z}^d$ and~$n \in \mathbb{N}_0$, all independent and Bernoulli($p$).
 Define~$(\eta_n)_{n \geq 0}$ on this probability space by taking~$\eta_0 \in \{0,1\}^{\Z^d}$ arbitrarily and letting
\begin{equation} \label{eq:def_of_eta}\eta_{n+1}(z) = Z(z,n+1)\cdot \max\left(\eta_n(z),\; \eta_n(z + \vec{e}_1)\right),\quad z \in \mathbb{Z}^d,\; n \in \mathbb{N}_0,\end{equation}
where~$\vec{e}_1$ denotes the first canonical vector of~$\Z^d$.
\end{definition}

  Note that~$(\eta_n)$ is a ``probabilistic cellular automaton'', or ``oriented percolation process", in which any space-time point~$(z,n+1)$ is in state one with probability~$p$ in case at least one of~$(z,n)$ and~$(z+\vec{e}_1,n)$ is in state one; otherwise~$(z,n+1)$ is in state zero. We note the following for future use:
\begin{claim} \label{cl:oriented_percolation}
The evolution of~$(\eta_n)$ in each line of the form~$\{{z + k\vec{e}_1}: k \in \mathbb{Z}\}$, with~$z\in \{0\} \times \mathbb{Z}^{d-1}$, is independent of the evolution in all other such lines.
\end{claim}

\begin{definition}\label{def_cellular_upper_inv}
We denote by~${\nu}_p$ the upper invariant distribution of the process $(\eta_n)_{n \geq 0}$, that is,~${\nu}_p$ is the weak limit, as~$n \to \infty$, of the law of $\eta_n$ started from~$\eta_0 = \underline{1}$.
\end{definition}

We will establish Theorem~\ref{thm:density_good_boxes} as a consequence of the following two results.
\begin{proposition}[$\omega_\xi$ under $\mu_{\lambda,R}$ dominates $\nu_{\theta_1(R)}$]
\label{prop:second_domination} Let $\kappa, \alpha,  R_{*},  c$
  be as in Proposition~\ref{prop:growth_to_adjacent_box}, and assume~$R \ge R_*$. Let~$\xi$ be a random configuration sampled from~$\mu_{\lambda,R}$, and let~$\omega_\xi$ be the corresponding configuration of good boxes, as in~\eqref{eq:def_of_omega_xi}. Then the law of $\omega_\xi$ stochastically dominates~$\nu_{\theta_1(R)}$, where
\begin{equation}\label{theta_1_R_def}
\theta_1(R) := \left(1-\exp\left\{-\frac{c}{3}R^d\right\}\right)^2.
\end{equation}
\end{proposition}

Our next result is the discrete-time analogue of \cite[Theorem 2.1]{LS06} (which pertains to the upper invariant distribution of the contact process).

\begin{theorem}\label{thm:apply_ls}
If~$p \ge \frac34$, then~$\nu_p$ dominates a   Bernoulli product measure $\pi_{\theta_2(p)}$ with parameter
\begin{equation}\label{theta_2_R_def}
\theta_2(p) := 1 - \left(\frac{1-p + \sqrt{1-p}}{p} \right)^2.
 \end{equation}
\end{theorem}
Before we prove Proposition \ref{prop:second_domination} and Theorem \ref{thm:apply_ls}, let us deduce Theorem~\ref{thm:density_good_boxes} from them.
\begin{proof}[Proof of Theorem~\ref{thm:density_good_boxes}]
The theorem follows readily from Proposition \ref{prop:second_domination} and Theorem \ref{thm:apply_ls} by choosing~$\gamma > 0$ small enough and~$R_{**} \geq R_*$ large enough so that~$\theta_1(R) > \frac34$ and~$\theta_2(\theta_1(R)) > 1 - \exp\{-\gamma R^d\}$ for all~$R \ge R_{**}$.
\end{proof}

 As for the proof of Theorem \ref{thm:apply_ls}, the authors of \cite{LS06}  observe (see the remark following Theorem~2.1 there) that their proof can be adapted to discrete-time versions of the contact process, such as the oriented percolation process~$(\eta_n)$ under consideration here. We go over the main steps of the proof of Theorem~\ref{thm:apply_ls} in the Appendix, both for completeness and because we wish to be clear about how the exact value~$\theta_2(p)$  arises.

\smallskip

We now turn to Proposition~\ref{prop:second_domination}. Before we give its proof, we recall the well-known Liggett-Schonmann-Stacey \cite{LSS97} stochastic domination result:
\begin{theorem}[\cite{Li99}, Theorem B26]\label{thm:lss}
Let~$\{X_x:x\in\mathbb{Z}^d\}$ be $\{0,1\}$-valued random variables (jointly defined under some probability measure~$P$) satisfying, for some~$k \in \mathbb{N}$ and~$p \geq \tfrac14$,
$$P( \, X_x = 1 \mid X_y:\;|y-x| > k \, ) \geq 1-(1-\sqrt{p})^{|B_k(0)|} \quad \text{a.s. for all }x \in \Z^d.$$
Then the law of this family stochastically dominates the Bernoulli product measure $\pi_p$ with density~$p$ on $\mathbb{Z}^d$.
\end{theorem}
We apply this result to obtain the following.
\begin{lemma}[One-step domination] \label{lem:domination_time_one}
Let~$\kappa, \alpha,  R_{*},  c$ and~$T$ be as in Proposition~\ref{prop:growth_to_adjacent_box}. Assume~$R \ge R_*$ and~$\xi_0 \in \{0,1\}^{\Z^d}$. Let~$(\xi_t)_{t \geq 0}$ be the contact process with infection rate $\lambda$ and range~$R$ started from~$\xi_0$, and let~$(\eta_n)_{n \ge 0}$ be the process of Definition \ref{def_cellular_automaton} with~$p = \theta_1(R)$ (cf.\ \eqref{theta_1_R_def}) and started from~$\eta_0 = \omega_{\xi_0}$. Then the law of $\omega_{\xi_T}$ stochastically dominates the law of $\eta_1$.
\end{lemma}
\begin{proof}
For each~$z \in \Z^d$, define the sets
$$\mathscr{B}(z) = [0,\kappa R)^d + \kappa R\cdot z,\quad \mathscr{B}'(z) = \mathscr{B}(z) + \kappa R \cdot \vec{e}_1,\quad \widehat{\mathscr{B}}(z) = \mathscr{B}(z) \cup \mathscr{B}'(z).$$
Let~$(\xi^{(z)}_t)_{t \geq 0}$ denote the  contact process with range $R$ and infection rate $\lambda$ on~$\widehat{\mathscr{B}}(z)$. Let us construct a joint realization of $(\xi^{(z)}_t)_{t \geq 0}$ for all
$z \in \Z^d$ simultaneously on the same probability space as follows.
 We let~$\xi^{(z)}_0$ be the restriction of~$\xi_0$ to~$\widehat{\mathscr{B}}(z)$, and the dynamics of~$(\xi_t^{(z)})_{t \geq 0}$ be given using the same graphical construction  as the one used for~$(\xi_t)_{t \geq 0}$ in Section \ref{subsection_graphical_constr_of_cont_proc}, but only the Poisson processes involving vertices and edges contained in the box~$\widehat{\mathscr{B}}(z)$.

We then define the family of $\{0,1\}$-valued random variables~$\{X_z:z\in\Z^d\}$ by prescribing that for each~$z$, we set~$X_z = 1$ if one of the following conditions hold:
\begin{itemize}
\item[(a)] both~$\mathscr{B}(z)$ and~$\mathscr{B}'(z)$ are not good for~$\xi_0$ (i.e., $\omega_{\xi_0}(z)=\omega_{\xi_0}(z + \vec{e}_1)=0$);
\item[(b)] at least one of~$\mathscr{B}(z)$ and~$\mathscr{B}'(z)$ is good for~$\xi_0$, and moreover~$\mathscr{B}(z)$ is good for~$\xi^{(z)}_T$
(i.e., $\max\{\omega_{\xi_0}(z),\omega_{\xi_0}(z + \vec{e}_1) \}=1$ and $\omega_{\xi^{(z)}_T}(z)=1$).
\end{itemize}
We set~$X_z = 0$ otherwise. This gives
\begin{equation} \label{eq:compare_omega}\omega_{\xi^{(z)}_T}(z) \geq X_z \cdot \max\{\omega_{\xi_0}(z),\omega_{\xi_0}(z + \vec{e}_1) \},\quad z \in \Z^d.\end{equation}

 By Proposition~\ref{prop:growth_to_adjacent_box}, we have
\begin{equation*}
\mathbb{P}(X_z = 1) \geq 1 - \exp\{-cR^d \}
\end{equation*}
(condition (a) in the definition of~$X_z$ is only present to make this lower bound trivially correct in case neither of the boxes involved are good for $\xi_0$).

\smallskip

Note that the values of~$X_\cdot$ in each line of the form~$\{z + k\vec{e}_1:k \in \Z\}$ are independent of the values of~$X_\cdot$ in all other such lines. Moreover,~$X_z$ is independent of~$\{X_{z + k\vec{e}_1}: k \in \Z\backslash \{-1,0,1\}\}$ under the joint realization described above.

\smallskip

By these considerations, we can use Theorem~\ref{thm:lss} with~$d=1$,~$k=1$ and~$p$ such that~$1-(1-\sqrt{p})^3 = 1 - \exp\{-cR^d\}$, that is,~$p = \theta_1(R)$ (cf.\ \eqref{theta_1_R_def}), to conclude that~$\{X_z:z \in \Z^d\}$ stochastically dominates a field of independent, Bernoulli($\theta_1(R)$) random variables. Comparing~\eqref{eq:def_of_eta} and~\eqref{eq:compare_omega}, the proof is then complete.
\end{proof}
\begin{proof}[Proof of Proposition~\ref{prop:second_domination}]
 Let~$\mu(t)$ be the distribution, at time~$t$, of the contact process $(\xi_.)$  with range~$R$, infection parameter $\lambda$ and initial state~$\xi_0  = \underline{1}$. Let~$\nu(n)$ be the distribution, at time~$n$, of the process $(\eta_.)$ from Definition~\ref{def_cellular_automaton} with~$p = \theta_1(R)$ and initial state~$\eta_0 = \underline{1}$. It is easy to see that~$\mu(mT)$ stochastically dominates~$\nu(m)$, for every~$m \in \mathbb{N}$ by using Lemma~\ref{lem:domination_time_one} iteratively together with the fact that the contact process is attractive (cf.\ below equation (1.1) of \cite[Part 1, Section 1]{Li99}).
  The result then follows by taking~$m \to \infty$.
\end{proof}

\subsection{Propagation of good boxes}
\label{ss:propagation}

In this section, we prove Proposition~\ref{prop:growth_to_adjacent_box}. Again we fix~$d \geq 1$ and~$\lambda > 1$.

We start by giving the value of the constant~$\kappa$ that appears in the statements of Proposition~\ref{prop:growth_to_adjacent_box}:
we choose (and fix)~$\kappa=\kappa(d,\lambda) \in \mathbb{N}$ large enough that
\begin{equation}
\label{eq:choice_of_kappa}
\kappa > 4,\qquad \lambda \left(1-\frac{1}{\kappa}\right)^d > 1 + \frac{3}{4}(\lambda - 1).
\end{equation}
The reason we need the second inequality will become apparent in Section~\ref{subsub_growth_in_box}, but heuristically we want
$\kappa$ to be big enough so that the contact process with infection rate $\lambda$ and range $R$ on $[0, \kappa R)^d$ already exhibits ``supercritical behavior".
 Proposition~\ref{prop:growth_to_adjacent_box} will follow from the next two lemmas.

\begin{lemma}[Infection spreads everywhere in a box]\label{lem:infect_small_boxes}
There exists~$\delta > 0$ and~$m_0 > 0$ (that only depend on $d$ and $\lambda$) such that the following holds. Let~$R > 1$ and
\begin{equation}\label{B_B_prime_B_hat_again}
B = [0,\kappa R)^d,\qquad B' = B + \kappa R \cdot \vec{e}_1,\qquad \widehat{B} = B \cup B',
\end{equation}
where~$\vec{e}_1$ denotes the first canonical vector. Let~$(\xi_t)$ be the contact process on~$\widehat{B}$ with infection rate $\lambda$ and range~$R$. If~$m := |\xi_0| \geq m_0$, then
\begin{equation}\label{eq_infect_small_boxes}
  \mathbb{P}\left(\begin{array}{l}
|\xi_{(4\kappa)^d} \cap ([0,R)^d + R\cdot z )| \geq \delta m \text{ for } \\[.15cm] \text{all~$z \in \Z^d$ such that } [0,R)^d + R\cdot z \subset \widehat{B}
\end{array}\right) > 1-\exp(-\delta m).
\end{equation}
\end{lemma}
We will prove Lemma \ref{lem:infect_small_boxes} in Section \ref{subsub_infect_spreads}.

\begin{lemma}[Supercritical behavior in a box]
\label{lem:growth_in_box}
There exist constants
\begin{equation}\label{bar_R_alpha_T_0_c_0}
 \bar{R} > 0,\quad \alpha \in (0,1),\quad T_0 > 0, \quad c_0 > 0
 \end{equation}
(that only depend on $d$ and $\lambda$) such that the following holds  for any~$R \geq \bar{R}$. Letting~$B = [0,\kappa R)^d$ and~$(\xi_t)$ be the contact process on~$B$ with infection rate~$\lambda$, range~$R$ and initial configuration satisfying~$|\xi_0| \geq \alpha R^d$, we have
\begin{equation}\label{eq:growth_in_box}
\mathbb{P}\left(|\xi_{T_0}| \geq \delta^{-1} \cdot \alpha  R^d\right) > 1 - \exp(-c_0 R^d),
\end{equation}
where~$\delta$ is the constant of Lemma~\ref{lem:infect_small_boxes}.
\end{lemma}
We will prove Lemma \ref{lem:growth_in_box} in Section \ref{subsub_growth_in_box}.

\begin{proof}[Proof of Proposition~\ref{prop:growth_to_adjacent_box}]
Let~$\kappa$ be as fixed in~\eqref{eq:choice_of_kappa},~$\delta, m_0$ as in Lemma~\ref{lem:infect_small_boxes} and $\bar{R}$, $\alpha$, $T_0$, $c_0$ as in Lemma~\ref{lem:growth_in_box}. Choose~$R \geq \bar{R}$, and also large enough that~$\delta^{-1}\alpha R^d \geq m_0$. Define boxes~$B,B',\widehat{B}$ as in~\eqref{eq:boxes12} and assume~$\xi_0$ is a configuration satisfying~\eqref{eq:initial_conf_good}. We clearly have~$|\xi_0 \cap B| \geq \alpha R^d$, hence by Lemma~\ref{lem:growth_in_box}, with probability above~$1-\exp(-c_0 R^d)$, we have~$|\xi_{T_0} \cap B| \geq \delta^{-1}\cdot \alpha R^d$. Conditioned on this, by Lemma \ref{lem:infect_small_boxes}, with probability above~$1 - \exp(- \delta \cdot \delta^{-1}\alpha R^d)$ we have~$|\xi_{T_0 + (4\kappa)^d} \cap ([0,R)^d + R \cdot z])| \geq \alpha R^d$ for all~$z \in \Z^d$ such that~$[0,R)^d + R\cdot z \subset \widehat{B}$. This implies that both~$B$ and~$B'$ are good (cf.\ Definition \ref{def:good_boxes}) at time~$T_0 + (4\kappa)^d$.

\smallskip

The constants in the statement of Proposition~\ref{prop:growth_to_adjacent_box} should thus be chosen as follows:~$\kappa$ and $\alpha$ as already described,~$T = T_0 + (4\kappa)^d$ and finally,~$R_*$ and~$c$ so that~$R_* \geq \max(\bar{R}, (\delta m_0/\alpha)^{1/d})$ and
$$ \exp(-c_0R^d) + \exp(-\alpha R^d) \leq \exp(-cR^d) \quad \text{for all } R \geq R_*. $$
\end{proof}

\subsubsection{Infection spreads everywhere in a box }\label{subsub_infect_spreads}

The goal of Section \ref{subsub_infect_spreads} is to prove Lemma~\ref{lem:infect_small_boxes}.
Let us note that this proof does not  use supercriticality (i.e., $\lambda>1$) in an essential way: in fact the very same proof works for
$\lambda=1$ as well. The proof of Lemma~\ref{lem:infect_small_boxes} will be obtained as a consequence of the following.
\begin{lemma}[Infection of a nearby box]\label{lem:growth_to_neighbor}
There exists~$\delta_0=\delta_0(d)  > 0$ such that the following holds for any~$R \geq 2$ and $\lambda>1$. Let
$$A = [0, \lfloor  R/2 \rfloor )^d,\qquad  A' = A +  u,\qquad  \widehat{A} = A \cup A',$$
where~$u \in \mathbb{Z}^d$ is chosen so that $\max_{x \in A,\; y \in A'} |x-y| \leq R$. Then, letting~$(\xi_t)_{t \geq 0}$ denote the contact process on~$\widehat{A}$ and
$m = |\xi_0 \cap A|$, we have
\begin{equation}\label{eq_growth_to_neigh}
\mathbb{P}\left(|\xi_1 \cap A'| > \delta_0 m \right) \ge 1-\exp\left(- \delta_0 m\right), \qquad m \geq 2.
\end{equation}
\end{lemma}

\begin{proof}[Proof of Lemma~\ref{lem:growth_to_neighbor}]
First assume that~$|\xi_0 \cap A'  | \leq m/2$, so that~$|\xi_0 \cap A \setminus A' | \geq m/2$. Recalling the graphical construction of Section \ref{subsection_graphical_constr_of_cont_proc}, let~$\mathcal{Z}$ be the set of~$x \in A \setminus A' $ such that~$\xi_0(x) = 1$ and there is no recovery mark at~$x$ in the time interval~$[0,1]$ (so that~$\xi_t(x) = 1$ for each~$t \in [0,1]$). Note that
\begin{equation}\label{eq:first_binomial}
|\mathcal{Z}| \text{ stochastically dominates } \mathrm{Bin}(\lfloor m/2 \rfloor,e^{-1}).
\end{equation}
Note that the expectation of this binomial distribution is greater than or equal to $ \frac{m}{4e}$.
For each~$y \in A'$, let~$Y_{y}$ be the indicator of the event that there is no recovery mark at~$y$ in the time interval~$[0,1]$, and moreover there is~$x \in \mathcal{Z}$ and~$t \in [0,1]$ such that there is an infection arrow from~$(x,t)$ to~$(y,t)$. Note that
\begin{equation}\label{eq:y_above}\xi_1(y) \geq Y_{y}\quad \text{for all } y \in A'.\end{equation}
We have  $\max_{x \in \mathcal{Z},\; y \in A'} |x-y| \leq R$, thus for all~$y \in A'$ we obtain that
\begin{equation}\label{Y_y_ind_parameter_lower_bound}
 \mathbb{P}(Y_{y} = 1 \mid \mathcal{Z}) = e^{-1} \cdot \left(1- \exp\left(-\frac{\lambda |\mathcal{Z}|}{|B(R)|} \right) \right) \stackrel{ \lambda >1 }{\geq}
  \frac{ e^{-1}|\mathcal{Z}|}{2|B(R)|}.
  \end{equation}
Next we show that
\begin{equation}\label{Y_y_ind_indep}
\text{conditioned on~$\mathcal{Z}$, the random variables~$(Y_{y})_{y  \in A'}$ are independent.}
\end{equation}
Indeed,
$\mathcal{Z}$ only depends on the recovery marks of the vertices of $\xi_0 \cap A \setminus A'$, and given $\mathcal{Z}$,  $Y_{y}$ (where $y \in A'$) only depends on the
recovery marks of $y$ and the infection arrows that point from $\mathcal{Z}$ to $y$. Recalling from Section \ref{subsection_graphical_constr_of_cont_proc} that these Poisson processes are all independent, we infer that \eqref{Y_y_ind_indep} holds.

 Let us define $Y = \sum_{y \in A'}Y_{y}$. Putting together \eqref{Y_y_ind_parameter_lower_bound} and \eqref{Y_y_ind_indep}, we obtain that
\begin{equation} \label{eq:second_binomial}
Y \text{ stochastically dominates $\mathrm{Bin}\left( \left\lfloor \frac{R}{2}\right\rfloor^d,\;
 \frac{ e^{-1}|\mathcal{Z}|}{2|B(R)|}  \right)$. }\end{equation}
Conditional on $\mathcal{Z}$, the expectation of this binomial distribution is greater than equal to $\delta'' |\mathcal{Z}|$, where $\delta''$
only depends on $d$ (but not on $R$).
We estimate
\begin{align}
\mathbb{P}\left(Y \geq \frac{\delta''}{16e} m\right) \ge  \mathbb{P}\left(\left. Y \geq \frac{\delta''  |\mathcal{Z}|}{2} \; \right| \;|\mathcal{Z}|= \left\lceil \frac{m}{8e} \right \rceil\right)  \mathbb{P}\left(|\mathcal{Z}| \geq \frac{m}{8e}\right).\label{eq:want_lower_bound}
\end{align}

We now use~\eqref{eq:first_binomial},~\eqref{eq:second_binomial} and the fact that
\begin{equation}\label{chernoff_bound}
Z \sim \mathrm{Bin}(n,p) \quad \Longrightarrow \quad \mathbb{P}(Z \geq \textstyle{\frac{1}{2}} np) \geq 1- \exp\left(- \textstyle{\frac{1}{8}} np \right),
\end{equation}
which is an easy consequence of the Chernoff bound, see for instance Theorem~2.21, page 70 in \cite{remco}. Combining \eqref{chernoff_bound} with the lower bounds on the expectations
of the binomial distributions that appear in \eqref{eq:first_binomial} and \eqref{eq:second_binomial}, we see that
the product on the r.h.s.\ of~\eqref{eq:want_lower_bound} is larger than
$$
 \left(1- \exp\left(- \frac{1}{8} \delta'' \frac{m}{8e}  \right)\right)\left(1- \exp\left(- \frac{1}{8} \frac{m}{4 e} \right)\right) .
$$
The desired bound \eqref{eq_growth_to_neigh} now follows using \eqref{eq:y_above} if~$\delta_0$ is taken small enough.

The case where~$|\xi_0 \cap A'| \geq m/2$ is much easier, so we omit it.
\end{proof}

\begin{proof}[Proof of Lemma~\ref{lem:infect_small_boxes}]
We abbreviate
$$K := (4\kappa)^d.$$
Let~$\Xi$ be a set of boxes of the form~$[0, \lfloor \tfrac{R}{2} \rfloor )^d + x$ with~$x\in \Z^d$ such that
\begin{equation}\label{cover_boxes}
 \widehat{B} = \bigcup_{A \in \Xi} A \qquad \text{and}\qquad |\Xi| \leq K,
 \end{equation}
where $\widehat{B}$ was defined in \eqref{B_B_prime_B_hat_again}.

Noting that~$m = |\xi_0| \leq K\cdot \max_{A \in {\Xi}} |\xi_0 \cap A|$,
there exists~$A_* \in \Xi$ such that
$$|\xi_0 \cap A_*| \geq \frac{m}{K}.$$

Now let~$\Gamma$ be the set of all sequences of the form~$\gamma = (A_0, A_1,\ldots, A_{K})$ of elements of~$\Xi$ with~$A_0 = A_*$ such that
$$
\max_{x \in A_i,\; y \in A_{i+1}} |x-y| \leq R, \qquad 0 \leq i \leq K-1.
$$
For each~$\gamma \in \Gamma$, define the event
$E_\gamma = \bigcap_{i = 1}^{K} \left\{|\xi_{i} \cap A_i| \geq (\delta_0)^i \cdot \frac{m}{K} \right\}$, where $\delta_0$ was defined in Lemma \ref{lem:growth_to_neighbor}.
Observe that it follows from \eqref{cover_boxes} and the definition of $\Gamma$ that
$$\bigcap_{\gamma \in \Gamma} E_\gamma \subset \left\{\begin{array}{l}
|\xi_{K} \cap ([0,R)^d + R\cdot z )| \geq \frac{(\delta_0)^{K}}{K} m \text{ for } \\[.15cm] \text{all~$z \in \Z^d$ such that } [0,R)^d + R\cdot z \subset \widehat{B}
\end{array}\right\}$$
and, by Lemma~\ref{lem:growth_to_neighbor} and a union bound, we obtain
$$\mathbb{P}\left(\bigcup_{\gamma \in \Gamma} E_\gamma^c \right) \leq (K)^{K} \cdot \sum_{i=1}^{K} \exp\left(-\delta_0\cdot \frac{(\delta_0)^i}{K} \cdot m \right),$$
thus we can choose~$\delta=\delta(d,\lambda)>0$  so that \eqref{eq_infect_small_boxes} holds for large enough~$m$.
\end{proof}

\subsubsection{Supercritical behavior in a box}\label{subsub_growth_in_box}

The goal of Section \ref{subsub_growth_in_box} is to prove Lemma~\ref{lem:growth_in_box}.
We will first need several preliminary definitions.
Recall that we have already fixed the value of $\kappa=\kappa(d,\lambda)$ in \eqref{eq:choice_of_kappa}.
Throughout this section, we write
$ B = [0,\kappa R)^d$.

Let~$P = (P(x,y): x,y \in B)$ denote the sub-stochastic matrix given by
\begin{equation}\label{P_substoch_matrix}
 P(x,y) := \frac{1}{|B(R)|}\cdot \mathds{1} [y \in B(x,R)].
\end{equation}
This is the transition matrix of the discrete-time random walk on $B$ with range $R$ which gets killed if it jumps out of $B$.

\begin{remark}
Let us denote by $\mu$ the principal eigenvalue of $P$.   It is easy to see that the expected population size of the branching random walk (BRW) on $B$ with range $R$, birth rate $\lambda$ and death rate $1$ grows/decays exponentially with rate $\lambda \mu -1$. We have $\mu<1$, but $\mu$ gets arbitrarily close to $1$ if $\kappa$ is big enough, so one can guarantee supercriticality  of the BRW on $B$ (i.e., $\lambda \mu >1$) for any given $\lambda>1$ by making $\kappa$ big enough.

\smallskip

We will show that the function $h$ defined in \eqref{h_prod_def} below satisfies
$P h\geq (1-\frac{1}{\kappa})^d h$  and our condition \eqref{eq:choice_of_kappa} on $\kappa$ will guarantee that contact process on $B$
with range $R$, birth rate $\lambda$ and death rate $1$ tends to behave in a supercritical fashion as long as the density of infected sites in $B$ is low enough.
\end{remark}

We define the functions~$\ell: [0,\kappa] \to \mathbb{R}_+$ and $h: B \to \mathbb{R}_+$ by
\begin{align}
\label{eq:def_of_ell}
\ell(u) := &
\frac{1}{\kappa} + \min\{u,\;\kappa -u \}, \\
\label{h_prod_def}
 h(x) := & \prod_{i=1}^d \ell \left( \frac{x_i}{R} \right), \qquad x=(x_1,\ldots, x_d) \in B.
\end{align}

\begin{lemma}[Perron-Frobenius bounds]
\label{lem:first_bound_density}
There exists~$R_\diamond = R_\diamond(d,\lambda)$ such that for any~$R \ge R_\diamond$ we have~$Ph\geq (1-\frac{1}{\kappa})^d h$, that is,
\begin{equation}
 \label{eq:first_bound_density}
 \frac{1}{|B(R)|}\cdot \sum_{y \in B \cap B(x,R)}\;h(y)\geq \left(1-\frac{1}{\kappa}\right)^d   h(x)\quad \text{for every~$x \in B$}.
 \end{equation}
\end{lemma}
We will prove Lemma \ref{lem:first_bound_density} in Section \ref{subsubsection_perron}.

We define $g: B \to \mathbb{R}_+$ by normalizing $h$:
\begin{equation}\label{g_def_eq_from_h}
 g(x) \stackrel{ \eqref{h_prod_def} }{:=} \frac{h(x)}{\max_{y \in B} h(y)},\qquad x \in B.
 \end{equation}

In Section \ref{subsub_growth_in_box}, $(\xi_t)_{t \geq 0}$ denotes the contact process on~$B=[0,\kappa R)^d$.
Denote by $(\mathscr{F}_t)_{t \geq 0}$  the natural filtration of~$(\xi_t)_{t \geq 0}$.
We define
\begin{equation}\label{N_X_def}
 N_t := \sum_{x\in B} \xi_t(x)= |\xi_t|  ,\qquad X_t := \sum_{x \in B} g(x) \xi_t(x).
 \end{equation}
 Roughly speaking, our goal is to show that $(N_t)$ grows exponentially as long as $N_t/R^d$ is small enough to guarantee that the fraction of  healthy individuals  in any ball of radius $R$ is close enough to $1$.
 We would like to achieve this using that the infection rate $\lambda$ is strictly greater than the recovery rate $1$.
 However, the effective outgoing infection rate of infected individuals near the boundary of the box $B$ can be strictly smaller than $1$, so
   it will be easier to prove that $(X_t)$ grows exponentially (using $P g \geq (1-\frac{1}{\kappa})^d g$ along the way).
     A lower bound on $N_0$ implies a lower bound on $X_0$ and a lower bound on $(X_t)$ implies a lower bound on $(N_t)$, since
\begin{equation} \label{eq:bounds_on_g}
g_{\min} :=  \min_{x \in B} g(x) \ge \left(\frac{\tfrac{1}{\kappa}}{\tfrac{1}{\kappa} + \tfrac{\kappa}{2}}\right)^d,\qquad
g_{\max} := \max_{x \in B} g(x) = 1,
\end{equation}
therefore $X_t$ and $N_t$ are comparable:
\begin{equation}\label{N_t_X_t_comparable}
g_{\min} \cdot  N_t \leq  X_t \leq N_t, \qquad t \geq 0.
\end{equation}
Roughly speaking, we will show that $X_t$ grows exponentially with rate at least $\frac{\lambda-1}{4}$ as long as $N_t/R^d$ is small enough, so in order to guarantee
 \eqref{eq:growth_in_box}, we define~$T_0=T_0(d,\lambda) > 0$ such  that
\begin{equation}\label{choose_T_0}
\frac{g_{\min} }{2} \exp\left\{ \frac{\lambda - 1}{4} {T_0}\right\} = \delta^{-1},
\end{equation}
where~$\delta=\delta(d,\lambda)$ is the constant of Lemma~\ref{lem:infect_small_boxes}.

However, we will only guarantee exponential growth of $X_t$ with rate at least~$\frac{\lambda-1}{4}$ as long as
 $N_t \leq \beta R^d$, where $\beta >0$ is small enough to guarantee that the fraction of  healthy individuals  in any ball of radius $R$ is close enough to $1$ (so that
 the contact process locally exhibits a positive expected rate of growth):
 \begin{equation}\label{beta_choice}
 \lambda \left(\left(1- \frac{1}{\kappa}\right)^d - \frac{\beta \cdot R^d}{|B(R)| \cdot g_{\min}}\right) > 1 + \frac12(\lambda - 1), \quad  \forall\,  R \ge 1.
\end{equation}
Such a choice of $\beta$ is possible by~\eqref{eq:choice_of_kappa}, so let us fix $\beta = \beta(d,\lambda)>0$ satisfying~\eqref{beta_choice} for
the rest of Section \ref{subsub_growth_in_box}.

\smallskip

In order to guarantee that $N_t \leq \beta R^d$ holds for $0 \leq t \leq T_0$, we will  argue that the exponential growth rate of $N_t$ is at most $2(\lambda-1)$
and choose the constant $\alpha=\alpha(d,\lambda)>0$ (cf.\  \eqref{bar_R_alpha_T_0_c_0})
 so that
\begin{equation}\label{eq:choice_of_alpha}
2\alpha \exp\{2(\lambda - 1) {T_0}\} = \beta.
\end{equation}

Define, for~$t \geq 0$,
\begin{equation}
\label{eq:def_of_a_and_b}
a(t) := \alpha \frac{ g_{\min}}{2} \exp\left\{\frac{\lambda - 1}{4}\cdot t\right\} R^d ,
\quad
b(t) := 2\alpha  \exp\{2(\lambda - 1)t\} R^d .
\end{equation}

Let us define the stopping times
\begin{equation}\label{tau_def}
\tau_a := \inf\{\, t: X_t \leq a(t) \, \} ,
\quad
\tau_b := \inf\{ \, t: N_t \geq b(t) \, \},
\quad \tau := \tau_a \wedge \tau_b.
\end{equation}
Note that for every~$0 \leq  t \leq {T_0}$ we have
\begin{equation}
\label{eq:condition_on_alpha_beta}
\frac{\alpha g_{\min}}{2} R^d \leq a(t) \leq X_t \leq N_t \leq b(t) \stackrel{ \eqref{eq:choice_of_alpha}  }{\leq } \beta R^d \quad\text{ on } \quad \{\tau > t\}.
\end{equation}
Our goal will be to show that $\mathbb{P}(\tau \leq T_0)$ is very small.
In order to do so, we define the stochastic processes $(U_t)_{0 \leq t \leq {T_0}}$ and $(V_t)_{0 \leq t \leq {T_0}}$ by
\begin{equation}\label{U_t_V_t_def}
 U_t := \exp\{-\varepsilon(X_t - a(t) + a(0))\} ,\quad V_t :=    \exp\{\varepsilon(N_t - b(t) + b(0))\},
\end{equation}
where
\begin{equation}\label{epsilon_choice}
  \varepsilon= \varepsilon(d,\lambda):=\frac{\lambda-1}{4(\lambda+1)} g_{\min}.
\end{equation}

\begin{proposition}[Exponential submartingales] \label{prop:supermartingales}
 If we define $R_\diamond$, $T_0$, $\alpha$, $\tau$, $(U_t)$ and $(V_t)$ as above and
assume~$|\xi_0| = \lfloor \alpha R^d \rfloor$, then for any $R \geq R_\diamond$
 \begin{equation}\label{U_t_V_t_are_supermartingales}
 (U_{t \wedge \tau})_{0 \leq t \leq {T_0}}\; \text{and}\; (V_{t \wedge \tau})_{0 \leq t \leq {T_0}} \;  \text{are supermartingales w.r.t. }(\mathscr{F}_t)_{ 0 \leq t \leq {T_0}}.
 \end{equation}
 \end{proposition}
Before we prove Proposition \ref{prop:supermartingales}, we deduce Lemma~\ref{lem:growth_in_box} from it.

\begin{proof}[Proof of Lemma~\ref{lem:growth_in_box}] We have already fixed $T_0=T_0(d,\lambda)$ in \eqref{choose_T_0} and $\alpha=\alpha(d,\lambda)$ in \eqref{eq:choice_of_alpha}.  We will use Proposition \ref{prop:supermartingales}, so we assume $R \geq R_\diamond$.
 By the monotonicity of the contact process, it suffices to prove that~\eqref{eq:growth_in_box} holds when~$\xi_0$ satisfies~$|\xi_0|= \lfloor \alpha R^d \rfloor$.
   We start the proof of \eqref{eq:growth_in_box} by observing that
  \begin{multline}
  \mathbb{P}\left(|\xi_{T_0}| \leq \delta^{-1} \cdot \alpha  R^d \right) \stackrel{ \eqref{N_X_def}, \eqref{N_t_X_t_comparable} }{\leq}
  \mathbb{P}\left(X_{T_0} \leq \delta^{-1} \cdot \alpha  R^d\right) \stackrel{ \eqref{choose_T_0}, \eqref{eq:def_of_a_and_b} }{=} \\
  \mathbb{P}\left(X_{T_0} \leq a(T_0) \right)
   \stackrel{ \eqref{tau_def} }{\leq}   \mathbb{P}\left( \tau \leq T_0 \right)
  =  \mathbb{P}(\tau = \tau_a \le {T_0}) + \mathbb{P}(\tau = \tau_b \le {T_0}).
  \end{multline}
 We will bound the two terms on the right-hand side.
   We treat the second term first:
\begin{multline}\label{splitting_a_b_prob}
\exp\{\varepsilon \lfloor \alpha R^d\rfloor\} \stackrel{ \eqref{N_X_def}, \eqref{U_t_V_t_def} }{=}
 V_0  \stackrel{\eqref{U_t_V_t_are_supermartingales} }{\geq} \mathbb{E}[V_{{T_0} \wedge \tau}]  \geq \mathbb{E}[V_{\tau_b} \cdot \mathds{1}_{\{\tau = \tau_b \le {T_0} \}}]
 \stackrel{ \eqref{tau_def}, \eqref{U_t_V_t_def} }{\geq} \\
  \exp\{\varepsilon \cdot b(0)\} \cdot \mathbb{P}(\tau = \tau_b \le {T_0})
  \stackrel{ \eqref{eq:def_of_a_and_b} }{=}  \exp\{\varepsilon \cdot 2\alpha R^d \} \cdot \mathbb{P}(\tau = \tau_b \le {T_0}).
\end{multline}
 Rearranging this,
we obtain $\mathbb{P}(\tau = \tau_b \le {T_0}) \leq \exp\{- \varepsilon \alpha R^d \}.$

Similarly, we bound
\begin{multline*}
\exp\{-\varepsilon g_{\min} \lfloor \alpha R^d\rfloor \}\stackrel{ \eqref{N_X_def} }{=}\exp\{-\varepsilon g_{\min} N_0 \}
\stackrel{ \eqref{N_t_X_t_comparable} }{\geq }  \exp\{-\varepsilon  X_0 \}
 \stackrel{ \eqref{U_t_V_t_def} }{=} U_0 \stackrel{ \eqref{U_t_V_t_are_supermartingales} }{\geq}
 \\
  \mathbb{E}[U_{{T_0} \wedge \tau}]
 \geq \mathbb{E}[U_{\tau_a} \cdot \mathds{1}_{\{\tau = \tau_a \le {T_0}\}}] \stackrel{ \eqref{tau_def}, \eqref{U_t_V_t_def} }{\geq} \exp\{-\varepsilon a(0)\} \cdot \mathbb{P}(\tau = \tau_a \le {T_0}) \stackrel{ \eqref{eq:def_of_a_and_b} }{=} \\
  \exp\{-\varepsilon g_{\min} \textstyle{\frac{1}{2}} \alpha R^d  \} \cdot \mathbb{P}(\tau = \tau_a \le {T_0}).
\end{multline*}
Rearranging this and introducing a factor~$\tfrac12$ to take care of the rounding, we obtain $\mathbb{P}(\tau = \tau_a \le {T_0})\leq  \exp\{-\varepsilon g_{\min} \frac14\alpha R^d)  \}$.
Putting together the bounds that we have obtained on the two terms on the right-hand side of \eqref{splitting_a_b_prob}, we see that indeed there exists
$\bar{R}=\bar{R}(d,\lambda)<+\infty$ and $c_0=c_0(d,\lambda)>0$ such that \eqref{eq:growth_in_box} holds for any $R \geq \bar{R}$. The proof of
Lemma~\ref{lem:growth_in_box} is complete.
\end{proof}

It remains to prove  Proposition \ref{prop:supermartingales}.

Recall the definition of $R_\diamond=R_\diamond(d,\lambda)$ (cf.\ Lemma \ref{lem:first_bound_density}), $(N_t)$ and $(X_t)$ (cf.\ \eqref{N_X_def}), $\beta=\beta(d,\lambda)$ (cf.\ \eqref{beta_choice}) and $\varepsilon=\varepsilon(d,\lambda)> 0$ (cf.\ \eqref{epsilon_choice}).
\begin{lemma}[Exponential drift bounds] \label{lem:for_supermartingale}
For all~$t \geq 0$ the following statements hold.
\begin{enumerate}[(i)]
\item  We have
\begin{equation}\label{eq:bound_for_exp_N}
\left. \frac{d}{ds} \mathbb{E}\left[ e^{\varepsilon N_{t+s}} | \mathscr{F}_t\right] \right|_{s = 0} \leq 2 \varepsilon (\lambda - 1)\cdot  N_t \cdot e^{\varepsilon N_t}.
\end{equation}
\item For all $R \geq R_\diamond$, on the event
$\left\{ N_t \leq \beta R^d \right\}$,  we have
\begin{equation}\label{eq:bound_for_exp_X}
\left. \frac{d}{ds} \mathbb{E}\left[ e^{- \varepsilon X_{t+s}} | \mathscr{F}_t\right] \right|_{s = 0}  \leq -\frac{\varepsilon (\lambda - 1)}{4}\cdot X_t \cdot e^{-\varepsilon X_t}.
\end{equation}
\end{enumerate}
\end{lemma}

Before we prove Lemma \ref{lem:for_supermartingale}, let us deduce Proposition \ref{prop:supermartingales} from it.

\begin{proof}[Proof of Proposition \ref{prop:supermartingales}]
It suffices to prove that, for~$0 \leq t < {T_0}$,
\begin{equation}\label{eq:wanted_ineqs}
\left. \frac{d}{ds}  \mathbb{E}[U_{(t+s) \wedge \tau}|\mathscr{F}_t] \right|_{s = 0} \leq 0,\qquad \left. \frac{d}{ds}  \mathbb{E}[V_{(t+s) \wedge \tau}|\mathscr{F}_t] \right|_{s = 0} \leq 0.
\end{equation}

For the first inequality in~\eqref{eq:wanted_ineqs}, we assume that $R\geq R_\diamond$, and we estimate
\begin{align*}
& \left. \frac{d}{ds}  \mathbb{E}[U_{(t+s) \wedge \tau}|\mathscr{F}_t] \right|_{s = 0} =
\mathds{1}_{\{\tau > t\}} \left. \frac{d}{ds}  \mathbb{E}[  e^{ \varepsilon (a(t+s) - a(0))} e^{-\varepsilon X_{t+s}}    |\mathscr{F}_t] \right|_{s = 0} \\[.2cm]
&=  \mathds{1}_{\{\tau > t\}} e^{\varepsilon(a(t) - a(0))}  \left( \left. \frac{d}{ds}  \mathbb{E}[e^{-\varepsilon X_{t + s}}|\mathscr{F}_t] \right|_{s = 0} + \varepsilon a'(t) e^{-\varepsilon X_t} \right)\\[.2cm]
&\stackrel{(*)}{\leq} \mathds{1}_{\{\tau > t\}} e^{ \varepsilon (a(t) - a(0))} e^{-\varepsilon X_{t}}   \frac{\varepsilon(\lambda - 1)}{4} \left(-X_t + a(t)\right) \stackrel{(**)}{\leq} 0,
\end{align*}
where in $(*)$ we used \eqref{eq:bound_for_exp_X} (which can be applied, since $\{\tau > t\} \subseteq \left\{ N_t \leq \beta R^d \right\} $, cf.\  \eqref{eq:condition_on_alpha_beta}) together with $a'(t)=\frac{\lambda-1}{4}a(t)$ (cf.\ \eqref{eq:def_of_a_and_b}), and in $(**)$ we used that
$\{\tau > t\} \subseteq \left\{ X_t \geq a(t) \right\} $, cf.\  \eqref{eq:condition_on_alpha_beta}.

For the second inequality in~\eqref{eq:wanted_ineqs}, we analogously estimate
\begin{align*}
&\left. \frac{d}{ds}  \mathbb{E}[V_{(t+s) \wedge \tau}|\mathscr{F}_t] \right|_{s = 0} =
\mathds{1}_{\{\tau > t\}} \left. \frac{d}{ds}  \mathbb{E}[  e^{-\varepsilon(b(t+s) - b(0))} e^{\varepsilon N_{t + s}}  |\mathscr{F}_t] \right|_{s = 0} \\[.2cm]
&=  \mathds{1}_{\{\tau > t\}} e^{-\varepsilon(b(t) - b(0))}  \left( \left. \frac{d}{ds}  \mathbb{E}[e^{\varepsilon N_{t + s}}|\mathscr{F}_t] \right|_{s = 0}  -\varepsilon b'(t) e^{\varepsilon N_t} \right)\\[.2cm]
&\stackrel{\eqref{eq:def_of_a_and_b},\eqref{eq:bound_for_exp_N} }{\leq} \mathds{1}_{\{\tau > t\}}  e^{-\varepsilon(b(t) - b(0))} e^{\varepsilon N_{t }}   2\varepsilon (\lambda - 1) (N_t - b(t) ) \stackrel{ \eqref{eq:condition_on_alpha_beta} }{\leq} 0.
\end{align*}
The proof of Proposition \ref{prop:supermartingales} is complete.
\end{proof}

It remains to prove Lemma \ref{lem:for_supermartingale}.
Recall the definition of $R_\diamond=R_\diamond(d,\lambda)$ (cf.\ Lemma \ref{lem:first_bound_density}), $g: B \to \mathbb{R}_+$ (cf.\ \eqref{g_def_eq_from_h})
and $\beta=\beta(d,\lambda)$ (cf.\ \eqref{beta_choice}).
\begin{lemma}[Local supercriticality at low density]
\label{lem:low_density}   If~$R \ge R_\diamond$  and~$\xi \in \{0,1\}^B$ satisfies~$|\xi| \leq \beta R^d$, then for any~$x \in B$ we have
\begin{equation}\label{eq:lower_bound_with_density}
\frac{\lambda}{|B(R)|} \cdot \sum_{y \in B \cap B(x,R)}  (1-\xi(y))\cdot g(y) > \frac{1+ \lambda}{2} \cdot g(x).
\end{equation}
\end{lemma}

Before we prove Lemma \ref{lem:low_density}, let us deduce Lemma \ref{lem:for_supermartingale} from it.

\begin{proof}[Proof of Lemma \ref{lem:for_supermartingale}]
We begin with the proof of~\eqref{eq:bound_for_exp_N}:
\begin{align}
\nonumber&\left. \frac{d}{ds} \mathbb{E}\left[ e^{\varepsilon N_{t+s}}|\mathscr{F}_t \right] \right|_{s = 0} = e^{\varepsilon N_t} \cdot \left. \frac{d}{ds}\mathbb{E}\left[ e^{\varepsilon (N_{t+s}-N_t)}|\mathscr{F}_t\right] \right|_{s=0}  \stackrel{ \eqref{eq:generator_on_B} }{=} \\[.2cm]
\nonumber& e^{\varepsilon N_t}  \sum_{x \in B} \xi_t(x)\left( e^{-\varepsilon} - 1 + \frac{\lambda}{|B(R)|}\cdot\sum_{y \in B(x,R) \cap B} (1-\xi_t(y))\cdot (e^{\varepsilon}-1) \right)\\[.2cm]
\label{eq:aux_eq_error}&\leq  e^{\varepsilon N_t} \sum_{x \in B} \xi_t(x) \left( e^{-\varepsilon} - 1+\lambda (e^{\varepsilon}-1) \right)
\stackrel{(**)}{\leq} 2 \varepsilon (\lambda - 1)\cdot  N_t \cdot e^{\varepsilon N_t},
\end{align}
where in $(**)$ we used \eqref{N_X_def} and that $e^{-\varepsilon} - 1+\lambda (e^{\varepsilon}-1) \leq 2 \varepsilon (\lambda - 1)$ holds by our choice of $\varepsilon$
in \eqref{epsilon_choice}.

Now we prove~\eqref{eq:bound_for_exp_X}. We have
\begin{equation} \label{eq:take_out_exp} \frac{d}{ds}\mathbb{E}\left.\left[e^{-\varepsilon X_{t+s}} | \mathscr{F}_t\right] \right|_{s= 0} = e^{-\varepsilon X_t} \cdot \frac{d}{ds} \mathbb{E}\left.\left[ e^{-\varepsilon (X_{t+s} - X_t)} | \mathscr{F}_t\right]\right|_{s = 0}\end{equation}
and
\begin{align}
\nonumber & \frac{d}{ds} \mathbb{E}\left.\left[ e^{-\varepsilon (X_{t+s} - X_t)} | \mathscr{F}_t\right]\right|_{s = 0} \stackrel{ \eqref{eq:generator_on_B} }{=} \\[.2cm]
\nonumber&     \sum_{\substack{x \in B}}  \xi_t(x) \left(  e^{\varepsilon g(x)} - 1  +  \frac{\lambda}{|B(R)|}\cdot\sum_{\substack{y \in B(x,R) \cap B}} (1-\xi_t(y))\cdot (e^{-\varepsilon g(y)} - 1) \right)  \\[.2cm]
\label{eq:aux_long_eq}&=   \varepsilon      \sum_{\substack{x \in B}}  \xi_t(x) \left(   g(x) -  \frac{\lambda}{|B(R)|}\cdot\sum_{\substack{y \in B(x,R) \cap B}} (1-\xi_t(y))\cdot  g(y) \right) \\[.2cm]
\label{eq:aux_long_eq_error}&\hspace{1cm} +   \sum_{\substack{x \in B}}  \xi_t(x) \left(  \mathcal{E}_x(\varepsilon) +  \frac{\lambda}{|B(R)|}\cdot\sum_{\substack{y \in B(x,R) \cap B}} (1-\xi_t(y))\cdot \tilde{\mathcal{E}}_y(\varepsilon) \right),
\end{align}
where we write
\begin{equation}\label{error_terms_def_g_eps}
 \mathcal{E}_x(\varepsilon) :=  e^{\varepsilon g(x)} - 1 -  \varepsilon g(x)\geq 0,\qquad \tilde{\mathcal{E}}_x(\varepsilon) := e^{-\varepsilon g(x)} - 1 + \varepsilon g(x)\geq 0.
\end{equation}
By our assumption $R \geq R_\diamond$ and Lemma~\ref{lem:low_density}, on the event~$\{N_t \leq \beta R^d \}$, the term in~\eqref{eq:aux_long_eq} is less than or equal to
\begin{equation}\label{linear_term_bound_g_eps}
\varepsilon    \sum_{x \in B} \xi_t(x) \left(g(x)- \frac{\lambda + 1}{2} g(x) \right) \stackrel{ \eqref{N_X_def} }{=} - \frac{\varepsilon(\lambda - 1)}{2} \cdot X_t.
\end{equation}
Let us denote $M=\max_{x \in B} \left(\mathcal{E}_x(\varepsilon) \vee \tilde{\mathcal{E}}_x(\varepsilon)\right)$. We bound the term in~\eqref{eq:aux_long_eq_error}
by
\begin{equation}\label{error_term_bound_g_eps}
\sum_{x \in B} \xi_t(x) ( M +\lambda M) \stackrel{ \eqref{N_X_def} }{=} M  (1+\lambda) N_t \stackrel{ \eqref{N_t_X_t_comparable} }{\leq} \frac{M  (1+\lambda)}{ g_{\min} } X_t \stackrel{(*)}{\leq } \frac{\varepsilon(\lambda - 1)}{4}  X_t,
\end{equation}
where in  $(*)$ we used
\begin{equation*}
 M \stackrel{\eqref{epsilon_choice}, \eqref{error_terms_def_g_eps} }{\leq} \max_{x \in B}\, ( \varepsilon g(x) )^2 \stackrel{\eqref{eq:bounds_on_g} }{\leq} \varepsilon^2
 \stackrel{\eqref{epsilon_choice}}{=} \frac{\varepsilon(\lambda-1)}{4(1+\lambda)}g_{\min}.
\end{equation*}
Putting together the upper bounds  we obtained for the terms in~\eqref{eq:aux_long_eq} and~\eqref{eq:aux_long_eq_error}, we see that the right-hand side of~\eqref{eq:take_out_exp}  satisfies the desired inequality for~\eqref{eq:bound_for_exp_X}.
The proof of Lemma \ref{lem:for_supermartingale} is complete.
\end{proof}

\begin{proof}[Proof of Lemma \ref{lem:low_density}]
Since~$g_{\max}= 1$, the left-hand side of~\eqref{eq:lower_bound_with_density} is larger than
\begin{equation}\label{intermediate_quantity}
\frac{\lambda}{|B(R)|}\cdot \left( \sum_{\substack{y \in B \cap B(x,R)}}\;  g(y) - |\xi| \right).
\end{equation}
Assuming $R \ge R_\diamond$, we can use~\eqref{eq:first_bound_density}, \eqref{g_def_eq_from_h} and the assumption~$|\xi| \leq \beta R^d$ to bound \eqref{intermediate_quantity} from below by
\begin{equation} \label{eq:kappa_beta}
\lambda \left(\left(1-\frac{1}{\kappa}\right)^d  g(x) - \frac{ \beta \cdot R^d}{|B(R)|} \right) \geq \lambda \left(\left(1- \frac{1}{\kappa}\right)^d - \frac{\beta \cdot R^d}{|B(R)|  g_{\min}}\right)  g(x).
\end{equation}
By our choice of $\beta$ in \eqref{beta_choice}, the right-hand side of~\eqref{eq:kappa_beta} is greater than or equal to~$\frac{1+\lambda}{2}\cdot g(x)$ for any~$R \ge R_\diamond$ and any~$x \in B$, completing the proof of \eqref{eq:lower_bound_with_density}.
\end{proof}

\subsubsection{Perron-Frobenius bounds}
\label{subsubsection_perron}

The goal of Section \ref{subsubsection_perron} is to prove Lemma \ref{lem:first_bound_density}.

Recall from \eqref{eq:def_of_ell} that $\ell(u)=
\frac{1}{\kappa} + \min\{u,\;\kappa -u \}, \, 0 \leq u \leq \kappa$.

 We define~$a: [0,\kappa] \to \mathbb{R}$ by
\begin{equation}
\label{eq:def_of_a}
a(u) :=
 \frac12 \int_{(u-1)\vee 0}^{(u+1)\wedge \kappa} \ell(v)\;\textrm{d}v - \left(1 - \frac{1}{\kappa}\right) \ell(u).
\end{equation}

\begin{lemma}[Plain calculus] \label{lem:calculus}
We have~$\min_{0 \leq u \leq \kappa } a(u) > 0$.
\end{lemma}
Before we prove Lemma \ref{lem:calculus}, let us deduce Lemma \ref{lem:first_bound_density} from it.

\begin{proof}[Proof of Lemma \ref{lem:first_bound_density}]
By \eqref{h_prod_def}, the left-hand side of~\eqref{eq:first_bound_density} is equal to
$$\prod_{i=1}^d \frac{1}{2R+1}  \sum_{\substack{y_i \in \{0,\ldots, \kappa R-1\},\\  |y_i - x_i| \leq R}} \ell \left( \frac{y_i}{R} \right), $$
so~\eqref{eq:first_bound_density} will follow if we prove that, for each~$x \in B$ and for each~$i \in \{1,\ldots, d\}$,
\begin{equation}\label{ineq_need_perronfrob}
 \frac{1}{2R+1}  \sum_{\substack{y_i \in \{0,\ldots, \kappa R-1\},\\ 0 \leq |y_i - x_i| \leq R}} \ell \left( \frac{y_i}{R} \right) \geq \left(1- \frac{1}{\kappa}\right) \ell \left(\frac{x_i}{R}\right).
 \end{equation}
This amounts to the same as proving the lemma in the case~$d= 1$, so we assume~$d = 1$ (and thus drop the subscript~$i$) in the rest of this proof.

Denote
$$ a_R(x) :=  \frac{1}{2R+1} \sum_{\substack{ y \in \{0,\ldots, \kappa R-1\},\\ 0 \leq |y - x| \leq R }} \ell \left(\frac{y}{R}\right) -  \left(1-\frac{1}{\kappa}\right)\ell\left( \frac{x}{R}\right),\;\; x \in \{0,\ldots, \kappa R - 1\}. $$
Comparing this to the definition of the function~$a$ from~\eqref{eq:def_of_a}, we readily obtain
$$ \lim_{R \to \infty} \max_{x \in \{0,\ldots, \kappa R-1\}} |a_R(x) - a(x/R)| = 0.$$
Then, by Lemma~\ref{lem:calculus}, we conclude that if~$R$ is large enough, we have~$a_R(x) > 0$ for all~$x \in \{0,\ldots,\kappa R - 1\}$, completing the proof of
\eqref{ineq_need_perronfrob}.
 \end{proof}

\begin{proof}[Proof of Lemma \ref{lem:calculus}]
 By symmetry, it is sufficient to prove that~$a(u) > 0$ for all~$u \in [0,\kappa/2]$.
 Recall from \eqref{eq:choice_of_kappa} that $\kappa >4$, thus~$\tfrac{\kappa}{2} + 1 < \kappa$, so we can write
$$a(u)= \frac12\int_0^{u+1} \ell(v)\;\textrm{d}v - \frac12\int_0^{(u-1)\vee 0} \ell(v)\;\textrm{d}v - \left(1-\frac{1}{\kappa}\right)\ell(u), \quad u \in [0,\tfrac{\kappa}{2}],$$
and then
\begin{align*}
a'(u) = \begin{cases}\frac12 \cdot \ell(u+1) - \left(1-\frac{1}{\kappa}\right)\cdot \ell'(u)&\text{if } u \in (0,1);\\[.2cm] \frac12 \cdot (\ell(u+1) - \ell(u-1)) - \left(1-\frac{1}{\kappa}\right)\cdot \ell'(u)&\text{if } u \in (1,\tfrac{\kappa}{2}). \end{cases}
\end{align*}
Using  that~$\ell(u) = (\tfrac{1}{\kappa} + u) \cdot \mathds{1}_{[0,\kappa/2)}(u) + (\tfrac{1}{\kappa} + \kappa -u) \cdot \mathds{1}_{[\kappa/2, \kappa]}(u)$,  we obtain
\begin{equation} \label{eq:a_derivative} a'(u) = \begin{cases} \frac{u}{2} + \tfrac{3}{2\kappa} - \tfrac{1}{2}& \text{if } u \in (0,1);\\[.2cm] \frac{1}{\kappa}& \text{if } u \in (1,\tfrac{\kappa}{2} - 1);\\[.2cm] - u+ \tfrac{\kappa}{2}    + \tfrac{1}{\kappa}- 1 &\text{if } u \in (\tfrac{\kappa}{2}-1,\tfrac{\kappa}{2}).\end{cases}\end{equation}
Note that the middle interval $(1,\tfrac{\kappa}{2} - 1)$ is non-empty, since $\kappa>4$.

\smallskip

 Using \eqref{eq:a_derivative} we see that~$a$ is concave on $[1,\kappa/2]$. One can easily compute
$a(1) = \frac{1}{\kappa^2} + \frac{1}{\kappa} > 0$ and $a(\kappa/2) = \frac{1}{\kappa^2} > 0$, so~$a(u) > 0$ for all~$u \in [1,\kappa/2]$.

Also, \eqref{eq:a_derivative} shows that~$a$ is convex on~$[0,1]$; to show that it is positive in this interval, note that~$a'(u) = 0$ is solved there at~$u_\star = 1 -\tfrac{3}{\kappa}$, for which we can compute
$a(u_\star) = \frac{1}{\kappa} - \frac{5}{4\kappa^2}> 0$, since~$\kappa > 5/4$. This completes the proof of Lemma \ref{lem:calculus}.
\end{proof}

\section{Coupling branching random walks and infection path indicators}
\label{section_graphical_construction}

The aim of Section \ref{section_graphical_construction} is to couple the infection path indicators of Definition \ref{Xi_graph_def} with independent branching random walks (BRWs)
and to bound the probabilities of the bad events that can ruin our coupling.
We construct the coupling in Section \ref{subsection_coupling}. The basic idea of the coupling is simple: (i) as long as we do not observe collisions of BRW particles, they coincide with the infection path indicators, (ii) if we observe collisions, we immediately remove all BRW particles with the same ancestor as the one that was involved in the collision, because we don't want them
to cause further collisions. We bound the generating function of the number of BRW family trees that survive these collisions in Section \ref{subsection_annihilating_lemma}  and in Section \ref{subsection_brw_estimates} we bound the terms that appear in our generating function estimate (i.e., we bound collision probabilities). In Section \ref{subsection_brw_facts} we collect some useful facts about branching processes.

\subsection{Coupling}
\label{subsection_coupling}

Let us fix a finite subset $\mathcal{X} \subset \mathbb{Z}^d$ for the rest of Section \ref{section_graphical_construction}.
The goal of Section \ref{subsection_coupling}  is to construct a coupling between the infection path indicators $\Xi_t(x,y)$ (cf.\ \eqref{Xi_graphical_def_eq})  and independent branching random walks $Z_t(x,y)$ (see below), where $x \in \mathcal{X}$ and $y \in \mathbb{Z}^d$. We consider the c.\`a.d.l.\`a.g.\ versions of all the stochastic processes that we define.

\begin{definition}[Labeled branching random walks (BRWs)]
\label{def_brw_from_X}
 At time zero, for each~$x \in \mathcal{X}$, there is one particle with label~$x$ at location~$x$.
 Independently, particles die with rate one and give birth to new particles with rate~$\lambda$; a newborn particle has the same label as the parent, and is placed at a site chosen uniformly at random in the translate of $B(R)$ centered at the position of the parent.

  For~$x \in \mathcal{X}$, $y \in \Z^d$ and $t \geq 0$ we let~$Z_t(x,y)$ be the number of particles with label~$x$ located at site~$y$ at time $t$.
\end{definition}

The label of a particle at time $t$ encodes the location of its ancestor at time $0$.
We have $Z_0(x,y)=\mathds{1}[ x=y \in \mathcal{X} ]$. For each $x \in \mathcal{X}$, the particles with label $x$ form a continuous-time BRW on~$\Z^d$ with one single ancestor at time zero at location $x$. The branching random walks with different labels evolve independently from each other.

\begin{claim}[Labeled branching processes]\label{claim_brw_bp} Denote by
\begin{equation}\label{Z_T_x}
Z_t(x):= \sum_{y \in \mathbb{Z}^d } Z_t(x,y), \quad t \geq 0, \quad x \in \mathcal{X}
\end{equation}
the total number of BRW particles with label $x$ at time $t$.
 For~$x\in \mathcal{X}$, the process~$(Z_t(x))_{t \ge 0}$ is a continuous-time branching process in which individuals die with rate one and give birth with rate~$\lambda$.  The branching processes with different labels evolve independently from each other.
\end{claim}

Next we define the following stopping times:
\begin{align}
\label{crowd_x}
\tau_x  &:= \min \left\{ \, t \, : \, \sum\nolimits_{y \in \mathbb{Z}^d }   Z_t(x,y)(Z_t(x,y)-1)   >0   \,  \right\}, \quad x \in \mathcal{X} \\
\label{crowd_x_x_prime}
\tau_{x,x'}  &:= \min \left\{ \, t \, : \, \sum\nolimits_{y \in \mathbb{Z}^d }  Z_t(x,y)Z_t(x',y)  >0   \,  \right\}, \quad x \neq x' \in \mathcal{X}
\end{align}
Thus $\tau_x$ is the first time when there are two particles with label $x$ at the same site, while $\tau_{x,x'}$ is the first time when a particle with label
$x$ and a particle with label $x'$ are at the same site. Both of these stopping times correspond to \emph{collisions}, and we will define \emph{annihilating BRWs}
by postulating that if a particle with label $x$ is involved in a collision then we immediately remove (annihilate) all particles with label $x$.
 We will denote by $\mathcal{X}_t$ the set of labels which did not yet get annihilated by time $t$.

\begin{definition}[Set of alive/annihilated labels]
\label{def_mathcal_X_t} For each $t \geq 0$  we define $\mathcal{X}_t \subseteq \mathcal{X}$ in the following way: let $\mathcal{X}_0=\mathcal{X}$. Then we increase $t$, and if $t$ reaches $\tau_x$ for some $x \in \mathcal{X}_t$ then we  remove $x$ from $\mathcal{X}_t$. Similarly, if $t$ reaches $\tau_{x,x'}$ where $x$ and $x'$ are both still in $\mathcal{X}_t$, then we remove both $x$ and $x'$ from $\mathcal{X}_t$. We say that $\mathcal{X}_t$ is the set of alive labels at time $t$. If $x \in \mathcal{X} \setminus \mathcal{X}_t$, we say that the label $x$ got annihilated by time $t$.
\end{definition}
See Figure~\ref{fig:coupling} for an illustration of this definition.

\begin{figure}[htb]
\begin{center}
\setlength\fboxsep{0pt}
\setlength\fboxrule{0pt}
\fbox{\includegraphics[width = .9\textwidth]{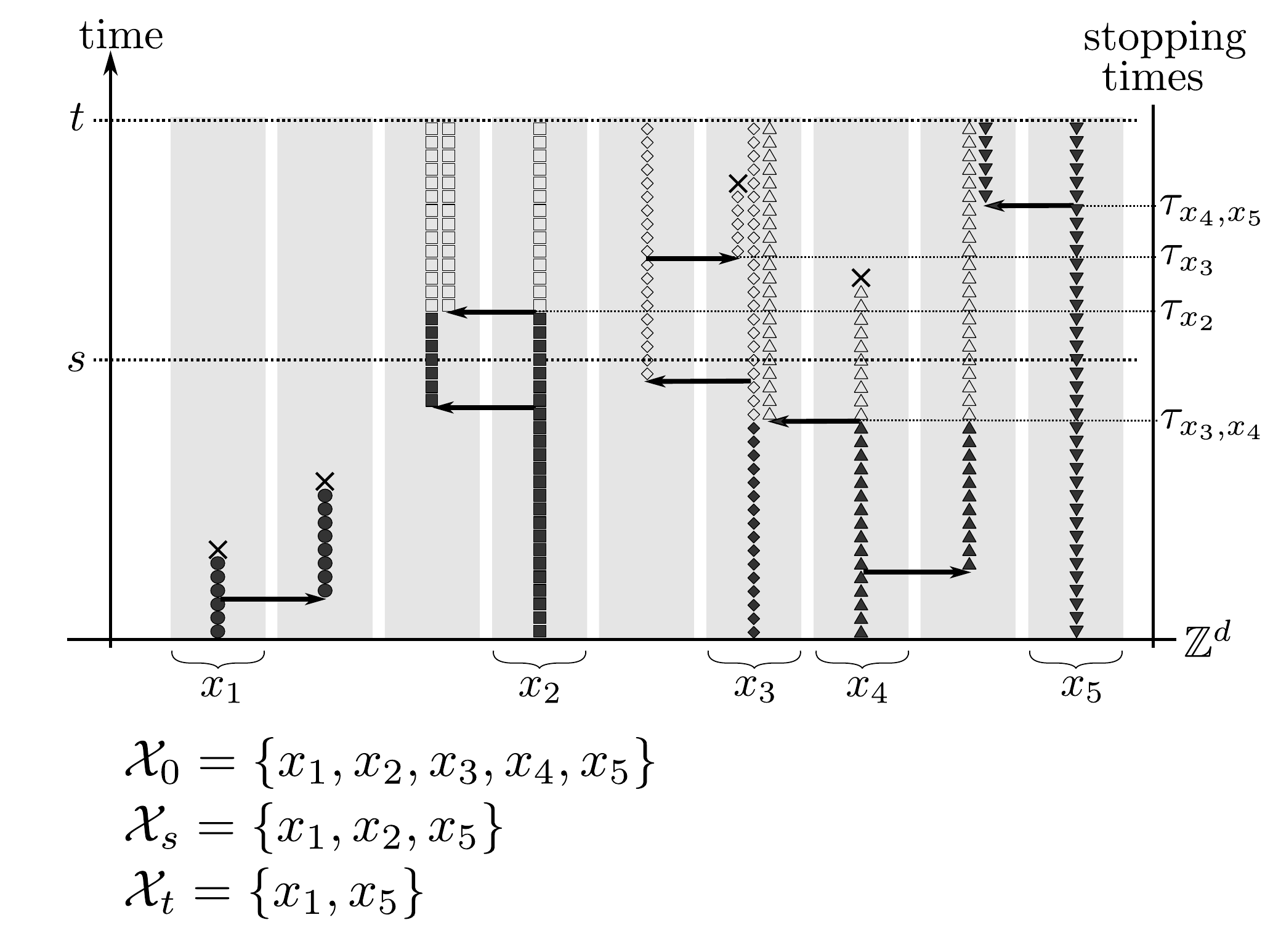}}
\end{center}
	\caption{\label{fig:coupling} Illustration of the coupling of infection paths and branching random walks, the stopping times~$\tau_x$ and~$\tau_{x,x'}$ and the label process~$(\mathcal{X}_t)_{t \ge 0}$ for a case where~$\mathcal{X}$ consists of five points,~$x_1,x_2,x_3,x_4,x_5$. The symbols~$\circ, \square, \diamond, \Delta, \nabla$ represent particles with labels~$x_1,x_2,x_3,x_4,x_5$, respectively.
 The dark (resp.\ light) color represents particles  whose label is alive (resp.\ annihilated).
 (Although the figure might give the impression that the dynamics is in discrete time, time is of course continuous). The contents of the sets~$\mathcal{X}_s$ and~$\mathcal{X}_t$ are made explicit for two times~$s < t$. Note  that at time~$\tau_{x_4,x_5}$, the label~$x_5$ remains alive, since it collides with a particle with label~$x_4$, which has already been annihilated.}
\end{figure}

Now we can state the key lemma of Section \ref{subsection_coupling} relating the infection path indicators $\Xi_t(x,y)$ (cf.\ Definition \ref{Xi_graph_def}), the independent labeled BRWs $Z_t(x,y)$ (cf.\ Definition \ref{def_brw_from_X}) and the set of alive labels $\mathcal{X}_t$ defined above.
\begin{lemma}[Coupling]  \label{lemma_coupling}
 There is a way to couple $(\Xi_t(x,y))_{x \in \mathcal{X},y \in \mathbb{Z}^d, t \geq 0}$ and
$(Z_t(x,y))_{x \in \mathcal{X},y \in \mathbb{Z}^d, t \geq 0}$ so that
\begin{equation}\label{coupling_of_Xi_Z_on_X_t_eq}
 \text{for any $x \in \mathcal{X}_{t}$ and $y \in \mathbb{Z}^d$ we have $Z_t(x,y)=\Xi_t(x,y)$.}
\end{equation}
\end{lemma}
Before we prove Lemma \ref{lemma_coupling}, let us make a couple of remarks.
\begin{remark}\label{remark_coupling_pitfalls}
 Our naive goal is to sample a configuration  $(\xi(x))_{x \in \mathcal{X}}$ distributed according to $\mu_{\lambda,R}$ using \eqref{upper_invariant_active_path}, but
 we want to use our independent labeled BRWs instead of the infection path indicators, since we want to argue that $(\xi(x))_{x \in \mathcal{X}}$ is close to being an i.i.d.\ Bernoulli
 configuration.
  Therefore,  we want to control the size of the set $\mathcal{X} \setminus \mathcal{X}_t$ of labels where \eqref{coupling_of_Xi_Z_on_X_t_eq} fails. Also, we want $\mathcal{X}$ to have a large cardinality, since percolation connectivity events involve the simultaneous status of many sites. Also, naively we would like to let $t \to \infty$, cf.\ \eqref{upper_invariant_active_path}.
    Nevertheless, if $\mathcal{X}$ is too ``fat'' then we see many collisions in the beginning, while making $t$ too big results in many collisions in the long run (since
    our BRWs are supercritical). In the end, it will be sufficient to choose a set $\mathcal{X}$ with a ``thin'' fractal structure (see Section \ref{section_renorm_intro}) and work with a large but fixed time $t$ (see Section \ref{section_proof_of_main}), while we are
    free to choose the range $R$ of the BRW jump distribution to be very large, which clearly helps if we want to reduce the number of collisions.
\end{remark}

\begin{proof}[Proof of Lemma \ref{lemma_coupling}] We begin by providing a suitable graphical construction of the independent labeled branching random walks $(Z_t(x,y))_{x \in \mathcal{X},y \in \mathbb{Z}^d, t \geq 0}$.

For each $y \in \mathbb{Z}^d$ and $\ell \in \mathbb{N}$, let $N^{\ell}_y$
denote a Poisson process of rate $1$ on $[0,\infty)$,  and for each $y,z \in \mathbb{Z}^d$ satisfying $z \in B(y,R)$ and each $\ell \in \mathbb{N}$, let $N^{\ell}_{y,z}$
denote a Poisson process of rate $\lambda/B(R)$ on $[0,\infty)$.
 All Poisson processes are independent.

\smallskip

  An arrival time of $N^{\ell}_y$ at time $t$ will be interpreted as a \emph{death mark} which kills the $\ell$'th particle at location $y$ at time $t$.  An arrival time of $N^{\ell}_{y,z}$ at time $t$ will be interpreted as a \emph{reproduction arrow}: assuming that the label of the $\ell$'th particle at location $y$ at time $t$ is $x$, this reproduction arrow will trigger this particle to produce a newborn particle with label $x$ at location $z$.

\smallskip

  To be more specific, for each~$x \in \mathcal{X}$, we start with  one particle with label~$x$ at location~$x$ at time zero. Assume that we have already constructed the evolution of the labeled BRW particle  configuration as well as the set of alive labels up to time $t$ in terms of these Poisson point processes. For~$x \in \mathcal{X}$, $y \in \Z^d$ we let~$Z_t(x,y)$ be the number of particles with label~$x$ located at site~$y$ at time $t$.
   Note that by Definition \ref{def_mathcal_X_t}, there is at most one particle with label belonging to $\mathcal{X}_t$ at site $y$  at time $t$ (since two such particles would
   have already annihilated each other by time $t$).
      Altogether, there are $\sum_{x \in \mathcal{X}} Z_t(x,y)$ particles sharing the same location $y$ at time $t$.
   We would like to talk about the $\ell$'th particle at location $y$ at time $t$ (where $1 \leq \ell \leq \sum_{x \in \mathcal{X}} Z_t(x,y)$), so let us order these particles with one rule in mind: if there is a particle at time $t$ at location $y$ with label belonging to $\mathcal{X}_t$, this particle has to be the first one in the ordering, so that it gets to use the Poisson processes with upper index $\ell=1$.
  Now each particle waits for the next instruction encoded by its own Poisson point processes. As soon as any of the particles receives the next instruction, we update the particle configuration, the set of alive vertices and the ordering of particles at each location accordingly. Let us note here that we assumed $|\mathcal{X}|<\infty$, and therefore the total number of particles remains finite after each such step, thus there is always a well-defined
``next instruction'', allowing us to proceed inductively.

\smallskip

It is easy to see that this construction indeed produces independent labeled BRWs  $(Z_t(x,y))_{x \in \mathcal{X},y \in \mathbb{Z}^d, t \geq 0}$ satisfying Definition \ref{def_brw_from_X}, which
also shows that the total number of particles does not explode in finite time.

\smallskip

 In order to define infection path indicators $(\Xi_t(x,y))_{x \in \mathcal{X},y \in \mathbb{Z}^d, t \geq 0}$ satisfying Definition \ref{Xi_graph_def} on the same probability space, we perform the graphical
 construction described in Section \ref{subsection_graphical_constr_of_cont_proc}, using the Poisson point processes with upper index $\ell=1$, more specifically, using  $N^{1}_y$, $y \in \mathbb{Z}^d$ as recovery marks and
 $N^{1}_{y,z}$ with $y,z \in \mathbb{Z}^d, z \in B(y,R)$ as infection arrows.

 \smallskip

 Now it is easy to prove \eqref{coupling_of_Xi_Z_on_X_t_eq} inductively: one just has to observe that the time evolution of $\Xi_t(x,y)$ (cf.\ \eqref{Xi_graphical_def_eq})
 agrees with that of $Z_t(x,y)$ as long as $x \in \mathcal{X}_t$ (since both graphical constructions use the same Poisson processes with upper index $\ell=1$ as an input, and they use this input to perform the same updates).
\end{proof}

Having proved Lemma \ref{lemma_coupling}, we see that the coupling \eqref{coupling_of_Xi_Z_on_X_t_eq} will only be useful if we can control the size of $\mathcal{X} \setminus \mathcal{X}_t$.
This is what we do in Section \ref{subsection_annihilating_lemma}.

\subsection{Generating function bounds}
\label{subsection_annihilating_lemma}

Variants of our next lemma have already appeared in our earlier work, cf.\ \cite[(5.14)]{RV15} and \cite[Lemma 2.4]{RV17}, however
those earlier lemmas pertained to annihilating random walks, while Lemma \ref{lemma_annihilation} below pertains to annihilating BRWs,
where the situation is a bit more complicated, since different types of annihilation events correspond to the collision times $\tau_x$ (cf.\ \eqref{crowd_x}), where only one label is involved,
and $\tau_{x,x'}$ (cf.\ \eqref{crowd_x_x_prime}), where a pair of distinct labels are involved.

Our next lemma bounds the generating function of the number of labels from $\mathcal{X}$ that got annihilated by time $t$ (cf.\ Definition \ref{def_mathcal_X_t}).
Let $\prec$ denote a total ordering of $\Z^d$.

\begin{lemma}[Generating function bounds]
\label{lemma_annihilation}
 For any $s \in [0,+\infty )$ we have
\begin{equation}\label{annihilation_momgen_ineq}
\mathbb{E}\left( s^{|\mathcal{X} \setminus \mathcal{X}_t|} \right)  \leq   \prod_{ x \in \mathcal{X} } \left( 1+ \mathbb{P}(\tau_{x} \leq t )s \right)  \prod_{x,y \in \mathcal{X} \, : \,  x \prec y} \left( 1+ \mathbb{P}(\tau_{x,y} \leq t ) s^2 \right).
\end{equation}
\end{lemma}
\begin{proof}  If  $e=\{ x,y \} \in \binom{\mathcal{X}}{2}$, let us denote $\tau_e:=\tau_{x,y}$ and
\begin{align}
 \eta_t( e )&:=\mathds{1}\left[ \, \tau_{e} < t, \;  \mathcal{X}_{(\tau_{e})_+ }=  \mathcal{X}_{(\tau_{e})_-}\setminus e    \,  \right], \quad e \in \binom{\mathcal{X}}{2},  \\
\eta_t(x)&:=\mathds{1}\left[ \, \tau_{x} < t, \;  \mathcal{X}_{(\tau_{x})_+ }=  \mathcal{X}_{(\tau_{x})_-}\setminus \{ x \}   \,  \right], \quad x \in \mathcal{X}.
\end{align}
In words: $\eta_t(e)$ (where $e=\{x,y\}$) is the indicator of the event that the labels $x$ and $y$ annihilated each other before time $t$,
while $\eta_t(x)$  is the indicator of the event that the collision of two particles with label $x$ annihilated the label $x$ before time $t$.

 Denote by $\mathcal{A}$ the set of edges $e \in \binom{\mathcal{X}}{2}$ for which $\eta_t(e)=1$ and denote by $\mathcal{B}$ the set of vertices $x \in \mathcal{X}$ for which
$\eta_t(x)=1$.

\smallskip

 Let us note that as soon as a label gets annihilated, it is gone forever, i.e.,  any label
 can be annihilated at most once, therefore the pair $(\mathcal{A}, \mathcal{B})$ is a random element of $H$, where
 $H$ denotes the set of pairs $(A,B)$ which satisfy (i) $A \subseteq \binom{\mathcal{X}}{2}$, (ii) $B \subseteq \mathcal{X}$, (iii) $A$ is a partial matching (i.e., $e \neq e' \in A$ implies $e \cap e' = \emptyset$) and (iv) $B$ is disjoint from
$\cup_{e \in A} e$.

Using the notation introduced above, we obtain
\begin{multline}
\mathbb{E}\left( s^{|\mathcal{X} \setminus \mathcal{X}_t|} \right)=\mathbb{E}\left(s^{2|\mathcal{A}|+|\mathcal{B}|} \right)  = \sum_{(A,B)\in H} \mathbb{P}(\mathcal{A}=A, \, \mathcal{B}=B)  s^{2|A|} s^{|B|} \leq \\
\sum_{(A,B)\in H} \mathbb{P}(\, \forall e \in A \, : \, \tau_e<t, \,\, \forall x \in B \, : \, \tau_x <t \, )  s^{2|A|} s^{|B|}\stackrel{(*)}{=}\\
\sum_{(A,B)\in H} \prod_{e \in A} \mathbb{P}( \tau_e<t) \prod_{x \in B} \mathbb{P}( \tau_x <t )  s^{2|A|} s^{|B|} \leq \\
\sum_{ A \subseteq \binom{\mathcal{X}}{2} } \sum_{B \subseteq \mathcal{X} } \prod_{e \in A} \mathbb{P}( \tau_e<t) \prod_{x \in B} \mathbb{P}( \tau_x <t )  s^{2|A|} s^{|B|}=\\
 \prod_{e \in \binom{\mathcal{X}}{2}} \left( 1+ \mathbb{P}(\tau_{e} < t ) s^2 \right)  \prod_{ x \in \mathcal{X} } \left( 1+ \mathbb{P}(\tau_{x} < t )s \right),
\end{multline}
where in $(*)$ we used criteria (iii) and (iv) of the definition of $H$ as well as the fact that branching random walks with different labels are independent, cf.\
Definition \ref{def_brw_from_X}. The proof of Lemma \ref{lemma_annihilation} is complete.
\end{proof}

In order to bound the right-hand side of \eqref{annihilation_momgen_ineq}, we need to bound  probabilities of form $\mathbb{P}(\tau_{x} \leq t )$ and $\mathbb{P}(\tau_{x,y} \leq t )$. This is what we will do in Section \ref{subsection_brw_estimates}.

\subsection{BRW collision probability bounds}
\label{subsection_brw_estimates}

The goal of Section \ref{subsection_brw_estimates} is to prove the following bounds on the probability that BRW particles with the same label (cf.\ \eqref{self_collision_bound_eq}) or different labels (cf.\ \eqref{pair_collision_bound_eq}) collide. We assume that birth rate of the BRW satisfies $\lambda>1$.

\begin{lemma}[Collision probability bounds] \label{lemma_crowd_bounds}
Recall the definitions of $\tau_{x}$ and $\tau_{x,y}$ from \eqref{crowd_x} and \eqref{crowd_x_x_prime}.
\begin{enumerate}[(i)]
\item \label{self_collision_bound}
For any $\lambda>1, T_0 \in \mathbb{R}_+$  we have
\begin{equation}\label{self_collision_bound_eq}
 \lim_{R \to \infty} \mathbb{P}(\tau_{0} \leq T_0 ) =0.
\end{equation}
\item \label{pair_collision_bound}
For any  $\lambda>1, T_0 \in \mathbb{R}_+$ there exist $C'=C'(\lambda,T_0)<+\infty$ and $C''=C''(\lambda,T_0)<+\infty$ such that
\begin{equation}\label{pair_collision_bound_eq}
 \mathbb{P}(\tau_{x,y} \leq T_0 ) \leq   \frac{C'}{R^{d}}  \exp\left(  -  \frac{|x-y|}{C'' R } \right), \quad x \neq y \in \mathbb{Z}^d, \quad R \in \mathbb{N}.
\end{equation}
\end{enumerate}
\end{lemma}

The proofs are standard, but we include them for completeness. We encourage the reader to skip the rest of Section \ref{subsection_brw_estimates} at first reading.

\begin{proof}[Proof of Lemma \ref{lemma_crowd_bounds}\eqref{self_collision_bound}] The argument is quite simple, but for brevity and to avoid introducing too much notation, we only sketch the main idea. Recall from \eqref{crowd_x} that $\tau_0$
denotes the first time when two particles with label $0$ collide. Note that these particles are all descendants of the particle that sits at site $0$ at time $0$.
 Let us denote by $N$ the total number of BRW particles
with label $0$ that are born before $T_0$. We have $\mathbb{E}(N^2)=C(T_0,\lambda)<+\infty$ by Claim \ref{claim_brw_bp}. The location of a newborn particle
is uniformly distributed in the translate of $B(R)$ centered at its parent,
thus the probability that it collides with a fixed particle that is already present is at most $1/|B(R)|$. The number of such potential collisions
up to time $T_0$ is bounded by $ \frac{1}{2} N (N-1)$, thus we obtain that the expected number of collisions of particles with label $0$ up to time $T_0$
is at most $\mathbb{E}( \frac{1}{2} N(N-1))/|B(R)|$, and this quantity goes to zero as $R \to \infty$, which implies \eqref{self_collision_bound_eq} by Markov's inequality.
\end{proof}

Before we prove Lemma \ref{lemma_crowd_bounds}\eqref{pair_collision_bound}, let us introduce some useful notation.

\begin{definition}[$R$-spread-out random walk  with jump rate $\lambda$ on $\mathbb{Z}^d$]\label{def_R_spread_out_rw}
 Denote by
 $(X_t)_{t\geq 0}$ the  continuous-time random walk on $\mathbb{Z}^d$ where the holding times between jumps
 are i.i.d.\ with $\mathrm{EXP}(\lambda)$ distribution, and the increments of the walk are i.i.d.\ with uniform distribution on $B(R)$.
  Denote by $p^{\lambda,R}_{t}(x,y)$ the transition kernel of $(X_t)_{t\geq 0}$, i.e.,
  \begin{equation}\label{cont_time_tr_k_def}
  p^{\lambda,R}_{t}(x,y)=\mathbb{P}(X_t=y \, | \, X_0=x ), \quad t \geq 0, \; x,y \in \mathbb{Z}^d.
  \end{equation}
 \end{definition}
Recalling the notion of  $Z_t(x,y)$ from Definition \ref{def_brw_from_X}, we note that by \cite[Proposition 1.21]{Li99} we have
 \begin{equation}\label{rw_brw_expect_correspond}
   \mathbb{E}(Z_t(x,y))= e^{(\lambda-1)t} p^{\lambda,R}_{t}(x,y),  \quad t \geq 0, \; x,y \in \mathbb{Z}^d.
 \end{equation}
The proof of Lemma \ref{lemma_crowd_bounds}\eqref{pair_collision_bound} will follow from the next bound.

\begin{lemma}[Heat kernel bound for $(X_t)$]
\label{lemma_heat_kernel_bound_cont}
 For any  $\lambda>1, T_* \in \mathbb{R}_+$ there exist $C^*=C^*(\lambda,T_*)<+\infty$ and $C^{\sstar}=C^{\sstar}(\lambda,T_*)<+\infty$ such that
\begin{equation}\label{heat_kernel_bound_eq_cont}
  p^{\lambda,R}_{t}(x,y) \leq   \frac{C^*}{R^{d}}  \exp\left(  -  \frac{|x-y|}{C^{\sstar} R } \right), \; \; 0 \leq t \leq T_*, \; \;   x \neq y  \in \mathbb{Z}^d, \; \; R \in \mathbb{N}.
\end{equation}
\end{lemma}

Before we prove Lemma \ref{lemma_heat_kernel_bound_cont}, let us deduce Lemma \ref{lemma_crowd_bounds}\eqref{pair_collision_bound} from it.

\begin{proof}[Proof of Lemma \ref{lemma_crowd_bounds}\eqref{pair_collision_bound}]
Let $T_1:=T_0+1$, $L:= \sum_{z \in \mathbb{Z}^d} \int_0^{T_1} Z_s(x,z)Z_s(y,z) \, \mathrm{d}s$. First we note that
$\mathbb{E}(L \, | \, \tau_{x,y} \leq T_0 ) \geq \int_0^1 e^{-2s} \, \mathrm{d}s \geq \frac{1}{4}$ because at time $\tau_{x,y}$ a particle
with label $x$ and a particle with label $y$ share the same location, and the subsequent lifetimes of these particles  are
i.i.d.\ with $\mathrm{EXP}(1)$ distribution. From this we obtain
$\mathbb{P}(\tau_{x,y} \leq T_0 )\leq 4 \mathbb{E}(L)$, thus it is enough to bound
\begin{multline}\label{brw_L_expect_upper_bound}
  \mathbb{E}(L)=\sum_{z \in \mathbb{Z}^d} \int_0^{T_1} \mathbb{E}(Z_s(x,z))\mathbb{E}(Z_s(y,z)) \, \mathrm{d}s \stackrel{ \lambda>1, \eqref{rw_brw_expect_correspond} }{\leq} \\
  e^{2(\lambda-1)T_1} \sum_{z \in \mathbb{Z}^d} \int_0^{T_1}  p^{\lambda,R}_{s}(x,z)  p^{\lambda,R}_{s}(y,z)   \, \mathrm{d}s\stackrel{(*)}{=}
  e^{2(\lambda-1)T_1} \int_0^{T_1}  p^{\lambda,R}_{2s}(x,y)   \, \mathrm{d}s,
\end{multline}
where  $(*)$ is Chapman-Kolmogorov for $(X_t)_{t\geq 0}$. The desired bound \eqref{pair_collision_bound_eq} follows from \eqref{heat_kernel_bound_eq_cont} (with $T_*=2T_1$) and \eqref{brw_L_expect_upper_bound}.
\end{proof}

Before we prove Lemma \ref{lemma_heat_kernel_bound_cont}, we introduce some further notation.

\begin{definition}\label{def_R_spread_out_discrete_rw}
 Denote by $(Y_n)_{n \in \mathbb{N}}$ the discrete-time $R$-spread-out random walk on $\mathbb{Z}^d$: the increments of $(Y_n)$ are i.i.d.\ with
 uniform distribution on~$B(R)$.
 Denote by $q_{n}^{R}(x,y)$ the transition probabilities of $(Y_n)$, i.e.,
  \begin{equation}\label{discr_time_tr_kernel}
  q_{n}^{R}(x,y)=\mathbb{P}(Y_{n}=y \, | \, Y_0=x ), \quad n \in \mathbb{N}_0, \; x,y \in \mathbb{Z}^d.
  \end{equation}
 \end{definition}

The proof of Lemma \ref{lemma_heat_kernel_bound_cont} will follow from the next bound. Note that this bound is very crude, but it will suffice.

\begin{lemma}[Heat kernel bound for $(Y_n)$]\label{lemma_transition_prob_bound}
  There exist $C_\diamond<+\infty$ and $C_{\ddiamond}<+\infty$ such that
\begin{equation}\label{heat_kernel_bound_eq_discrete}
  q_{n}^{R}(x,y) \leq   \frac{C_{\diamond}^n  }{R^{d}}  \exp\left(  -  \frac{|x-y|^2}{C_{\ddiamond}  n R^2 } \right), \; \; n,R \in \mathbb{N}, \; \;   x \neq y  \in \mathbb{Z}^d.
\end{equation}
\end{lemma}

Before we prove Lemma \ref{lemma_transition_prob_bound}, let us deduce Lemma \ref{lemma_heat_kernel_bound_cont} from it.

\begin{proof}[Proof of Lemma \ref{lemma_heat_kernel_bound_cont}] Let $0 \leq t \leq T_*$ and $x\neq y \in \mathbb{Z}^d$.
\begin{multline*}
 p^{\lambda,R}_{t}(x,y)\stackrel{ (*) }{=} \sum_{n=1}^{\infty} \frac{  e^{-\lambda t} (\lambda t)^n}{n!} q_{n}^{R}(x,y)
 \stackrel{ \eqref{heat_kernel_bound_eq_discrete} }{\leq}
 \sum_{n=1}^{\infty}  \frac{(\lambda T_*)^n}{n!} \frac{C_{\diamond}^n  }{R^{d}}  \exp\left(   \frac{-|x-y|^2}{C_{\ddiamond}  n R^2 } \right)
 \stackrel{(**)}{\leq} \\
  \sum_{n=0}^{\infty}  \frac{(\lambda T_* C_{\diamond} )^n}{n! R^d }   \exp\left(   \frac{ - 2 |x-y|}{ \sqrt{C_{\ddiamond}} R } + n \right) =
    \exp\left(   \frac{- 2 |x-y|}{ \sqrt{C_{\ddiamond}} R }  \right) \frac{ \exp\left( \lambda T_* C_{\diamond} e  \right)  }{R^d} ,
  \end{multline*}
  where in $(*)$ we used that  the continuous-time walk $(X_t)$ performs $\mathrm{POI}(\lambda t)$ steps in the time interval $[0,t]$
  (cf.\ Definitions \ref{def_R_spread_out_rw}, \ref{def_R_spread_out_discrete_rw}), and in $(**)$ we used the inequality
  $-b^2/n \leq -2b+n$ with $b=\frac{|x-y|}{ \sqrt{C_{\ddiamond} } R }$.
\end{proof}

\begin{proof}[Proof of Lemma \ref{lemma_transition_prob_bound}] Let us first observe that the coordinates of the $d$-dimensional random walk $(Y_n)$ evolve independently, therefore we have
\begin{equation}\label{rw_coordinates_evolve_indepenetly}
   q_{n}^{R}(x,y)=\prod_{i=1}^d \tilde{q}_n^R(x_i,y_i), \qquad x=(x_1,\dots,x_d), \quad y=(y_1,\dots,y_d),
\end{equation}
where $\tilde{q}_n^R(x_i,y_i)$ denote the transition probabilities of the $R$-spread-out random walk on $\mathbb{Z}$. From \eqref{rw_coordinates_evolve_indepenetly}
it follows that it is enough to prove \eqref{heat_kernel_bound_eq_discrete} when $d=1$. In the $d=1, R=1$ case (lazy nearest-neighbour random walk on $\mathbb{Z}$) the bounds
\begin{equation}\label{lazy_heat_kernel}
\frac{C'''}{\sqrt{n}} \mathds{1}[\, |y| \leq \sqrt{n}\,]   \leq  \tilde{q}^1_n(0,y) \leq \frac{C'}{\sqrt{n}} \exp\left( -\frac{y^2}{ C'' n } \right), \quad y \in \mathbb{Z}, \; n \in \mathbb{N}
\end{equation}
are classical (cf.\ \cite[Section 2]{LL10}). Using the lower bound of \eqref{lazy_heat_kernel} we obtain $\tilde{q}_1^R(0,y) \leq C \tilde{q}^1_{R^2}(0,y), \, y \in \mathbb{Z}$ for some $C<+\infty$, and taking the $n$-fold convolution of both sides of this inequality we obtain
\begin{equation}\label{upper_bound_on_R_spread_out_proof}
  \tilde{q}_n^R(0,y) \leq C^n \tilde{q}^1_{n R^2}(0,y) \stackrel{\eqref{lazy_heat_kernel}}{\leq} C^n \frac{C'}{\sqrt{n R^2 }  } \exp\left( -\frac{y^2}{ C'' n R^2 } \right),
  \quad y \in \mathbb{Z},
\end{equation}
from which the $d=1$ case of  \eqref{heat_kernel_bound_eq_discrete} follows.
\end{proof}

\subsection{Branching process facts}
\label{subsection_brw_facts}

Let us collect some formulas that will be relevant to us because of Claim \ref{claim_brw_bp}.

\begin{definition}\label{def_bp_pop_at_t}
Let $Z_t$ denote the population size at time $t$ of a branching process with birth rate $\lambda$ and death rate $1$, starting from $Z_0=1$.
\end{definition}

Given $\lambda>1$ we introduce the survival probabilities
\begin{align}
\label{sigma_lamb_T}
  \sigma(\lambda,t):=  \mathbb{P}(\, Z_t >0 \, )&=\frac{\lambda-1}{\lambda -e^{(1-\lambda)t } }, \; t \geq 0, \\
  \label{sigma_lambda}
   \sigma(\lambda):=\mathbb{P}(\, \forall \, t \geq 0 \, : \, Z_t >0 \, )&=1-1/\lambda.
\end{align}

Also note that for any $\lambda>1$ and $s \in [0,1]$ we have
\begin{equation}\label{Z_t_gen_fn}
  \mathbb{E}\left( s^{Z_t} \, | \, Z_t >0  \right)=  \frac{ \tilde{q}(\lambda,t)  s}{1-(1-\tilde{q}(\lambda,t))s}, \quad \tilde{q}(\lambda,t):= \frac{\lambda-1}{\lambda e^{(\lambda-1)t} -1 }.
\end{equation}

\begin{remark}
 In order to prove \eqref{sigma_lamb_T} and \eqref{Z_t_gen_fn}, it is enough to note that if one defines
$G(t,s)=\mathbb{E}(s^{Z_t})$ then $G(t,s)$ solves the PDE
\begin{equation}
  \partial_t G(t,s)=(\lambda s^2 -(\lambda+1)s+1) \partial_s G(t,s), \qquad G(0,s)=s,
\end{equation}
and that
\begin{equation}
  G(t,s)=  (1-\sigma(\lambda,t)) + \sigma(\lambda,t) \frac{ \tilde{q}(\lambda,t)  s}{1-(1-\tilde{q}(\lambda,t))s}
\end{equation}
also solves this PDE.
\end{remark}

\section{Renormalization}
\label{section_renorm_intro}

We will use the multi-scale renormalization scheme as in \cite{Ra15}, \cite{RV15} and \cite{RV17}, which in turn is a variant of the one introduced in \cite{sznitman_decoupling}. The idea is that if we see a percolation crossing event in a large annulus then this implies that many small and sparsely-located annuli are also crossed.

\smallskip

The renormalization involves the embedding of dyadic trees into $\mathbb{Z}^d$. For $n \in \mathbb{N}$, let us denote by $T_{(n)}=\{1,2\}^n$ the set of binary strings of length $n$
 (in particular, $T_{(0)}=\emptyset$).
  Denote by
\begin{equation}
T_n =\bigcup_{k=0}^n T_{(k)}
\end{equation}
the dyadic tree of depth $n$.
For $0 \leq k < n$ and $m \in T_{(k)}$, $m=(\xi_1,\dots,\xi_k)$,  we denote
by $m_1$ and $m_2$ the two children of $m$ in $T_{(k+1)}$:
\begin{equation}\label{def_eq_m1_m2}
 m_i=(\xi_1,\dots,\xi_k,i), \qquad i \in \{ 1, 2 \}.
 \end{equation}
Denote by $L \in \mathbb{N}$ the bottom scale of our renormalization.  We define the $n$'th scale by
\begin{equation} \label{def:scalesLn}
L_n := L \cdot 6^n , \quad n \in \mathbb{N}_0.
\end{equation}

\begin{definition}\label{def_proper_embedding_of_trees}
 The mapping $\mathcal{T}: T_n \to \mathbb{Z}^d$ is a \emph{proper embedding} of the tree $T_n$  with
root at the origin of $\mathbb{Z}^d$ if
\begin{enumerate}
 \item The root is mapped to $0$: $\mathcal{T}(\emptyset)=0$;
\item For any $0 \leq k \leq n$ and any $m \in T_{(k)}$ we have $\mathcal{T}(m) \in L_{n-k} \mathbb{Z}^d$;
\item For any $0 \leq k < n$ and any $m \in T_{(k)}$ the embedding of the children $m_1$ and $m_2$ of $m$ satisfy
\begin{equation}\label{tree_children_spread_out}
|\mathcal{T}(m_1) -\mathcal{T}(m)|= L_{n-k}, \qquad |\mathcal{T}(m_2) -\mathcal{T}(m)|= 2 L_{n-k}.
\end{equation}
\end{enumerate}
Let us denote by $\Lambda_{n,L}$ the set of proper embeddings of the tree $T_n$ with bottom scale $L$ into $\mathbb{Z}^d$ with root at $0$.
\end{definition}

\begin{remark}\label{remark_fractal} Let us denote $\mathscr{X}:=\cup_{m \in T_{(n)} } \{ \mathcal{T}(m) \}$. We have $|\mathscr{X}|= 2^n$ and $\mathrm{diam}(\mathscr{X}) \asymp 6^n$.
  Heuristically, for large $n$,  the set $ \mathscr{X}^*:= 6^{-n}\mathscr{X}$ is a fractal  with
dimension $\frac{\ln(2)}{\ln(6)}$, because if we blow $ \mathscr{X}^*$  up by a factor of six, we see two sets which are very similar to the set $\mathscr{X}^*$ that we started with.
Also, $|B(x,R)\cap \mathscr{X}| \asymp (\mathrm{diam}(\mathscr{X}) \wedge R)^{\frac{\ln(2)}{\ln(6)}}$ holds for any $x \in \mathscr{X}$.
The number $\frac{\ln(2)}{\ln(6)}$ will show up twice in later calculations, see the end of the proof of Lemma \ref{lemma_X_t_big} and the end of the proof
of Lemma \ref{lemma_ground_is_mostly_fertile}.
\end{remark}

We now recall three lemmas about proper embeddings from \cite{Ra15}. Occasionally we will slightly modify the formulation compared to \cite{Ra15} in order to fit our current purposes.
 The first lemma bounds the number of proper embeddings.

\begin{lemma}[Lemma 3.2 of \cite{Ra15}]
\label{lem:renorm_count} There exists a dimension-dependent constant $C_d > 0$  such that for every $n,L \in \mathbb{N}$ we have
\begin{equation}
\label{eq:renorm_count}
|\Lambda_{n,L}| \leq (C_d)^{2^n}, \qquad n \in \mathbb{N}.
\end{equation}
\end{lemma}

The second lemma relates the notion of proper embeddings to that of crossing events (see Figure \ref{fig:renorm} for an illustration). Recall the notion of $*$-connected paths
from Definition \ref{def_paths} and the notion of spheres from \eqref{sphere_def}.

\begin{figure}[htb]
\begin{center}
\setlength\fboxsep{0pt}
\setlength\fboxrule{0pt}
\fbox{\includegraphics[width = 1\textwidth]{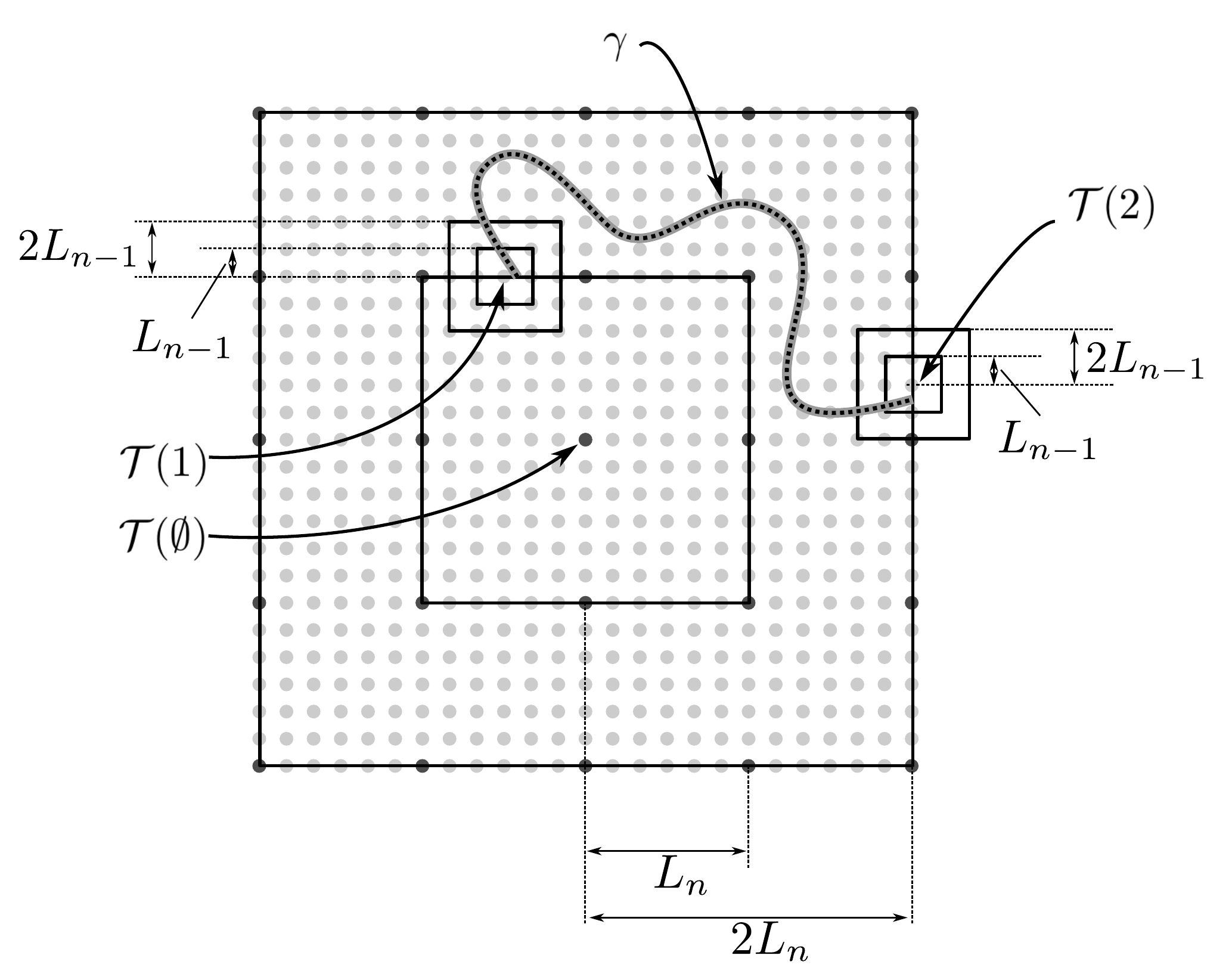}}
\end{center}
\caption{\label{fig:renorm} Illustration of the relation between the proper embedding $\mathcal{T}$
(see Definition \ref{def_proper_embedding_of_trees}) and
the path $\gamma$ that appears in Lemma \ref{lem:renorm_paths}. The light grey circles and dark grey circles represent points of the lattices $\mathcal{L}_{n-1}$ and $\mathcal{L}_n$, respectively.}
\end{figure}

\begin{lemma}[Lemma 3.3 of \cite{Ra15}]
\label{lem:renorm_paths}
If $\gamma$ is a $*$-connected path in $\Z^d$ which crosses the annulus at scale $L_n$:
$$\{\gamma\} \cap S(L_n-1) \neq \varnothing, \quad \{\gamma\} \cap S(2L_n) \neq \varnothing,$$
then there exists $\mathcal{T} \in \Lambda_{n,L}$ such that $\gamma$ crosses these bottom-level annuli:
\begin{equation}
\label{eq:renorm_paths} \{\gamma\} \cap S(\mathcal{T}(m),L-1) \neq \varnothing,\;\; \{\gamma\} \cap S(\mathcal{T}(m),2L) \neq \varnothing \quad \forall m \in T_{(n)}.
\end{equation}
\end{lemma}

Before we state the third lemma, we need some further notation.

For $0 \leq k \leq n$ and $m=(\xi_1,\dots,\xi_n) \in T_{(n)}$ we denote by
$\left. m \right|_k = (\xi_1,\dots,\xi_k) \in T_{(k)}$ the ancestor of $m$ at depth $k$.

The lexicographic distance of $m, m' \in T_{(n)}$ is defined by
\begin{equation}\label{def_eq_lex_dist}
\rho(m, m')= \min \{ \;  k \geq 0 \; : \; \left. m \right|_{n-k}=\left. m' \right|_{n-k} \; \}.
\end{equation}
For any $m \in T_{(n)}$ and $0 \leq k \leq n$ we define the set of $k$'th cousins of $m$ by
\begin{equation}\label{def_eq_tree_sphere}
T_{(n)}^{m,k}=\{ \, m' \in T_{(n)} \; : \; \rho(m,m')=k \, \},
\end{equation}
see Figure \ref{fig:canopy} for an illustration.  Note that the number of $k$'th cousins of $m$ is
\begin{equation}\label{cardinality_tree_sphere}
|T_{(n)}^{m,k}| \leq 2^{k}, \qquad  0 \leq k \leq n.
\end{equation}

\begin{figure}[htb]
\begin{center}
\setlength\fboxsep{0pt}
\setlength\fboxrule{0pt}
\fbox{\includegraphics[width = 0.5\textwidth]{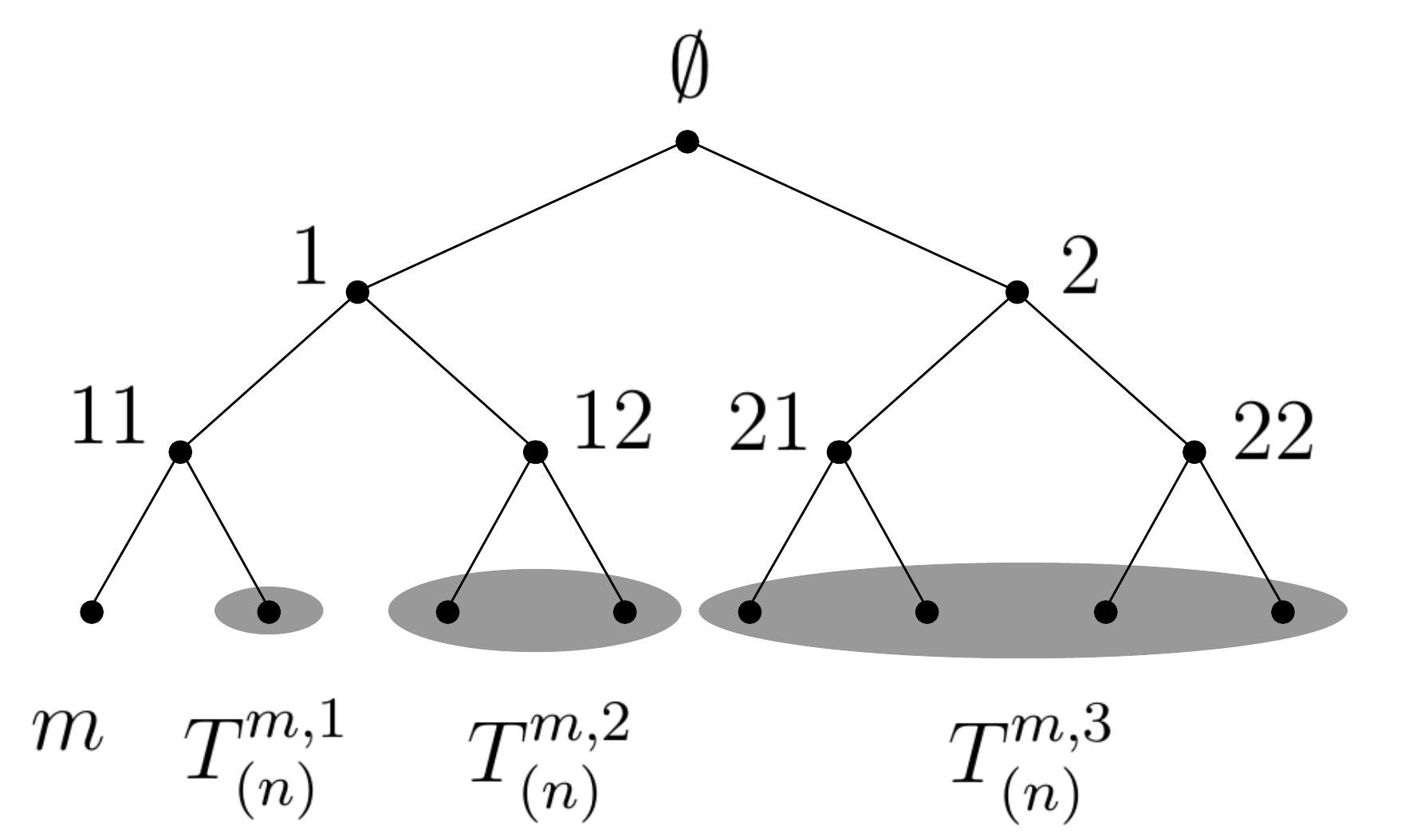}}
\end{center}
\caption{ \label{fig:canopy} An illustration of the subsets $T_{(n)}^{m,k}$ of leaves of $T_n$ defined in \eqref{def_eq_tree_sphere}.
The dyadic tree on the picture is of depth $n=3$ and the leaf denoted by $m$ is $111 \in T_{(n)}$.}
\end{figure}

The third lemma guarantees that if the lexicographic distance of $m$ and $m'$ is big then their embedded images are also far away from each other in $\mathbb{Z}^d$.

\begin{lemma}[Lemma 3.4 of \cite{Ra15}]
\label{lemma_far_in_tree_far_in_embedding} For all $n \in \mathbb{N}$, $\mathcal{T} \in \Lambda_{n,L}$, $m \in T_{(n)}$, $1 \leq k \leq n $,
$m' \in T_{(n)}^{m,k}$,  $y \in B(\mathcal{T}(m), 2L )$ and
 $z \in B(\mathcal{T}(m'), 2L )$, we have
 \begin{equation}\label{far_in_embedding_eq}
  |y-z| \geq L_{k-1} \stackrel{ \eqref{def:scalesLn} }{=} L 6^{k-1}.
 \end{equation}
\end{lemma}

\section{Proof of main result}
\label{section_proof_of_main}

In this section we prove Theorem \ref{thm_main}. Section \ref{subsection_lower_bound} contains the proof of
 $\liminf_{R \to \infty } \lambda_p(R) \geq \frac{1}{1-p_c}$, which is easier and shorter than the proof of $\limsup_{R \to \infty } \lambda_p(R) \leq \frac{1}{1-p_c}$,
 which will be given in Section \ref{subsection_upper_bound}. But first, let us state and prove Claim \ref{claim_binomial_largedev_bound} and Lemma \ref{lemma_X_t_big}, which will be used in both subsections.

\begin{claim}\label{claim_binomial_largedev_bound} Let $a, \varepsilon, p, q \in (0,1)$. Let $X \sim \mathrm{Bin}(N,q)$.
If  $q \leq p \leq \left( \frac{\varepsilon}{2} \right)^{1/a}$ holds then  $\mathbb{P}\left( X \geq a N \right) \leq \varepsilon^N. $
\end{claim}
\begin{proof} In the inequality marked by $(*)$ we use $q \leq p$ and $1-q \leq 1$:
\begin{equation}
\mathbb{P}\left( X \geq a N \right) \leq  \mathbb{E}\left( p^{-X} \right) / p^{-aN} =
   (q/p+1-q)^N p^{aN} \stackrel{(*)}{\leq} \left( 2 p^a \right)^N \leq \varepsilon^N.
   \end{equation}
\end{proof}

Recall from Section \ref{section_renorm_intro} that we denote by $L$ the bottom scale of our renormalization scheme. Let us  recall the notion of a proper embedding $\mathcal{T} \in \Lambda_{n,L}$ of the dyadic tree of depth $n$ from Definition \ref{def_proper_embedding_of_trees}.

 Given some  $\mathcal{T} \in \Lambda_{n,L}$
we will define the finite subset $\mathcal{X}$ of $\mathbb{Z}^d$ by
\begin{equation}\label{mathcal_X_tree}
 \mathcal{X}=\mathcal{X}(\mathcal{T})= \bigcup_{m  \in T_{(n)}} B( \mathcal{T}(m),2L)
\end{equation}
Note that it follows from Lemma \ref{lemma_far_in_tree_far_in_embedding} that for any $n  \in \mathbb{N}$ and any $\mathcal{T} \in \Lambda_{n,L}$
\begin{equation}\label{disjoint_X_cardinality}
 \text{the balls } B(\mathcal{T}(m), 2L ), \, m \in  T_{(n)} \text{ are disjoint, thus } |\mathcal{X}|=|B(2L)|\cdot 2^n.
\end{equation}
 Given such a set $\mathcal{X}$, we consider a joint realization of the labeled BRWs $Z_t(x,y)$ and the infection path indicators $\Xi_t(x,y)$, where $x \in \mathcal{X}, y \in \mathbb{Z}^d, 0 \leq t \leq T_0$, which satisfies the properties stated in Lemma  \ref{lemma_coupling}.

\smallskip

Let us recall from  Definition \ref{def_mathcal_X_t} the notion of the set of labels $\mathcal{X}_t$ that are alive at time $t$. Recall from  Lemma \ref{lem:renorm_count}
the constant $C_d$ that quantifies the combinatorial complexity of our renormalization scheme.
As we have already discussed in Remark \ref{remark_coupling_pitfalls}, we want to control the number $ |\mathcal{X} \setminus \mathcal{X}_{T_0}|$ of labels that got annihilated by time $T_0$. This is what we do in the next lemma.

\begin{lemma}[Few collisions if $R$ is large]\label{lemma_X_t_big}
 For any value of
$\lambda$, $L$ and $T_0$, there exists $R_0=R_0(\lambda,L,T_0) \in \mathbb{N}$ such that for
any $R \geq R_0$, any $n \in \mathbb{N}$ and any $\mathcal{T} \in \Lambda_{n,L}$ we have
\begin{equation}\label{X_t_big_eq}
  \mathbb{P}\left(\, |\mathcal{X} \setminus \mathcal{X}_{T_0}| \geq  2^n /2 \, \right) \leq \left( 2 C_d  \right)^{-2^n}.
\end{equation}
\end{lemma}
Before we prove Lemma \ref{lemma_X_t_big}, let us stress that the same $R_0$ works for all $n$:
this is because $\mathcal{X}$ is spread-out on all scales, cf.\ Remark \ref{remark_fractal}.
\begin{proof}[Proof of Lemma \ref{lemma_X_t_big}] Let us fix the value of $\lambda$, $L$ and $T_0$. Given any $n \in \mathbb{N}$ and $\mathcal{T} \in \Lambda_{n,L}$ we define $\mathcal{X}$ by
\eqref{mathcal_X_tree}, and for any $s \geq 1$ we can bound
\begin{equation}\label{momgen_exp_bound_matchcal_X}
\mathbb{E}\left( s^{|\mathcal{X} \setminus \mathcal{X}_{T_0}|} \right) \stackrel{ \eqref{annihilation_momgen_ineq} }{\leq }
\exp\left( |\mathcal{X}| \mathbb{P}(\tau_{0} \leq T_0 ) s + |\mathcal{X}| \max_{x \in \mathcal{X}}
\sum_{y \in \mathcal{X}\setminus \{ x \} }  \mathbb{P}(\tau_{x,y} \leq T_0 ) s^2
 \right).
\end{equation}
We will show that there exists $R_0$ such that for
any $R \geq R_0$ we have
\begin{align}
\label{tau_0_easy_bound}
\mathbb{P}(\tau_{0} \leq T_0 ) & \leq (48 C^2_d |B(2L)| )^{-1},\\
\label{tau_x_y_harder_bound}
\forall \, n \in \mathbb{N}, \, \mathcal{T} \in \Lambda_{n,L} \; : \;
\max_{x \in \mathcal{X}}
\sum_{y \in \mathcal{X}\setminus \{ x \} }  \mathbb{P}(\tau_{x,y} \leq T_0 )  & \leq (576 C_d^4 |B(2L)|)^{-1}.
\end{align}
 Before we prove \eqref{tau_0_easy_bound}
and \eqref{tau_x_y_harder_bound}, let us show that \eqref{X_t_big_eq} follows from them. First observe that
\begin{equation}\label{momgen_annih_ineq_simple}
  \mathbb{E}\left( (12 C_d^2)^{|\mathcal{X} \setminus \mathcal{X}_{T_0}|} \right)
  \stackrel{ \eqref{disjoint_X_cardinality}, \eqref{momgen_exp_bound_matchcal_X},\eqref{tau_0_easy_bound}, \eqref{tau_x_y_harder_bound}     }{ \leq }
  \exp\left( 2^n/2 \right)\leq 3^{2^n/2}.
\end{equation}
Now the desired \eqref{X_t_big_eq} follows using the exponential Chebyshev's inequality:
\begin{equation}
   \mathbb{P}\left(\, |\mathcal{X} \setminus \mathcal{X}_{T_0}| \geq  2^n /2 \, \right) \leq
    \mathbb{E}\left( (12 C_d^2)^{|\mathcal{X} \setminus \mathcal{X}_{T_0}|} \right) / (12 C_d^2)^{2^n/2}
    \stackrel{ \eqref{momgen_annih_ineq_simple} }{ \leq }
    \left( 2 C_d  \right)^{-2^n}.
\end{equation}
It remains to prove \eqref{tau_0_easy_bound}
and \eqref{tau_x_y_harder_bound}. Note that \eqref{tau_0_easy_bound} holds by Lemma \ref{lemma_crowd_bounds}\eqref{self_collision_bound} for
large enough $R$. The proof of  \eqref{tau_x_y_harder_bound} is more complicated: for any
$n \in \mathbb{N}$, $\mathcal{T} \in \Lambda_{n,L}$, $m  \in T_{(n)}$ and  $x \in B( \mathcal{T}(m),2L)$
\begin{multline}\label{first_ln2_ln_6}
\sum_{y \in \mathcal{X}\setminus \{ x \} }   \mathbb{P}(\tau_{x,y} \leq T_0 ) \stackrel{ \eqref{pair_collision_bound_eq} }{\leq}
 \frac{C'}{R^{d}} \sum_{y \in \mathcal{X} }    \exp\left(  -  \frac{|y-x|}{C'' R } \right)
 \stackrel{ \eqref{def_eq_tree_sphere}, \eqref{mathcal_X_tree}  }{=}
  \\
\frac{C'}{R^{d}} \sum_{k=0}^n \sum_{m' \in T_{(n)}^{m,k} } \sum_{y \in B( \mathcal{T}(m'),2L)  }  \exp\left(  -  \frac{|y-x|}{C'' R } \right)
\stackrel{ \eqref{cardinality_tree_sphere}, \eqref{far_in_embedding_eq} }{\leq} \\
\frac{C'}{R^{d}} \sum_{k=0}^n 2^{k} |B(2L)|  \exp\left(  -  \frac{L 6^{k-1} }{C'' R } \right)
\leq \\
\frac{C' }{R^{d}} \int_{-\infty}^\infty 2^{x+1}|B(2L)|  \exp\left(  -  \frac{L 6^{x-2} }{C'' R } \right)\, \mathrm{d}x
 \stackrel{(*)}{=}\\
 \frac{C' 2 |B(2L)|}{R^d \ln(6)} \left( \frac{C'' R}{L 6^{-2}} \right)^{\frac{\ln(2)}{\ln(6)}}\int_0^\infty u^{ \frac{\ln(2)}{\ln(6)}-1}e^{-u}\mathrm{d}u=
  \widehat{C} R^{\frac{\ln(2)}{\ln(6)}-d}
\end{multline}
for some $\widehat{C}=\widehat{C}(L,C',C'')=\widehat{C}(\lambda,L,T_0)$, where in $(*)$ we performed the substitution
$u=\frac{L 6^{x-2} }{C'' R }$. Thus \eqref{tau_x_y_harder_bound}  holds for large enough $R$, since $\frac{\ln(2)}{\ln(6)} <d$.
The proof of Lemma \ref{lemma_X_t_big} is complete.
\end{proof}

\subsection{Lower bound on the percolation threshold}
\label{subsection_lower_bound}

The aim of this subsection is to prove $\liminf_{R \to \infty } \lambda_p(R) \geq \frac{1}{1-p_c}$.
 It is enough to  show that the following proposition holds. Recall the definition of the percolation event $\mathrm{Perc}$ from
\eqref{perc_def_eq}.
  \begin{proposition}[Subcritical behavior] \label{prop_subcrit}
  For any $\lambda < \frac{1}{1-p_c}$ there exists some $R_0=R_0(\lambda)$ such that if $R \geq R_0$ then $\mu_{\lambda,R}(\mathrm{Perc})=0$.
  \end{proposition}

Let us fix some $\lambda < \frac{1}{1-p_c}$  for the rest of Section \ref{subsection_lower_bound}. Recall the notation  $A \stackrel{\xi}{\longleftrightarrow} B$
from Definition \ref{def_connected_in_a_config}.

\begin{lemma}[Annulus crossing]\label{lemma_subcrit_1}
There exists $L \in \mathbb{N}$ and $R_0 \in \mathbb{N}$ such that for any $n \in \mathbb{N}$ and any $R \geq R_0$ we have
\begin{equation}
 \mu_{\lambda,R}\left( B(L_{n}) \stackrel{ \xi }{\longleftrightarrow}  B(2L_n)^c  \right) \leq 2 \cdot  2^{-2^n}, \quad \text{where} \quad L_n  \stackrel{ \eqref{def:scalesLn} }{=}L \cdot 6^n.
\end{equation}
\end{lemma}
 Before proving Lemma \ref{lemma_subcrit_1}, we first note that Proposition \ref{prop_subcrit} follows from it:
\begin{proof}[Proof of Proposition \ref{prop_subcrit}] Observing that
\begin{equation}
\mathrm{Perc} \subseteq \bigcup_{m \geq 1}\bigcap_{n \geq m} \{ B(L_{n}) \stackrel{ \xi }{\longleftrightarrow}  B(2L_n)^c  \},
\end{equation}
one concludes that if Lemma \ref{lemma_subcrit_1}  holds then $\mu_{\lambda,R}(\mathrm{Perc}) = 0$.
\end{proof}
 It remains to prove Lemma \ref{lemma_subcrit_1}.
Recalling \eqref{sigma_lambda} we note that $\sigma(\lambda)<p_c$, so let us fix some $p\in (\sigma(\lambda), p_c)$.
Given this value of $p$ we can use the results about the exponential decay of the cluster radius in subcritical Bernoulli percolation (cf.\
\cite[Section 5.2]{Gr99}) to
 choose (and fix for the rest of Section \ref{subsection_lower_bound}) $L$ such that under the product Bernoulli distribution $\pi_p$ with density $p$ we have
\begin{equation}\label{subcrit_benoulli_perc_crossing}
  \pi_p \left( B( L ) \stackrel{ \xi }{\longleftrightarrow}  B(2L-1)^c \right) \leq (4 C_d)^{-2},
\end{equation}
where $C_d$ was defined in Lemma \ref{lem:renorm_count}.

\smallskip

It follows from \eqref{sigma_lamb_T} and \eqref{sigma_lambda} that $t \mapsto \sigma(\lambda,t)$ is continuous, moreover $\lim_{t \to \infty} \sigma(\lambda,t)=\sigma(\lambda)$, therefore we can choose (and fix for the rest of Section \ref{subsection_lower_bound}) $T$ such that
\begin{equation}\label{sigma_p}
\sigma(\lambda,T)=p.
\end{equation}

Let us recall from Definition \ref{Xi_graph_def} that $\Xi_t(x,y)$ denotes the indicator of the event that
there is an infection path connecting $(x,0)$ to $(y,t)$ in the graphical construction of the contact process. Let us define
\begin{equation}\label{def_xi_star_T}
  \xi^*_T(x):= \mathds{1}\left[ \, \exists \, y \in \mathbb{Z}^d \, : \, \Xi_T(x,y) >0    \, \right], \qquad x \in \mathbb{Z}^d.
\end{equation}
Having fixed the bottom scale $L$ of our renormalization scheme, let us  recall the notion of a proper embedding $\mathcal{T} \in \Lambda_{n,L}$ of the dyadic tree of depth $n$ from Definition
\ref{def_proper_embedding_of_trees}. The following lemma will give the proof of Lemma \ref{lemma_subcrit_1}.

\begin{lemma}[Bottom-level annuli crossings in $\xi^*_T$] \label{lemma_subcrit_spread_out} Having fixed $\lambda$, $L$ and $T$ as above, there exists $R_0 \in \mathbb{N}$ such that for
any $R \geq R_0$, any $n \in \mathbb{N}$ and any $\mathcal{T} \in \Lambda_{n,L}$ we have
 \begin{equation}\label{subcrit_spread_out_eq}
   \mathbb{P}\left( \bigcap_{m  \in T_{(n)}} \left\{  B(\mathcal{T}(m),  L ) \stackrel{ \xi^*_T }{\longleftrightarrow}  B(\mathcal{T}(m), 2L-1)^c \right\}    \right)
   \leq 2 \left( 2 C_d  \right)^{-2^n}.
 \end{equation}
\end{lemma}
Before we prove Lemma \ref{lemma_subcrit_spread_out}, let us deduce Lemma \ref{lemma_subcrit_1} from it.

\begin{proof}[Proof of Lemma \ref{lemma_subcrit_1}]
Recall from \eqref{upper_invariant_active_path} that if we define $\xi^*_\infty:= \lim_{t \to \infty } \xi^*_t(x) $ then the law of the configuration $\xi^*_\infty$ is
$\mu_{\lambda,R}$. Also note that $\xi^*_T(x) \geq \xi^*_\infty(x)$, thus the equation marked by $(*)$ below holds, since
$\{ B(L_{n}) \stackrel{ \xi }{\longleftrightarrow}  B(2L_n)^c \}$ is an increasing event:
\begin{multline}\label{we_use_monotone}
\mu_{\lambda,R}\left( B(L_{n}) \stackrel{ \xi }{\longleftrightarrow}  B(2L_n)^c  \right) \stackrel{(*)}{\leq} \mathbb{P} \left( B(L_{n}) \stackrel{ \xi^*_T }{\longleftrightarrow}  B(2L_n)^c \right) \stackrel{(**)}{\leq} \\
 \mathbb{P}\left( \bigcup_{ \mathcal{T} \in \Lambda_{n,L}  }  \bigcap_{m  \in T_{(n)}} \left\{  B(\mathcal{T}(m),  L ) \stackrel{ \xi^*_T }{\longleftrightarrow}  B(\mathcal{T}(m), 2L-1)^c \right\}   \right) \stackrel{(***)}{\leq}   \\
 \sum_{ \mathcal{T} \in \Lambda_{n,L}  } 2 \left( 2 C_d  \right)^{-2^n} \stackrel{ \eqref{eq:renorm_count} }{\leq} (C_d)^{2^n} 2 \left( 2 C_d  \right)^{-2^n}=2 \cdot 2^{-2^n},
\end{multline}
where $(**)$ holds by Lemma \ref{lem:renorm_paths}, $(***)$ holds by the union bound and \eqref{subcrit_spread_out_eq}  (as soon $R \geq R_0$, where $R_0$ appears in the statement of Lemma \ref{lemma_subcrit_spread_out}). The proof of Lemma \ref{lemma_subcrit_1} is complete.
\end{proof}

It remains to prove Lemma  \ref{lemma_subcrit_spread_out}.
Having already fixed $L$ (the bottom scale of our renormalization scheme), for any proper embedding $\mathcal{T} \in \Lambda_{n,L}$
we define the finite subset $\mathcal{X}$ of $\mathbb{Z}^d$ by \eqref{mathcal_X_tree}.
Note that in order to determine the outcome of the event that appears in \eqref{subcrit_spread_out_eq}, it is enough to observe the random variables $\xi^*_T(x), x \in \mathcal{X}$. Given this set $\mathcal{X}$ we construct a joint realization of  $Z_t(x,y)$ and $\Xi_t(x,y)$, where $x \in \mathcal{X}, y \in \mathbb{Z}^d, 0 \leq t \leq T$ as in Lemma  \ref{lemma_coupling}.
Recall from \eqref{Z_T_x} that $Z_T(x)$ denotes the number of BRW particles with label $x$ at time $T$, where $x \in \mathcal{X}$. We define a random configuration of zeros and ones $\left(\zeta(x)\right)_{x \in \mathcal{X}}$  by letting
\begin{equation}\label{zeta_def}
  \zeta(x):= \mathds{1}\left[ \, 0<  Z_T(x)  \, \right]= \mathds{1}\left[ \, \exists \, y \in \mathbb{Z}^d \, : \, Z_T(x,y) >0    \, \right], \quad x \in \mathcal{X},
\end{equation}
thus $\zeta(x)$ is the indicator of the event that the number of BRW particles with label $x$ is nonzero at time $T$.
Let us define the random variable
\begin{equation}\label{def_Y_n}
  Y_n:= \sum_{m \in T_{(n)} }  \mathds{1}\left[  B(\mathcal{T}(m),  L ) \stackrel{ \zeta }{\longleftrightarrow}  B(\mathcal{T}(m), 2L-1)^c  \right],
\end{equation}
thus $Y_n$ is the number of bottom-level annuli crossed by $\zeta$.
Let us state a lemma, which (together with Lemma \ref{lemma_X_t_big}) will give the proof of Lemma  \ref{lemma_subcrit_spread_out}.

\begin{lemma}[Bottom-level annuli crossings in $\zeta$]\label{lemma_Y_n_not_too_big} Having fixed $\lambda$, $L$ and $T$ as above, for any $n \in \mathbb{N}$ and any $\mathcal{T} \in \Lambda_{n,L}$ we have
  \begin{equation}\label{Y_n_not_too_big}
    \mathbb{P}\left( Y_n \geq  2^n /2 \right) \leq \left( 2 C_d  \right)^{-2^n}.
  \end{equation}
\end{lemma}
\begin{proof}
  The random variables $\left(\zeta(x)\right)_{x \in \mathcal{X}}$ are i.i.d.\ by Claim \ref{claim_brw_bp}.
Note that $\mathbb{E}( \zeta(x)) = \sigma(\lambda,T) = p$ by Claim \ref{claim_brw_bp}, \eqref{sigma_lamb_T} and \eqref{sigma_p}, thus
$Y_n \sim \mathrm{BIN}(2^n, q)$ with $q \leq (4 C_d)^{-2}$ (cf.\ \eqref{subcrit_benoulli_perc_crossing}).
 Choosing $N=2^n$, $a=1/2$, $p=(4 C_d)^{-2}$ and $\varepsilon=\left( 2 C_d  \right)^{-1}$ in Claim \ref{claim_binomial_largedev_bound}, the desired inequality  \eqref{Y_n_not_too_big}  follows, since $p \leq \left( \frac{\varepsilon}{2} \right)^{1/a}$ indeed holds.
\end{proof}

\begin{proof}[Proof of Lemma  \ref{lemma_subcrit_spread_out}] Having already fixed $\lambda$, $L$ and $T$ as above, let us choose $R_0=R_0(\lambda,L,T)$ as in
Lemma \ref{lemma_X_t_big}. In order to prove Lemma  \ref{lemma_subcrit_spread_out},
it is enough to show that the inclusion
\begin{equation}\label{bad_inclusion_subcr}
  \bigcap_{m  \in T_{(n)}} \left\{  B(\mathcal{T}(m),  L ) \stackrel{ \xi^*_T }{\longleftrightarrow}  B(\mathcal{T}(m), 2L-1)^c \right\} \subseteq
  \{   |\mathcal{X} \setminus \mathcal{X}_T| + Y_n \geq  2^n   \}
\end{equation}
holds,
because \eqref{bad_inclusion_subcr} together with Lemma \ref{lemma_X_t_big}, Lemma \ref{lemma_Y_n_not_too_big} and the union bound give
\eqref{subcrit_spread_out_eq} for any $R \geq R_0$.

\smallskip

 Note that by Lemma \ref{lemma_coupling} (cf.\ the definitions \eqref{def_xi_star_T} and \eqref{zeta_def}) we have
$\xi^*_T(x)=\zeta(x)$ for all  $x \in \mathcal{X}_T$,
therefore if $B(\mathcal{T}(m),  L ) \stackrel{ \xi^*_T }{\longleftrightarrow}  B(\mathcal{T}(m), 2L-1)^c$ holds for some $m \in  T_{(n)}$ then either
there exists a vertex $x \in B(\mathcal{T}(m),2L) \setminus \mathcal{X}_T$ or we have
$B(\mathcal{T}(m),  L ) \stackrel{ \zeta }{\longleftrightarrow}  B(\mathcal{T}(m), 2L-1)^c$. Now taking into consideration that $|T_{(n)}|=2^n$
and that the union in the definition \eqref{mathcal_X_tree} of $\mathcal{X}$ is disjoint (cf.\ \eqref{disjoint_X_cardinality}),
we obtain \eqref{bad_inclusion_subcr}. The proof of Lemma  \ref{lemma_subcrit_spread_out} is complete.
\end{proof}

The proof of $\liminf_{R \to \infty } \lambda_p(R) \geq \frac{1}{1-p_c}$ is complete.

\subsection{Upper bound on the percolation threshold}
\label{subsection_upper_bound}

The aim of this subsection is to prove $\limsup_{R \to \infty } \lambda_p(R) \leq \frac{1}{1-p_c}$.

 It is enough to  show that the following proposition holds.
  \begin{proposition}[Supercritical behavior]\label{prop_supercrit}
  For any $\lambda > \frac{1}{1-p_c}$ there exists some $R_1=R_1(\lambda)$ such that if $R \geq R_1$ then $\mu_{\lambda,R}(\mathrm{Perc})=1$.
  \end{proposition}

Let us fix some $\lambda > \frac{1}{1-p_c}$  for the rest of Section \ref{subsection_upper_bound}.

For any $x \in \mathbb{Z}^d$ and $L \in \mathbb{N}$
we will define the  event $H_L(x)$ involving the restriction of the configuration $\xi$ to the box with center $x$ and radius $L$ which occurs if the percolation configuration restricted to the box $B(x,L)$ is ``locally supercritical''. More precisely, let $H_L(x)$ denote the set of configurations $\xi \in \{0,1\}^{\Z^d}$ that satisfy
\begin{equation}\label{high_connect}
H_L(x) = \left\{\begin{array}{c} \xi \in \{0,1\}^{\Z^d}:  \xi|_{B(x,L)}\text{ has a unique open cluster of }\\
\text{  diameter greater than or equal to $L$, moreover } \\
\text{ this cluster intersects all the faces of $B(x,L)$}\end{array} \right\}.
\end{equation}

\begin{remark}\label{remark_not_monotone} The event $H_L(x)$ is not monotone in the variable $\xi$, so we cannot use stochastic domination to compare the probability of
$H_L(x)$ under various measures (like we did in the inequality marked by $(*)$ in \eqref{we_use_monotone}).
\end{remark}

Recall that the constant $C_d$ was defined in Lemma \ref{lem:renorm_count} and that $\sigma(\lambda)$ was defined in \eqref{sigma_lambda}.

\begin{claim}[Choice of $L$]\label{claim_local_supercrit_probab} Having already fixed $\lambda$ as above, we can choose  $L \in \mathbb{N}$ (and fix it for the rest of Section \ref{subsection_upper_bound}) such that
under the product Bernoulli measure $\pi_{\sigma(\lambda)}$ with density $\sigma(\lambda)$  we have
\begin{equation}\label{local_supcrit_prob}
  \pi_{\sigma(\lambda)}\left[ \, \xi \in H_L(0)^c \, \right] <  (4 C_d)^{-6}.
\end{equation}
\end{claim}
\begin{proof}
By \eqref{sigma_lambda} our assumption  $\lambda > \frac{1}{1-p_c}$ implies  $\sigma(\lambda)>p_c $.
The local supercriticality event defined in \eqref{high_connect} is formulated in terms of Bernoulli \emph{site} percolation, and it follows from \cite[Theorem (7.61)]{Gr99} that if we consider the analogous event $H^*_L(0)$ for Bernoulli \emph{bond} percolation then
  $\pi_{p}\left(  \xi \in H^*_L(0) \right) \to 1$ holds as $L \to \infty$ for any $p$  above the Bernoulli bond percolation threshold of $\mathbb{Z}^d$.
  The proof is based on a block argument originally developed in \cite{Pi96} and \cite{DP96} which, as mentioned in the latter reference, works equally well for site percolation. We thus omit further details of the proof of Claim \ref{claim_local_supercrit_probab}.
  \end{proof}

Having fixed the bottom scale $L$ of our renormalization scheme, let us  recall the notion of a proper embedding $\mathcal{T} \in \Lambda_{n,L}$ of the dyadic tree of depth $n$ from Definition \ref{def_proper_embedding_of_trees}. The following lemma will give the proof of Proposition \ref{prop_supercrit}.

\begin{lemma}[Bottom level local supercriticality under $\mu_{\lambda,R}$]\label{lemma_supercrit_spread_out} Having fixed $\lambda$ and $L$ as above, there exists $R_1 \in \mathbb{N}$ such that for
any $R \geq R_1$, any $n \in \mathbb{N}$ and any $\mathcal{T} \in \Lambda_{n,L}$ we have
 \begin{equation}\label{supercrit_spread_out_eq}
   \mu_{\lambda,R}\left( \bigcap_{m  \in T_{(n)}} \left\{ \, \xi \in H_L(\mathcal{T}(m)  )^c \,  \right\}    \right)
   \leq 5 \left( 2 C_d  \right)^{-2^n}.
 \end{equation}
\end{lemma}
Before we prove Lemma \ref{lemma_supercrit_spread_out}, let us deduce Proposition \ref{prop_supercrit} from it.
\begin{proof}[Proof of Proposition \ref{prop_supercrit}]
Let us consider $R \geq R_1$, where $R_1$ is defined in Lemma \ref{lemma_supercrit_spread_out}.
 Given a percolation configuration $\xi \in \{0, 1 \}^{\mathbb{Z}^d}$ we will define an auxiliary
coarse-grained  percolation configuration $\eta \in \{0, 1 \}^{\mathbb{Z}^d}$  by letting
\begin{equation}\label{eta_coarse_grained}
 \eta(z):= \mathds{1}\left[ \, \xi \in H_L( L z  ) \, \right], \quad \overline{\eta}(z):=1-\eta(z),  \qquad  z \in \mathbb{Z}^d.
  \end{equation}
  First we show that in order to prove $\mu_{\lambda,R}( \xi \in \mathrm{Perc})=1$ (i.e., that a configuration $\xi$ with distribution $\mu_{\lambda,R}$ percolates, cf.\ \eqref{perc_def_eq}),
  it is enough to show that the corresponding $\eta$ percolates. Indeed, if $\eta \in \mathrm{Perc}$, i.e., if there is a sequence of vertices $z_1,z_2,\ldots \in \mathbb{Z}^d$ which forms an infinite nearest neighbour simple path (cf.\ Definition \ref{def_paths})
and $\eta(z_i)=1$ for each $i \in \mathbb{N}$, then by \eqref{high_connect} and \eqref{eta_coarse_grained}, for each $i \in \mathbb{N}$ the ball $B(L z_i,L)$ contains a unique special
$\xi$-cluster with a big diameter, moreover the balls $B(L z_i,L)$ and $B(L z_{i+1},L)$ overlap in a way that \eqref{high_connect} guarantees that their unique special $\xi$-clusters must intersect,
thus there is a $\xi$-cluster which intersects $B(L z_i,L)$ for each $i \in \mathbb{N}$, therefore $\xi$ percolates.

\smallskip

It remains to show $\mu_{\lambda,R}( \eta \in \mathrm{Perc})=1$. Recalling the notation of $*$-connectedness from Definition \ref{def_connected_in_a_config} and the notion of spheres from \eqref{sphere_def}, we will show that
it is unlikely to see a $*$-connected $\eta$-closed path crossing a large annulus:
\begin{equation}\label{no_star_crossing_in_overline_eta_whp}
\mu_{\lambda,R} \left( S(6^n-1) \stackrel{*\overline{\eta}}{ \longleftrightarrow} S(2 \cdot 6^n)   \right) \leq 5 \cdot 2^{-2^n}, \qquad n \in \mathbb{N}.
\end{equation}
Before we prove \eqref{no_star_crossing_in_overline_eta_whp}, let us first deduce $\mu_{\lambda,R}( \eta \in \mathrm{Perc})=1$ from it
using a variant of the classical Peierls argument, see \cite{Pei36} and \cite[Section 1.4]{Gr99}.
Let us consider a copy $F$ of $\mathbb{Z}^2$ inside $\mathbb{Z}^d$ for which $0 \in F$. Denote by $A_n$ the event that  $S(6^n)$ is connected to infinity by
an $\eta$-open nearest neighbour simple path that lies within $F$.
For every $n \in \mathbb{N}$ we have
\begin{multline}\label{peierls_bound}
\mu_{\lambda,R} \left( A_n^c  \right) \stackrel{(*)}{\leq}
\mu_{\lambda,R}\left( \bigcup_{k=n}^{\infty} \bigcup_{|z|\leq 12,\, z \in F} \left\{   S(6^k z ,6^k-1) \stackrel{*\overline{\eta}}{\longleftrightarrow} S(6^k z ,2 \cdot 6^k)   \right\}  \right) \\
\stackrel{\eqref{no_star_crossing_in_overline_eta_whp}}{\leq}  \sum_{k=n}^\infty 5^3  2^{-2^k},
\end{multline}
where $(*)$ holds because if the set of sites $z' \in F$ for which $z' \stackrel{\eta}{\longleftrightarrow} S(6^n)$ holds  is finite then $S(6^n-1)$ is surrounded by a $*$-connected $\overline{\eta}$-open path (cf.\
Definition \ref{def_paths}) that lies within $F$, cf.\   Definition 4, Definition 7 and Corollary 2.2 of \cite{Kesten82}.
We obtain $\mu_{\lambda,R}( \eta \in \mathrm{Perc})=1$ by letting $n \to \infty$ in \eqref{peierls_bound}.

It remains to show \eqref{no_star_crossing_in_overline_eta_whp}:
\begin{multline}
\mu_{\lambda,R} \left( S(6^n-1) \stackrel{*\overline{\eta}}{ \longleftrightarrow} S(2 \cdot 6^n)   \right)
\stackrel{(**)}{\leq} \sum_{\mathcal{T} \in \Lambda_{n,1} } \mu_{\lambda,R}\left( \bigcap_{m \in T_{(n)}} \{\, \overline{\eta}(\mathcal{T}(m))=1\, \}  \right)
 \\
 \stackrel{(***)}{\leq}
 (C_d)^{2^n} 5 \left( 2 C_d  \right)^{-2^n}= 5 \cdot 2^{-2^n}, \quad n \in \mathbb{N},
\end{multline}
where $(**)$ holds by the $L=1$ case of Lemma \ref{lem:renorm_paths} and the union bound,
while $(***)$ follows using Lemma \ref{lem:renorm_count}, \eqref{eta_coarse_grained} and Lemma \ref{lemma_supercrit_spread_out} together with the fact that $\mathcal{T} \in \Lambda_{n,1}$
if and only if $L\mathcal{T} \in \Lambda_{n,L}$, cf.\ Definition \ref{def_proper_embedding_of_trees}. The proof of Proposition \ref{prop_supercrit} is complete.
\end{proof}

It remains to prove Lemma \ref{lemma_supercrit_spread_out}.

\begin{remark}\label{remark_supercrit_plan_heu}
In our subcritical proof (i.e.,  Section \ref{subsection_lower_bound}), we fixed the time horizon $T$
and defined a configuration $\xi^*_T$  in \eqref{def_xi_star_T} whose law stochastically dominates $\mu_{\lambda,R}$, and later in the proof of Lemma \ref{lemma_subcrit_spread_out}
we have shown that the configuration $\xi^*_T$ is actually quite similar to the i.i.d.\ configuration $\zeta$ defined in \eqref{zeta_def}.
However, as we have already discussed in Remark \ref{remark_not_monotone}, now we cannot resort to stochastic domination in a way we did in Section \ref{subsection_lower_bound}. This time we will show that one can choose $T$ such that the law of $\zeta$ is a good enough approximation of  $\mu_{\lambda,R}$ to yield Lemma \ref{lemma_supercrit_spread_out}. We will argue that
if there is an infection path from $x$ reaching up to time $T$, then actually there are many such paths, and at least one of them can be continued to infinity with high probability.
 We will prove this by
placing an independent copy of the upper invariant configuration $\xi^{\sstar}$ at time $T+1$ (which represents the starting points of infection paths from $T+1$ to infinity) and using that the coarse-grained boxes  associated to $\xi^{\sstar}$ are mostly good (cf.\ Definition \ref{def:good_boxes})  by Theorem \ref{thm:density_good_boxes}. The time interval $[T,T+1]$ is needed because we will require our infection paths (more precisely, the BRW particles at time $T$ with label $x$) to perform one jump in $[T,T+1]$. We need this jump because
 the definition of good boxes  only guarantees that the empirical density of $\xi^{\sstar}$ is higher than $\alpha$ in translates of
 the box $[0,R )^d$, and a jump of range $R$ allows us pick a uniform sample from a translate of $[0,R )^d$ that contains a particle with label $x$ at time $T$.
 \end{remark}

Our next goal is to fix a time horizon $T$. Recall that we  defined the constant $\alpha=\alpha(\lambda,d) \in (0,1)$ in Proposition \ref{prop:growth_to_adjacent_box} (and that the same $\alpha$ appears in Theorem \ref{thm:density_good_boxes}). Having already fixed the value of $\lambda$, let us define the constant $p_* \in (0,1)$ by
\begin{equation}\label{p_star}
  p_*  := 1-(1-e^{-\lambda})e^{-2}\alpha/3^{d}.
\end{equation}

Recall the notion of the branching process $(Z_t)$ and $\sigma(\lambda,t)$ from Definition \ref{def_bp_pop_at_t} and \eqref{sigma_lamb_T}. The inequality \eqref{local_supcrit_prob_sigma_lambda_T}
below is just a variant of \eqref{local_supcrit_prob}.
 The inequality \eqref{many_children_Z_T} below
quantifies the heuristic that after a long time $T$, a supercritical branching process has either already died out (i.e., $Z_T=0$) or otherwise the population $Z_T$ at time $T$  is very big.
\begin{claim}[Choice of $T$] \label{claim_supcrit_T_choice}
 Having already fixed $\lambda$ and $L$,  we can choose  $T \in \mathbb{R_+}$ (and fix it for the rest of Section \ref{subsection_upper_bound}) such that the following inequalities both hold:
\begin{align}
\label{local_supcrit_prob_sigma_lambda_T}
\pi_{\sigma(\lambda,T)}\left[ \, \xi \in H_L(0)^c \, \right] \leq & \, (4 C_d)^{-6},
   \\
\label{many_children_Z_T}
  \mathbb{E}\left( \, p_*^{Z_T} \mathds{1}[ \, 0 < Z_T \, ] \, \right) \leq & \, \left( 2^{|B(2L)|+1} C_d \right)^{-9}.
\end{align}
\end{claim}
\begin{proof} The function $p \mapsto \pi_{p}\left[ \, \xi \in H_L(0)^c \, \right]$ is continuous (in fact it is a polynomial), since the event $\{\xi \in H_L(0)^c \}$ only depends on
$\xi(x), x \in B(L)$. Also note that $\lim_{t \to \infty}\sigma(\lambda,t)=\sigma(\lambda)$ by \eqref{sigma_lamb_T} and \eqref{sigma_lambda}, thus \eqref{local_supcrit_prob_sigma_lambda_T} holds by \eqref{local_supcrit_prob} for large enough $T$.

\smallskip

Using $\lambda>1$, \eqref{sigma_lamb_T} and \eqref{Z_t_gen_fn} we see that $\lim_{t \to \infty }  \mathbb{E}\left( \, p_*^{Z_t} \mathds{1}[ \, 0 < Z_t \, ] \, \right) =0$,
so \eqref{many_children_Z_T} holds for large enough $T$. The proof of Claim \ref{claim_supcrit_T_choice} is complete.
\end{proof}

\begin{remark} At this point the choice of $T$ in Claim \ref{claim_supcrit_T_choice} seems unmotivated. We decided to fix $T$ early in the proof in order to
make it clear that there is no circularity in the definition of the parameters $\lambda$, $L$, $T$ and $R$.
We will use \eqref{local_supcrit_prob_sigma_lambda_T} later on to prove \eqref{Z_n_largedev_bd}
like we have already used \eqref{subcrit_benoulli_perc_crossing} to prove \eqref{Y_n_not_too_big}.
We will use \eqref{many_children_Z_T} later in the proof of Lemma \ref{lemma_survival_to_T_suffices} (cf.\ \eqref{here_is_where_we_use_many_children}).
\end{remark}

Having already fixed $L$ (the bottom scale of our renormalization scheme), for any proper embedding $\mathcal{T} \in \Lambda_{n,L}$
we define the finite subset $\mathcal{X}$ of $\mathbb{Z}^d$ by \eqref{mathcal_X_tree}.
Note that in order to determine the outcome of the event that appears in \eqref{supercrit_spread_out_eq}, it is enough to observe the random variables $\xi(x), x \in \mathcal{X}$.

\begin{definition}\label{def_ingredients_supercrit} Let us consider
\begin{itemize}
\item the coupling of  $Z_t(x,y)$ and $\Xi_t(x,y)$, $x \in \mathcal{X}, y \in \mathbb{Z}^d, 0 \leq t \leq T+1$, cf.\ Lemma  \ref{lemma_coupling},
\item an independent  $\{0,1\}^{\mathbb{Z}^d}$ -valued
 random variable $\left(\xi^{\sstar}(x)\right)_{x \in \mathbb{Z}^d}$ with    $\mu_{\lambda,R}$  distribution.
\end{itemize}
Using these ingredients, we  construct
\begin{align}
\label{xi_star}
\xi^{*}(x) & := \mathds{1}\left[  \, \exists y \in \mathbb{Z}^d \, : \;  \Xi_{T+1}(x,y)>0 \text{ and }  \xi^{\sstar}(y)=1 \,  \right], \qquad x \in \mathcal{X}, \\
\label{zeta_star}
  \zeta^*(x) & := \mathds{1}\left[ \, \exists y \in \mathbb{Z}^d \, : \;  Z_{T+1}(x,y)>0 \text{ and }  \xi^{\sstar}(y)=1 \, \right],  \qquad x \in \mathcal{X}, \\
\label{zeta_def_supcrit}
  \zeta(x) & := \mathds{1}\left[ \, 0<  Z_T(x)  \, \right]= \mathds{1}\left[ \, \exists \, y \in \mathbb{Z}^d \, : \, Z_T(x,y) >0    \, \right], \qquad x \in \mathcal{X}.
\end{align}
\end{definition}

\begin{remark}  We will show that $(\xi^*(x))_{x \in \mathcal{X}}$ has the upper invariant distribution $\mu_{\lambda,R}$ and  $(\zeta(x))_{x \in \mathcal{X}}$ has product Bernoulli distribution $\pi_{\sigma(\lambda,T)}$. Loosely speaking, we want to show that $\xi^*$ and $\zeta$ are close, and we will achieve
this by showing that $\zeta^*$ is  close to both $\xi^*$ and $\zeta$.
\end{remark}

Note that it follows from \eqref{zeta_star} and \eqref{zeta_def_supcrit} that
\begin{equation}\label{zeta_geq_zeta_star}
  \zeta(x) \geq \zeta^*(x), \quad \text{thus} \quad \zeta(x)-\zeta^*(x)=\zeta(x)(1-\zeta^*(x)), \quad   x \in \mathcal{X}.
\end{equation}

The proof of Lemma \ref{lemma_supercrit_spread_out} will easily follow from the next lemma.

\begin{lemma}[$\zeta$ and $\zeta^*$ are close]\label{lemma_survival_to_T_suffices}
 Having already fixed $\lambda$, $L$ and $T$ as above, there exists $R_2=R_2(\lambda,L,T) \in \mathbb{N}$ such that for
any $R \geq R_2$, $n \in \mathbb{N}$ and  $\mathcal{T} \in \Lambda_{n,L}$ we have
\begin{equation}\label{survival_to_T_suffices_eq}
 \mathbb{P}\left( \sum_{x \in \mathcal{X} } \left( \zeta(x) - \zeta^{*}(x) \right) \geq  2^n /3 \right) \leq 3 \left( 2 C_d  \right)^{-2^n}.
\end{equation}
\end{lemma}

Before we prove Lemma \ref{lemma_survival_to_T_suffices}, let us deduce Lemma \ref{lemma_supercrit_spread_out} from it.

\begin{proof}[Proof of Lemma \ref{lemma_supercrit_spread_out}]
We first note that
\begin{equation}\label{uim_repr}
\text{the law of $\left(\xi^{*}(x)\right)_{x \in \mathcal{X} }$
 is the same as the restriction of $\mu_{\lambda,R}$ to $\mathcal{X}$,}
 \end{equation}
    as we now explain. Recalling the graphical construction of Section \ref{subsection_graphical_constr_of_cont_proc}, we may think about $\xi^{\sstar}(y)$ as the indicator
of the event that there is an infection path from $(y,T+1)$ to infinity (cf.\ Claim \ref{claim_graphical_upper_invariant}) and $\Xi_{T+1}(x,y)$ as the indicator
of the event that there is an infection path from $(x,0)$ to $(y,T+1)$ (cf.\ \eqref{Xi_graphical_def_eq}). With this interpretation,  $\xi^{*}(x)$ (defined by \eqref{xi_star}) becomes the indicator of the event that there exists an
infection path from $(x,0)$ to infinity, thus \eqref{uim_repr} holds by Claim \ref{claim_graphical_upper_invariant}. We thus have
\begin{equation}\label{xi_star_invariant}
\mu_{\lambda,R}\left( \bigcap_{m  \in T_{(n)}} \left\{ \, \xi \in H_L(\mathcal{T}(m))^c \,  \right\}    \right) \stackrel{ \eqref{uim_repr} }{=}
\mathbb{P}\left( \bigcap_{m  \in T_{(n)}} \left\{ \, \xi^* \in H_L( \mathcal{T}(m) )^c \,  \right\}    \right).
\end{equation}
Let  $Z_n:= \sum_{m \in T_{(n)} }  \mathds{1}\left[ \, \zeta \in H_L(\mathcal{T}(m))^c  \, \right]$, i.e., $Z_n$ denotes the number of bottom-level boxes in
which local supercriticality fails for $\zeta$.   Next we will prove that the inclusion
\begin{equation}\label{bad_inclusion_supercrit}
  \bigcap_{m  \in T_{(n)}} \left\{ \, \xi^* \in H_L( \mathcal{T}(m) )^c \,  \right\} \subseteq \{   |\mathcal{X} \setminus \mathcal{X}_{T+1}| +
  \sum_{x \in \mathcal{X} } \left( \zeta(x) - \zeta^{*}(x) \right) + Z_n \geq  2^n   \}
\end{equation}
holds.
Indeed: first recall that $\zeta(x) \geq \zeta^*(x)$ for any $x \in \mathcal{X}$ (cf.\ \eqref{zeta_geq_zeta_star}).
Also note that by Lemma \ref{lemma_coupling} (cf.\ the definitions \eqref{xi_star} and \eqref{zeta_star}) we have
$\xi^*(x)=\zeta^*(x)$ for all  $x \in \mathcal{X}_{T+1}$. Using these observations we will also show that
 if $ \xi^* \in H_L( \mathcal{T}(m) )^c $  holds for some $m \in  T_{(n)}$ then at least one of the following three events must happen:
\begin{enumerate}[(i)]
\item \label{first_i} there exists  $x \in B(\mathcal{T}(m),2L) \setminus \mathcal{X}_{T+1}$,
\item \label{second_ii} there exists
$x \in B(\mathcal{T}(m),2L)$ for which $\zeta(x)=1$ and $\zeta^*(x)=0$,
\item \label{third_iii} $ \zeta \in H_L( \mathcal{T}(m) )^c $  holds.
\end{enumerate}
Indeed, if \eqref{first_i} and \eqref{second_ii} does not occur then we have $\xi^*(x)=\zeta^*(x)=\zeta(x)$ for all $x \in B(\mathcal{T}(m),2L)$, therefore
 \eqref{third_iii} must hold by our assumption that $ \xi^* \in H_L( \mathcal{T}(m) )^c $ holds.
 Taking into consideration that $|T_{(n)}|=2^n$
and that the union in the definition \eqref{mathcal_X_tree} of $\mathcal{X}$ is disjoint (cf.\ \eqref{disjoint_X_cardinality}),
we obtain \eqref{bad_inclusion_supercrit}.

\smallskip

It remains to bound the probability of the event on the r.h.s.\ of \eqref{bad_inclusion_supercrit}.
If we let $R_1:=R_0 \vee R_2$ (where $R_0$ and $R_2$ appear in the statements of Lemmas \ref{lemma_X_t_big} and \ref{lemma_survival_to_T_suffices}, respectively), then
for
any $R \geq R_1$, $n \in \mathbb{N}$ and  $\mathcal{T} \in \Lambda_{n,L}$ we have
\begin{align}
\label{ineq_X_T_plus_1_big}
\mathbb{P}\left(\, |\mathcal{X} \setminus \mathcal{X}_{T+1}| \geq  2^n /2 \, \right) & \stackrel{ \eqref{X_t_big_eq} }{\leq} \left( 2 C_d  \right)^{-2^n}, \\
\label{ineq_zeta_and_zeta_star_do_not_differ_that_much}
 \mathbb{P}\left( \sum_{x \in \mathcal{X} } \left( \zeta(x) - \zeta^{*}(x) \right) \geq  2^n /3 \right) & \stackrel{ \eqref{survival_to_T_suffices_eq} }{\leq} 3 \left( 2 C_d  \right)^{-2^n}, \\
\label{Z_n_largedev_bd}
\mathbb{P}\left( Z_n \geq  2^n /6 \right) & \leq \left( 2 C_d  \right)^{-2^n},
\end{align}
where \eqref{Z_n_largedev_bd} can be proved analogously to \eqref{Y_n_not_too_big} using that $Z_n \sim \mathrm{BIN}(2^n,q)$ with
$q \leq (4 C_d)^{-6}$ (cf.\ \eqref{local_supcrit_prob_sigma_lambda_T}) and choosing $N=2^n$, $a=1/6$, $p=(4 C_d)^{-6}$ and $\varepsilon=(2C_d)^{-1}$ in  Claim \ref{claim_binomial_largedev_bound}, since $p \leq \left( \frac{\varepsilon}{2} \right)^{1/a}$.
 Using the inequalities \eqref{ineq_X_T_plus_1_big}-\eqref{Z_n_largedev_bd} together with
\eqref{xi_star_invariant} and \eqref{bad_inclusion_supercrit}, the desired inequality \eqref{supercrit_spread_out_eq} follows by the union bound.
\end{proof}

It remains to prove Lemma \ref{lemma_survival_to_T_suffices}. We proceed with making the ideas of Remark \ref{remark_supercrit_plan_heu} rigorous.

\smallskip

 Recall from Definition \ref{def:good_boxes} that the box $\kappa Rz+[0,\kappa R)^d$ (where $z \in \mathbb{Z}^d$) is good for $\xi^{\sstar}$
 if each of its sub-boxes of form $R z' +[ 0,R)^d$ (where  $z' \in \mathbb{Z}^d$) contains at least $\alpha R^d$ infected sites in the configuration $\xi^{\sstar}$.
  Somewhat simplifying the notation introduced in \eqref{eq:def_of_omega_xi}, we define the configurations $\omega$ and $\overline{\omega}$ of zeros and ones on the coarse-grained lattice by
\begin{equation}\label{def_omega_balazs}
  \omega(z):= \mathds{1}\left[ \text{ $\kappa Rz+[0,\kappa R)^d$  is good for $\xi^{\sstar}$ } \right],
 \; \; \overline{\omega}(z):=1-\omega(z),
   \; \; z \in \mathbb{Z}^d.
\end{equation}

Let us define the indicators $\overline{\epsilon}(x)$ and $\epsilon(x)$ for any  $x \in \mathcal{X}$ by
\begin{equation}\label{def_epsilon_bad_event}
\overline{\epsilon}(x):= \mathds{1} \left[ \, \exists y  \, : \, Z_T(x,y)>0 \text{ and } \overline{\omega}\left( \left\lfloor \frac{y}{\kappa R} \right\rfloor \right)=1 \, \right],
\;\; \epsilon(x):=1-\overline{\epsilon}(x).
\end{equation}
Thus $\overline{\epsilon}(x)$ is the indicator of the bad event that there is a BRW particle with label $x$ in a bad box at time $T$.
The following lemma will be used in the proof of Lemma \ref{lemma_survival_to_T_suffices}.

\begin{lemma}[Few particles land on bad boxes]\label{lemma_ground_is_mostly_fertile}
 Having already fixed $\lambda$, $L$ and $T$ as above, there exists $R_2=R_2(\lambda,L,T) \in \mathbb{N}$ such that for
any $R \geq R_2$, any $n \in \mathbb{N}$ and any $\mathcal{T} \in \Lambda_{n,L}$ we have
\begin{equation}\label{ground_is_mostly_fertile_eq}
 \mathbb{P}\left( \sum_{x \in \mathcal{X} } \overline{\epsilon}(x) \geq \frac{2}{9} 2^n  \right) \leq 2 \left( 2 C_d  \right)^{-2^n}.
\end{equation}
\end{lemma}

Before we prove Lemma \ref{lemma_ground_is_mostly_fertile}, let us deduce Lemma \ref{lemma_survival_to_T_suffices} from it.

\begin{proof}[Proof of Lemma \ref{lemma_survival_to_T_suffices}]
Let us choose $R_2=R_2(\lambda, L, T)$ as in Lemma \ref{lemma_ground_is_mostly_fertile} and let us consider any $R \geq R_2$,
$n \in \mathbb{N}$ and any $\mathcal{T} \in \Lambda_{n,L}$. Let
\begin{equation}\label{Y_n_prome_def}
  Y'_n:=  \sum_{x \in \mathcal{X} } \left( \zeta(x) - \zeta^{*}(x) \right)\epsilon(x), \qquad s:=\left(2^{|B(2L)|+1} C_d \right)^9.
\end{equation}
  It is enough to prove
 \begin{equation}\label{momgen_ineq}
 \mathbb{E} \left( s^{Y'_n} \right)  \leq 2^{ |B(2L)| 2^n},
 \end{equation}
  because then  $\mathbb{P}\left( Y'_n \geq \frac{1}{9} 2^n  \right) \leq  \left( 2 C_d  \right)^{-2^n}$ follows by the exponential Chebyshev's inequality, and this bound together with \eqref{ground_is_mostly_fertile_eq} implies  \eqref{survival_to_T_suffices_eq}.

In order to prove \eqref{momgen_ineq},
we  define the sigma-algebra $\mathcal{G}$ generated by the random variables $(Z_t(x,y)), x \in \mathcal{X}, y \in \mathbb{Z}^d, 0 \leq t \leq T$ and $(\xi^{\sstar}(y) )_{y \in \mathbb{Z}^d}$.

 Note that it follows from the definitions  \eqref{zeta_def_supcrit}, \eqref{def_omega_balazs} and \eqref{def_epsilon_bad_event} that
 \begin{equation}\label{G_measurable}
 \left(\zeta(x)\right)_{x \in \mathcal{X}}, \; \left(\omega(z)\right)_{z \in \mathbb{Z}^d}, \; \; \left(\epsilon(x)\right)_{x \in \mathcal{X}}\;
 \text{ are all $\mathcal{G}$-measurable.}
 \end{equation}
  Observe that
\begin{equation}\label{condindep_zeta_star}
 \zeta^*(x), x \in \mathcal{X} \text{ are conditionally independent given } \mathcal{G},
\end{equation}
 since the BRW particles used in the definition of $\zeta^*(x), x\in \mathcal{X}$ (cf.\ \eqref{zeta_star})  reproduce and die independently from each other and $\xi^{\sstar}$ on $[T,T+1]$  (cf.\ Definitions \ref{def_brw_from_X} and \ref{def_ingredients_supercrit}).
Recalling the definitions of $p_*$ from \eqref{p_star} and $Z_T(x)$ from \eqref{Z_T_x},  we will prove
\begin{equation}\label{condexp_fertile_landing}
  \mathbb{E}\left( \left(\zeta(x) - \zeta^{*}(x) \right)\epsilon(x) \,   | \,  \mathcal{G} \right) \leq  p_*^{Z_T(x)} \zeta(x), \qquad x \in \mathcal{X}.
\end{equation}
 Before we show \eqref{condexp_fertile_landing}, let us deduce \eqref{momgen_ineq} from it:
\begin{multline} \label{here_is_where_we_use_many_children}
 \mathbb{E} \left( s^{Y'_n} \right)
   \stackrel{ \eqref{Y_n_prome_def} }{=}
   \mathbb{E} \left( \mathbb{E} \left( \prod_{x \in \mathcal{X} } s^{\left( \zeta(x) - \zeta^{*}(x) \right)\epsilon(x)} \, \big| \, \mathcal{G} \right) \right)
   \stackrel{ \eqref{G_measurable}, \eqref{condindep_zeta_star} }{=}
   \\
 \mathbb{E} \left(  \prod_{x \in \mathcal{X} } \mathbb{E} \left(  s^{\left( \zeta(x) - \zeta^{*}(x) \right)\epsilon(x)} \, \big| \, \mathcal{G} \right) \right)
    \stackrel{ \eqref{zeta_geq_zeta_star}, \eqref{condexp_fertile_landing} }{\leq } \mathbb{E}\left( \prod_{x \in \mathcal{X} }  \left( s p_*^{Z_T(x)} \zeta(x) +1  \right) \right)
    \stackrel{ \eqref{zeta_def_supcrit} }{=}
    \\
   \prod_{x \in \mathcal{X} }  \mathbb{E}  \left( s p_*^{Z_T(x)} \mathds{1}\left[ \, 0<  Z_T(x)  \, \right] +1  \right)
   \stackrel{ \eqref{many_children_Z_T}, \eqref{Y_n_prome_def} }{\leq } \prod_{x \in \mathcal{X} }   \left( 1 +1  \right)
   \stackrel{ \eqref{disjoint_X_cardinality} }{=} 2^{  |B(2L)| 2^n }.
\end{multline}
It remains to show \eqref{condexp_fertile_landing}. We begin by observing that
\begin{multline}\label{fertile_landing_alternate_form}
 \mathbb{E}\left( \left(\zeta(x) - \zeta^{*}(x) \right)\epsilon(x) \,   | \,  \mathcal{G} \right) \stackrel{\eqref{zeta_geq_zeta_star}, \eqref{G_measurable} }{=}
 \mathbb{E}\left( 1 - \zeta^{*}(x) \,   | \,  \mathcal{G} \right) \zeta(x) \epsilon(x)\stackrel{ \eqref{zeta_star} }{=} \\
   \mathbb{P}\left(\, \forall y' \in \mathbb{Z}^d \, : \;  Z_{T+1}(x,y') \cdot  \xi^{\sstar}(y')=0
       \, \big| \,  \mathcal{G}\, \right) \zeta(x) \epsilon(x).
\end{multline}
The goal is to upper bound the r.h.s.\ of \eqref{fertile_landing_alternate_form} by  $p_*^{Z_T(x)} \zeta(x)$. At time $T$ there are $Z_T(x)$ BRW particles with label $x$, so it is enough to show that if $\epsilon(x)=1$ then each of these particles, with probability at least $1-p_*$, independently from the others, will produce an offspring which is located at a site
 with positive $\xi^{\sstar}$ value at time $T+1$. In order to show this, first
observe that if $\epsilon(x)=1$ then for any $y$ the condition $Z_T(x,y)>0$ implies $\omega\left( \lfloor \frac{y}{\kappa R} \rfloor \right)=1$ (cf.\ \eqref{def_epsilon_bad_event}), which in turn implies
  $\sum_{y' \in B(y,R)} \xi^{\sstar}(y') \geq \alpha R^d$ (cf.\  \eqref{def_omega_balazs}). So if there is a particle at time $T$ at location $y$, which
  \begin{enumerate}[(i)]
  \item \label{i_success} does not die in $[T,T+1]$,
  \item \label{ii_success} produces at least one child in $[T,T+1]$,
  \item \label{iii_success} this child lands at a site $y'$ with $\xi^{\sstar}(y')=1$,
  \item \label{iiii_success} this child does not die until $T+1$,
\end{enumerate}
then $Z_{T+1}(x,y') \cdot  \xi^{\sstar}(y')>0$. Now \eqref{i_success} occurs with probability $e^{-1}$, given this \eqref{ii_success} occurs
with probability $1-e^{-\lambda}$, given this \eqref{iii_success} occurs with probability at least $\alpha R^d /|B(R)| \geq \alpha / 3^d$, given this
\eqref{iiii_success} occurs with probability at least $e^{-1}$. Altogether, the probability that \eqref{i_success}-\eqref{iiii_success} all occur is at least $1-p_*$ (cf.\
\eqref{p_star}). Since in $[T,T+1]$ the BRW particles reproduce and die independently from each other and  $\mathcal{G}$, we indeed obtain that the r.h.s.\ of
\eqref{fertile_landing_alternate_form} is upper bounded by the r.h.s.\ of \eqref{condexp_fertile_landing}. The proof
of Lemma \ref{lemma_survival_to_T_suffices} is complete.
\end{proof}

It remains to prove Lemma \ref{lemma_ground_is_mostly_fertile}. Given some $r \in \mathbb{N}$ let us define
\begin{equation}\label{delta_R_def_eq}
  \overline{\delta}_r(x):=\mathds{1} \left[ 0 < \sum_{y \in \mathbb{Z}^d \setminus B(x,r\kappa R)} Z_T(x,y)   \right], \; \;
  \delta_r(x):=1-\overline{\delta}_r(x), \; \;
   x \in \mathcal{X},
\end{equation}
thus  $\overline{\delta}_r(x)$ is the indicator of the bad event that a particle with label $x$ travels too far from $x$.
In the
 next lemma we control the number of $x \in \mathcal{X}$ for which this bad event occurs.
This bound will be useful in the proof of Lemma \ref{lemma_ground_is_mostly_fertile}.
\begin{lemma}[Few particles travel too far]\label{lemma_no_travel_far}
 Having fixed  $\lambda$, $\kappa$, $L$ and $T$ as above, we can choose $r \in \mathbb{N}$ (and fix it for the rest of Section \ref{subsection_upper_bound}) such that for
any $R \in \mathbb{N}$, any $n \in \mathbb{N}$ and any $\mathcal{T} \in \Lambda_{n,L}$ we have
\begin{equation}\label{no_travel_far_eq}
 \mathbb{P}\left( \sum_{x \in \mathcal{X} } \overline{\delta}_r(x) \geq  2^n /9 \right) \leq  \left( 2 C_d  \right)^{-2^n}.
\end{equation}
\end{lemma}
\begin{proof}
 The random variables $\left(\overline{\delta}_r(x) \right)_{x \in \mathcal{X}}$ are i.i.d., since the BRWs with different labels are independent (cf.\ Definition \ref{def_brw_from_X}), therefore
 $$
 \sum_{x \in \mathcal{X} } \overline{\delta}_r(x) \sim \mathrm{BIN}\left( |B(2L)| 2^n, \mathbb{E}\left( \overline{\delta}_r(x) \right)  \right)
 $$
 by \eqref{disjoint_X_cardinality}.
The inequality \eqref{no_travel_far_eq} will follow by choosing
$N= |B(2L)| 2^n $, $a=(9 |B(2L)|)^{-1}$, $p=\mathbb{E}\left( \overline{\delta}_r(x) \right)$ and $\varepsilon=(2 C_d)^{-1/|B(2L)|}$  in Claim \ref{claim_binomial_largedev_bound} if we show $p \leq \left( \frac{\varepsilon}{2} \right)^{1/a}$, i.e.,
$\mathbb{E}\left( \overline{\delta}_r(x) \right) \leq  \left( 2^{|B(2L)|+1} C_d \right)^{-9}$. Thus we bound
\begin{equation}\label{brw_rw_expect}
  \mathbb{E}\left( \overline{\delta}_r(x) \right)\stackrel{ \eqref{delta_R_def_eq} }{ \leq } \sum_{y \in \mathbb{Z}^d \setminus B(x,r\kappa R)} \mathbb{E}\left( Z_T(x,y) \right)
  \stackrel{ \eqref{rw_brw_expect_correspond} }{=} e^{(\lambda-1)T} \mathbb{P}\left( X_T > r\kappa R \right),
\end{equation}
where $(X_t)$ denotes the continuous-time $R$-spread-out random walk on $\mathbb{Z}^d$ with jump rate $\lambda$ with $X_0= 0$ (cf.\ Definition \ref{def_R_spread_out_rw}).
If we denote by $N_T$ the number of jumps that $(X_t)$ performs on $[0,T]$ then we have $N_T \sim \mathrm{POI}(\lambda T) $ and $\mathbb{P}\left( X_T > r\kappa R \right)\leq \mathbb{P}\left( N_T > r \kappa \right) $, so if we choose $r$ big enough that
$\mathbb{P}\left( N_T > r \kappa \right) \leq  e^{(1-\lambda)T} \left( 2^{|B(2L)|+1} C_d \right)^{-9}$ holds then the statement of Lemma \ref{lemma_no_travel_far} also holds with the same choice of $r$.
\end{proof}

\begin{proof}[Proof of Lemma \ref{lemma_ground_is_mostly_fertile}]
We have already fixed $r$ in the statement of Lemma \ref{lemma_no_travel_far}. It is enough to show that
\begin{equation}\label{no_travel_far_epsilon_not_many}
\exists  R_2 \in \mathbb{N}  : \, \forall  R \geq R_2, \, n \in \mathbb{N}, \, \mathcal{T} \in \Lambda_{n,L}  : \,
\mathbb{P}\left( \sum_{x \in \mathcal{X} } \delta_r(x)\overline{\epsilon}(x) \geq  \frac{2^n}{9} \right) \leq \left( 2 C_d  \right)^{-2^n},
\end{equation}
because \eqref{no_travel_far_eq} and \eqref{no_travel_far_epsilon_not_many} together give the desired \eqref{ground_is_mostly_fertile_eq}.

In order  to prove \eqref{no_travel_far_epsilon_not_many}, we need some notation. For any $z \in \mathbb{Z}^d$ and $x \in \mathcal{X}$ we
say that  $z$ endangers $x$ and denote
\begin{equation}\label{endangers_def}
 z \stackrel{R}{\rightarrow} x \quad \text{ if } \quad  \left(z \kappa R + [0,\kappa R)^d \right)\cap B(x,r\kappa R) \neq \emptyset.
 \end{equation}
Let us also introduce $\mathcal{Z}=\mathcal{Z}( \mathcal{T}, R ) \subset \mathbb{Z}^d $ and $K=K(\mathcal{T},R) \in \mathbb{N} $ by
\begin{equation}\label{Z_K}
  \mathcal{Z}:=\{\, z \in \mathbb{Z}^d \, : \, \exists \, x \in \mathcal{X} \, : \, z \stackrel{R}{\rightarrow} x \, \},
 \quad K:= \max_{z \in \mathbb{Z}^d } \sum_{ x \in \mathcal{X} } \mathds{1}\left[ \,  z \stackrel{R}{\rightarrow} x \, \right].
\end{equation}
Note that we did not emphasize the dependence of  $\mathcal{Z}$ and $K$ on $\lambda$, $\kappa$, $L$, $T$ and $r$, because the values of these constants
have already been fixed.
Next we observe that it follows from \eqref{def_epsilon_bad_event}, \eqref{delta_R_def_eq}, and \eqref{Z_K} that
\begin{equation}\label{delta_eps_ruin_simple}
  \delta_r(x)\overline{\epsilon}(x)  \leq
  \mathds{1}\left[ \,  \exists \, z \in \mathcal{Z} \, : \, z \stackrel{R}{\rightarrow} x \text{ and } \overline{\omega}(z)=1 \, \right]  , \quad x \in \mathcal{X}.
\end{equation}
We can therefore proceed with the proof of \eqref{no_travel_far_epsilon_not_many} by bounding
\begin{equation} \label{ruined_x_interchange_sums_bound}
\sum_{x \in \mathcal{X} }  \delta_r(x)\overline{\epsilon}(x) \stackrel{ \eqref{delta_eps_ruin_simple} }{\leq}
\sum_{x \in \mathcal{X} }\sum_{z \in \mathcal{Z}}
\mathds{1}\left[ \,  z \stackrel{R}{\rightarrow} x \, \right]\overline{\omega}(z) \stackrel{ \eqref{Z_K} }{\leq} K  \sum_{z \in \mathcal{Z}} \overline{\omega}(z).
\end{equation}
It follows from Theorem \ref{thm:density_good_boxes} and \eqref{def_omega_balazs} that
\begin{equation}\label{bad_boxes_binomial}
 \text{$\sum_{z \in \mathcal{Z}} \overline{\omega}(z)$ is stochastically dominated by $\mathrm{BIN}( |\mathcal{Z}|, q_R)$ },
 \end{equation}
where $q_R=\exp\left(-\gamma R^d \right)$ if $R \geq R_{\sstar}=R_{\sstar}(\lambda)$. Next we will show that
\begin{align}
\label{Z_card_bound_short}
 \exists  C^* : & \, \forall  R \in \mathbb{N},\, n \in \mathbb{N},\, \mathcal{T} \in \Lambda_{n,L}  : \,
 |\mathcal{Z}|= |\mathcal{Z}( \mathcal{T}, R )| \leq C^* 2^n, \\
\label{K_final_fractal_bounds}
 \exists  C^{\sstar} : & \, \forall  R \in \mathbb{N},\, n \in \mathbb{N},\, \mathcal{T} \in \Lambda_{n,L}  : \,
 K=K(\mathcal{T},R) \leq C^{\sstar} R^{\frac{\ln(2)}{\ln(6)}}.
\end{align}
We do not emphasize the dependence of  $C^*$ and $C^{\sstar}$ on $\lambda$, $\kappa$, $L$, $T$ and $r$, because the values of these constants
have already been fixed.
We begin with the proof of \eqref{Z_card_bound_short}:
\begin{multline}\label{Z_cardinality_bound}
|\mathcal{Z}| \stackrel{ \eqref{Z_K} }{\leq}  \sum_{z \in \mathbb{Z}^d} \sum_{x \in \mathcal{X} } \mathds{1}\left[ \,  z \stackrel{R}{\rightarrow} x \, \right]
\leq
 |\mathcal{X}| \max_{x \in \mathbb{Z}^d } \sum_{z \in \mathbb{Z}^d}  \mathds{1}\left[ \,  z \stackrel{R}{\rightarrow} x \, \right]
 \stackrel{ \eqref{disjoint_X_cardinality}, \eqref{endangers_def} }{\leq} \\
  |B(2L)| 2^n \max_{x \in \mathbb{Z}^d } | \kappa R \mathbb{Z}^d \cap B(x,(r+1)\kappa R)  |
 \leq |B(2L)| 2^n (2r+3)^d.
\end{multline}
In order to prove \eqref{K_final_fractal_bounds}, we first bound
\begin{multline}\label{K_bound_tree}
K \stackrel{  \eqref{mathcal_X_tree}, \eqref{disjoint_X_cardinality}, \eqref{Z_K}  }{=}
\max_{z \in \mathbb{Z}^d } \sum_{ m' \in T_{(n)} } \sum_{ x \in B( \mathcal{T}(m'),2L ) } \mathds{1}\left[ \,  z \stackrel{R}{\rightarrow} x \, \right]
\stackrel{ \eqref{endangers_def} }{\leq} \\
 |B(2L)| \max_{z \in \mathbb{Z}^d } \sum_{ m' \in T_{(n)} } \mathds{1}\left[ \, | z\kappa R- \mathcal{T}(m')  |\leq (r+1)\kappa R +2L   \, \right]
 \stackrel{(*)}{\leq}
 \\
 |B(2L)| \max_{m \in T_{(n)} } \sum_{ m' \in T_{(n)} } \mathds{1}\left[ \, | \mathcal{T}(m)-\mathcal{T}(m')  |\leq 2(r+1)\kappa R +4L   \, \right],
\end{multline}
where $(*)$ follows from the triangle inequality.
Next we show that the r.h.s.\ of \eqref{K_bound_tree} can be upper bounded by $|B(2L)| 2^{k_0}$, where $k_0$ is the smallest integer for which
 $L 6^{k_0-1} > 2(r+1)\kappa R +4L$. Indeed, if $\rho(m, m') \geq k_0$ (cf.\ \eqref{def_eq_lex_dist}) then
 $| \mathcal{T}(m)-\mathcal{T}(m') | \geq L_{k_0-1}=L 6^{k_0-1}$ by Lemma \ref{lemma_far_in_tree_far_in_embedding}, and
 the number of $m' \in T_{(n)}$ for which $\rho(m, m') < k_0$ is less than or equal to $2^{k_0}$, cf.\ \eqref{def_eq_lex_dist}.
   From these bounds the inequality  \eqref{K_final_fractal_bounds} easily follows using a bit of calculus.

 By \eqref{ruined_x_interchange_sums_bound}, \eqref{bad_boxes_binomial}, \eqref{Z_card_bound_short} and
\eqref{K_final_fractal_bounds} we see that in order to prove the desired \eqref{no_travel_far_epsilon_not_many}, it is enough to show that there exists  $R_2 \in \mathbb{N}$ such that  for all $R \geq R_2$ and $n \in \mathbb{N}$
we have \begin{equation}
\mathbb{P}\left( C^{\sstar} R^{\frac{\ln(2)}{\ln(6)}}  Z_n' \geq \frac{2^n}{9}    \right) \leq \left( 2 C_d  \right)^{-2^n}, \quad \text{where} \quad Z_n' \sim \mathrm{BIN}(C^* 2^n, q_R).
\end{equation}
We want to apply Claim \ref{claim_binomial_largedev_bound} with $N=C^* 2^n$, $a=(9 C^{\sstar}  R^{\frac{\ln(2)}{\ln(6)}} C^* )^{-1}$, $p=q_R=\exp\left(-\gamma R^d \right)$ and $\varepsilon=(2 C_d)^{-{1/C^*}}$, so we need $p \leq \left( \frac{\varepsilon}{2} \right)^{1/a}$, i.e., $\exp\left(-\gamma R^d \right) \leq
(2^{C^*+1} C_d)^{-9C^{\sstar} R^{\frac{\ln(2)}{\ln(6)}}}$, but this inequality clearly holds for large enough $R$, since $\frac{\ln(2)}{\ln(6)} <d$.
The proof of Lemma \ref{lemma_ground_is_mostly_fertile} is complete.
\end{proof}

The proof of $\limsup_{R \to \infty } \lambda_p(R) \leq \frac{1}{1-p_c}$ is complete.

\begin{remark} \label{rem_sketch_measure} Let us  sketch how the methods of Section \ref{subsection_upper_bound} can be used to show that for any $\lambda >1$ the sequence of probability measures $(\mu_{\lambda,R})_{R=1}^{\infty}$ weakly converges   to the Bernoulli product measure $\pi_{\sigma(\lambda)}$ (cf.\ \eqref{sigma_lambda})  as $R \to \infty$.
It is enough to show that for any  finite subset $\mathcal{X}$ of $\mathbb{Z}^d$ and any $\varepsilon>0$ the total variation distance
$\mathrm{d}_{\mathrm{TV}}(  \mu_{\lambda,R} |_{\mathcal{X}},   \pi_{\sigma(\lambda)}|_{\mathcal{X}})$ is less than or equal to $\varepsilon$ if $R$ is large enough.
Similarly to Claim \ref{claim_supcrit_T_choice}, let us fix $T$ big enough so that $\mathrm{d}_{\mathrm{TV}}(  \pi_{\sigma(\lambda)}|_{\mathcal{X}} ,   \pi_{\sigma(\lambda,T)}|_{\mathcal{X}})\leq \varepsilon/2 $ (cf.\ \eqref{sigma_lamb_T}) and $ \mathbb{E}\left( \, p_*^{Z_T} \mathds{1}[ \, 0 < Z_T \, ] \, \right) \leq \frac{\varepsilon}{8|\mathcal{X}| }$ (cf.\ Definition \ref{def_bp_pop_at_t} and \eqref{p_star}). It remains to show $\mathrm{d}_{\mathrm{TV}}(  \mu_{\lambda,R} |_{\mathcal{X}},   \pi_{\sigma(\lambda,T)}|_{\mathcal{X}}) \leq \varepsilon/2$. Given our $\mathcal{X}$, let us define $\xi^*(x), \zeta^*(x)$ and $\zeta(x)$ for all $x \in \mathcal{X}$ as in Definition \ref{def_ingredients_supercrit}.
 It is enough to prove that $\mathbb{P}(\, \exists \, x \in \mathcal{X} \, : \, \xi^*(x) \neq  \zeta(x) \,)\leq \varepsilon/2 $ holds if $R$ is big enough, since
 $ \xi^*|_{\mathcal{X}} \sim  \mu_{\lambda,R} |_{\mathcal{X}}$ and
$\zeta|_{\mathcal{X}} \sim \pi_{\sigma(\lambda,T)}|_{\mathcal{X}}$. It is enough to show that
$\mathbb{P}( \xi^*(x) \neq  \zeta^*(x) )\leq \frac{\varepsilon}{4 |\mathcal{X}|} $ and $\mathbb{P}( \zeta^*(x) \neq  \zeta(x) )\leq \frac{\varepsilon}{4 |\mathcal{X}|} $ hold  for every $x \in \mathcal{X}$ if $R$ is big enough. Similarly to the proof of  Lemma \ref{lemma_supercrit_spread_out}, we have
\begin{multline*}
  \mathbb{P}( \xi^*(x) \neq  \zeta^*(x) ) \leq \mathbb{P}( \, \exists \, y \in \mathbb{Z}^d \, : \, \Xi_{T+1}(x,y) \neq Z_{T+1}(x,y) ) \stackrel{ \eqref{coupling_of_Xi_Z_on_X_t_eq} }{\leq} \\
   \mathbb{P}(x \notin \mathcal{X}_{T+1})
\stackrel{ \eqref{crowd_x}, \eqref{crowd_x_x_prime} }{\leq}      \mathbb{P}( \tau_x \leq T+1 ) + \sum_{y \in \mathcal{X} \setminus \{x \} }
   \mathbb{P}( \tau_{x,y} \leq T+1 ),
\end{multline*}
and the r.h.s.\ is smaller than $\frac{\varepsilon}{4 |\mathcal{X}|}$ if $R$ is large enough by Lemma \ref{lemma_crowd_bounds}.

Recalling the definition of the indicators $\epsilon(x)$ and $\overline{\epsilon}(x)$ from \eqref{def_epsilon_bad_event}, we have
\begin{equation*}
\mathbb{P}\left( \zeta^*(x) \neq  \zeta(x) \right) \stackrel{ \eqref{zeta_geq_zeta_star} }{\leq} \mathbb{E}\left[ (\zeta(x)-\zeta^*(x))\epsilon(x)  \right] +
 \mathbb{E}\left[ \overline{\epsilon}(x)\right] \leq \frac{\varepsilon}{8|\mathcal{X}| } + \frac{\varepsilon}{8|\mathcal{X}| } = \frac{\varepsilon}{4|\mathcal{X}| }
\end{equation*}
if $R$ is large enough, since
 $\mathbb{E}\left[ (\zeta(x)-\zeta^*(x))\epsilon(x)  \right]\leq \mathbb{E}\left( \, p_*^{Z_T} \mathds{1}[ \, 0 < Z_T \, ] \, \right) \leq \frac{\varepsilon}{8|\mathcal{X}| } $ holds by \eqref{condexp_fertile_landing} and our assumptions on $T$, while one can show that $\mathbb{E}\left[ \overline{\epsilon}(x)\right] \leq \frac{\varepsilon}{8|\mathcal{X}| }$ holds
 if $R$ is large using similar (but simpler) ideas as the ones that we used to prove Lemma \ref{lemma_ground_is_mostly_fertile}.
\end{remark}

\section{Appendix: Liggett-Steif stochastic domination for discrete time}
\label{subsection_appendix}

Fix~$p \in (0,1)$ and recall from Definition \ref{def_cellular_automaton} the notion of the discrete-time process~$(\eta_n)_{n \geq 0}$ with parameter~$p$. Also recall from Definition \ref{def_cellular_upper_inv} the notion of~$\nu_p$, the upper invariant measure of this process. Our goal in this section is to prove Theorem~\ref{thm:apply_ls} by recalling the steps of the proof of the analogous result for the upper invariant measure of the contact process in~\cite{LS06}.

By Claim~\ref{cl:oriented_percolation}, it suffices to prove Theorem~\ref{thm:apply_ls} for the case~$d = 1$, so we assume this from now on.

\subsection{Downward FKG property}
One of the ingredients of the proof of Theorem~\ref{thm:apply_ls} is Proposition \ref{prop:vdb} below.

If $\eta \in \{0,1\}^{\mathbb{Z}}$ and $\Lambda \subseteq \mathbb{Z}$,  we say that $\eta \equiv 0$ on $\Lambda$ if $\eta(z)=0$ for all $z \in \Lambda$.
\begin{definition}\label{downward_fkg}
 Let~$\mu$ be a probability measure on~$(\{0,1\}^{\Z}, \mathcal{F})$.
\begin{enumerate}
\item $\mu$ is called \emph{positively associated} if, for any two $\mathcal{F}$-measurable and increasing sets~$B, B' \subset \{0,1\}^{\Z}$, we have~$\mu(B \cap B') \ge \mu(B) \mu(B')$.
\item $\mu$ is called \emph{downward FKG} if, for any finite~$\Lambda \subset {\Z}$, the conditional measure~$\mu(\cdot \mid \eta \equiv 0 \text{ on } \Lambda)$ is positively associated.
\end{enumerate}
\end{definition}
\begin{proposition}\label{prop:vdb}
The  measure~$\nu_p$   is downward FKG.
\end{proposition}
Note that the continuous-time analogue of this result (i.e., that the upper invariant measure of the contact process on $\mathbb{Z}$ is downward FKG) is
proved in \cite{vdBHK06}, see equation (20) therein.  As it is pointed out in the remark after \cite[Theorem 2.1]{LS06}, the paper \cite{vdBHK06} is mainly devoted to the discrete time setting, and the continuous time results are deduced from them.
We decided to omit the details of the proof of Proposition \ref{prop:vdb}, which can be deduced from \cite[Theorem 3.1]{vdBHK06} analogously to \cite[(20)]{vdBHK06}.

\subsection{Liggett's auxiliary renewal measure}\label{subsect_liggett_renewal}

We now follow~\cite{Li95} to give a number of definitions. Define the function~$F:\mathbb{N} \to \mathbb{R}$ by letting
$$F(1) =  1 \quad \text{and}\quad F(k) = -\frac{2}{(a-b)^2}\sum_{j=0}^k c_j c_{k-j}a^{2j}b^{2(k-j)},\;\;k > 1,$$
where
\begin{equation} \label{a_b_c_k}
a = \frac{1-p+\sqrt{1-p}}{p},\;
 b= \frac{1-p-\sqrt{1-p}}{p},\;
  c_k = \frac{(2k)!}{4^k(k!)^2(2k-1)},\; k \geq 0.
\end{equation}

The following proposition is  a special case of \cite[Proposition~2.1]{Li95} so we omit the proof.
\begin{proposition}\label{liggett_formula_F} If
$p \geq \frac{3}{4}$ then~$F$ is positive and non-increasing, moreover the inequality~$\sum_k F(k) < \infty$ also holds.
\end{proposition}

   The following estimate will also be useful.
\begin{lemma}\label{lem:bound_psi} For any~$k \geq 2$ we have
	$$F(k) \leq \frac{4}{(a-b)^2}a^{2k}.$$
\end{lemma}
\begin{proof}
	Note that~$a \geq 0$,~$b < 0$ with~$a \geq |b|$,~$c_0 = -1$ and~$c_k \in [0,1]$ for~$k \geq 1$ (since~${2k\choose k} \leq 4^k$). Hence, in the sum in the definition of~$F(k)$, the first and last terms are negative and the other terms are non-negative, and we have
	$$F(k) \leq -\frac{2}{(a-b)^2} (c_0 c_k a^0 b^{2k} + c_kc_0a^{2k}b^0) \leq \frac{4}{(a-b)^2}a^{2k}.$$
\end{proof}

\begin{definition}\label{def_renew_measure}
 By Proposition \ref{liggett_formula_F},
 we can define the probability distribution~$\psi$  on~$\mathbb{N}$ by setting~$\psi([n,\infty)) = F(n)$, and also the (unique) translation invariant renewal measure~$\Psi$ on~$\{0,1\}^\Z$ with the property that the distance between two successive 1's are independent and with distribution~$\psi$. We denote by $\eta'$ the
 random element of $\{0,1\}^\Z$ with distribution $\Psi$.
\end{definition}
If $k \leq \ell \in \mathbb{Z}$, let us denote $[k,\ell]:=\{k,\dots,\ell\}$.
 Our next Proposition will easily follow from the results of~\cite{Li95}.
\begin{proposition}\label{prop:ligg_dom}
	Assume~$p \geq \frac{3}{4}$. For any~$k \in \mathbb{N}$,
	$$ \Psi(\, \eta' \equiv 0 \text{ on } [0,k]\, ) \ge \nu_p(\,  \eta \equiv 0 \text{ on }  [0,k+1] \, ) .$$
\end{proposition}
Before we prove Proposition \ref{prop:ligg_dom}, we state a corollary, which will be a key ingredient of Theorem~\ref{thm:apply_ls}.
 Recall the definition of $\theta_2(p)$ from \eqref{theta_2_R_def}.
\begin{corollary}\label{corollary_liminf_nu} If $p \geq \frac{3}{4}$ then
$  \liminf_{k \to \infty} \nu_p(\,  \eta \equiv 0 \text{ on }  [0,k] \, )^{1/k} \leq 1-\theta_2(p).$
\end{corollary}
The proof of Corollary \ref{corollary_liminf_nu} easily follows from Proposition \ref{prop:ligg_dom}, the formula $\Psi(\, \eta \equiv 0 \text{ on } [0,k-1]\, )=\Psi( \eta(0)=1 ) \sum_{i=1}^{\infty}F(k+i)$, Lemma \ref{lem:bound_psi}
 and the fact that
$a^2=1-\theta_2(p)$, cf.\ \eqref{a_b_c_k}. We omit the details.

Before we prove Proposition \ref{prop:ligg_dom}, let us introduce some further notation.
\begin{definition}\label{def_joint_measure_ori_renew}
We take a probability space with probability measure $P_{p,\Psi}$ under which (i) the independent Bernoulli($p$) random variables~$Z(z,n), z \in \mathbb{Z}, n \in \mathbb{N}$ required for the definition of the oriented percolation process
 (cf.\ Definition \ref{def_cellular_automaton})  and (ii) independently a random configuration~$\eta'$ with distribution~$\Psi$  are both defined.
Given some $A \subseteq \mathbb{Z}$, denote by $(\eta^{A}_n)$ the oriented percolation process started from initial state $\eta_0(z)=\mathds{1}[ \, z \in A  \, ]$.
\end{definition}
The proof of Proposition \ref{prop:ligg_dom} will follow from the next lemma.
\begin{lemma}\label{lem:or_perc1}
	For any~$k \in \mathbb{N}$, and any $p \geq \frac{3}{4}$ we have
	\begin{equation}\label{or_perc1_eq}
\lim_{n \to \infty} P_{p,\Psi}\left(\, \eta^{[0,k]}_n \cap \eta'= \varnothing  \,\right)= \nu_p \left(\,  \eta \equiv 0 \text{ on }  [0,k+1] \, \right).
\end{equation}
\end{lemma}
Notice that we have $k$ on the l.h.s., but $k+1$ on the r.h.s.
The proof of Lemma \ref{lem:or_perc1} follows from standard arguments of time reversal and duality for oriented percolation. We include a proof for completeness, but postpone it to Section~\ref{subsection_oriented_paths_and_duality}.

\begin{proof}[Proof of Proposition \ref{prop:ligg_dom}]
	We start by noting that the transition rules of our process~$(\eta_n)_{n \geq 0}$ with parameter~$p$ are the same as the ones of the process~$(A_n)_{n\geq 0}$ of~\cite{Li95}, with parameters~$q = p$ in the notation of that paper (see the beginning of the Introduction there). The assumption for the aforementioned Proposition~2.1 of~\cite{Li95} is that~$p > \frac12$ and~$q \geq 4p(1-p)$, which when~$p = q$ means simply that~$p \geq \frac34$.
	
\smallskip

	Equation~(1.4) in~\cite{Li95} gives $(*)$ below for any finite subset $A$ of $\mathbb{Z}$:
	
	$$\Psi(\, \eta' \cap A = \varnothing\, )=P_{p,\Psi}(\, \eta' \cap \eta^A_0 = \varnothing\, )  \stackrel{(*)}{\geq} P_{p,\Psi}(\, \eta' \cap \eta^A_1 = \varnothing\, ).
$$
	This can be iterated to yield
	$
	\Psi(\, \eta' \cap A = \varnothing\, )  \geq P_{p,\Psi}(\, \eta' \cap \eta^A_n = \varnothing\, )$,  $n \in \mathbb{N}.
	$
	Applying this to~$A=[0,k]$ and  using \eqref{or_perc1_eq}, we obtain Proposition \ref{prop:ligg_dom}.
\end{proof}

\subsection{$\nu_p$ dominates product Bernoulli}\label{subsect_oriented_perco_dominates_ber}
The goal of Section \ref{subsect_oriented_perco_dominates_ber} is to prove Theorem~\ref{thm:apply_ls}. The key ingredients are
Proposition~\ref{prop:vdb}, Corollary~\ref{corollary_liminf_nu} and the following result, which
is a variant of \cite[Proposition 2.2]{LS06}.
\begin{lemma}\label{lem:quase_dom} If $\nu$ is a probability measure on~$(\{0,1\}^{\Z}, \mathcal{F})$ with the downward FKG property (cf.\ Definition \ref{downward_fkg})
and
\begin{equation}
 \liminf_{k \to \infty} \nu(\,  \eta \equiv 0 \text{ on }  [0,k] \, )^{1/k} \leq 1-\theta
\end{equation}
holds for some $\theta \in (0,1)$, then
$\nu(\, \eta(0) = 1 \mid \eta \equiv 0 \text{ on } [1,k]\, ) \geq \theta$, $k \in \mathbb{N}$
\end{lemma}	
We omit the proof, which is a step-by-step replica of the proof of \cite[Proposition 2.2]{LS06}. Note that by
Proposition~\ref{prop:vdb} and Corollary~\ref{corollary_liminf_nu} we can apply Lemma \ref{lem:quase_dom} with $\nu=\nu_p$ and $\theta=\theta_2(p)$, from
which the next lemma follows.
\begin{lemma}\label{lem:quase2}
Let~$\Lambda, \Lambda' \subset \mathbb{N}$ be disjoint and finite. Then,
$$\nu_p(\, \eta(0)= 1 \mid \eta \equiv 0 \text{ on } \Lambda,\; \eta \equiv 1 \text{ on } \Lambda' \, ) \geq \theta_2(p). $$
\end{lemma}
\begin{proof}
This can be obtained as an immediate consequence of Lemma~\ref{lem:quase_dom} and another application of Proposition \ref{prop:vdb}, so we omit the details (see the proof of Proposition~2.3 in~\cite{LS06}).
\end{proof}
\begin{proof}[Proof of Theorem~\ref{thm:apply_ls}]
As observed in~\cite{LS06}, a coupling between~$\nu_p$ and the Bernoulli product measure $\pi_{\theta_2(p)}$ which establishes the required domination can now be constructed sequentially, using Lemma~\ref{lem:quase2}. This completes the proof of Theorem~\ref{thm:apply_ls}.
\end{proof}

\subsection{Oriented paths and duality: proof of Lemma~\ref{lem:or_perc1}}\label{subsection_oriented_paths_and_duality}

As already mentioned, the proof of Lemma~\ref{lem:or_perc1} will depend on standard tools of time reversal and duality.  Rather than considering a time reversal of the process~$(\eta_n)$ itself, it will be more convenient to pass to an auxiliary process~$(\hat{\eta}_n)$, and then consider its dual process~$(\doublehat{\eta}_n)$.

In order to define these processes, recall the Bernoulli($p$) random variables
\begin{equation}
\{Z(z,n):z \in \Z,\;n\in \mathbb{N}_0\}
\end{equation} used in the definition of~$(\eta_n)$, see~\eqref{eq:def_of_eta}. Let us introduce notation to refer to the oriented site percolation paths produced from~$Z(\cdot,\cdot)$.

\begin{definition}\label{def_active_path_op}
	Given~$w,z \in \Z$ and~$m \leq n \in \mathbb{N}_0$, we write~$(w,m) \rightsquigarrow (z,n)$ if there exists a sequence
	$w =z_0 ,\; z_1,\;\ldots,\; z_{n-m} = z$
	such that
	\begin{align*}&z_{i+1} - z_i \in \{-1,0\},\quad i \in \{0,\ldots, n-m-1\} \\&\qquad\text{and}\quad Z(z_i,m+i)=1,\quad i \in \{0,\ldots,n-m\}.\end{align*}
\end{definition}
(Note that, with this notation,~$(z,n) \rightsquigarrow (z,n)$ if and only if~$Z(z,n) = 1$).

Given~$\eta \in \{0,1\}^\Z$, define~$\mathcal{S}(\eta) \in \{0,1\}^\Z$ by
$$[\mathcal{S}(\eta)](z) = \max(\eta(z),\eta(z+1)),\quad z \in \Z.$$
It follows from~\eqref{eq:def_of_eta} that, given an initial configuration~$\eta_0 \in \{0,1\}^\Z$, we have
\begin{equation}\label{eq:or_paths_basic}
\eta_n(z) = \mathds{1}\left[ \, \exists z': [\mathcal{S}(\eta_0)](z') = 1 \text{ and } (z',1) \rightsquigarrow (z,n) \, \right],\quad n \geq 1,\;z \in \Z.\end{equation}
We now define a process~$(\hat{\eta}_n)_{n \geq 0}$ on~$\{0,1\}^\Z$ for arbitrary~$\hat{\eta}_0$  by setting
\begin{equation}\label{eq:or_paths}\hat{\eta}_n(z) = \mathds{1}\left[ \, \exists z': \hat{\eta}_0(z') = 1 \text{ and } (z',0) \rightsquigarrow (z,n) \, \right],\quad n \geq 1,\;z \in \Z.\end{equation}
Note that the evolution of~$(\hat{\eta}_n)$ is given by
\begin{align}
\label{hat_eta_first_step}
&\hat{\eta}_1(z) = Z(z,1)\cdot \max\left(Z(z,0)\cdot \hat{\eta}_0(z),\;Z(z+1,0) \cdot \hat{\eta}_0(z+1)\right),\\[.2cm]
\label{hat_eta_other_steps}
&\hat{\eta}_{n+1}(z) = Z(z,n+1) \cdot \max\left(\hat{\eta}_n(z),\; \hat{\eta}_n(z+1)\right),\quad n \geq 1
\end{align}
that is, apart from the first step, the evolution of~$(\hat{\eta}_n)$ follows the same rule as that of~$(\eta_n)$. The following observation, which follows from comparing~\eqref{eq:or_paths_basic} and~\eqref{eq:or_paths}, will be useful.
\begin{claim}\label{cl:or_1}
	For any~$\eta \in \{0,1\}^\Z$ and~$n \geq 1$, the law of~$\eta_n$ given~$\eta_0 = \eta$ is equal to the law of~$\hat{\eta}_{n-1}$ given~$\hat{\eta}_0 =\mathcal{S}(\eta)$.
\end{claim}

Recall the notation of the probability measure $P_{p,\Psi}$ from Definition \ref{def_joint_measure_ori_renew}, under which
we can define the process $(\hat{\eta}^{\mathcal{A}'}_n)$, where $\mathcal{A}'=\{ \, z \in \mathbb{Z}\, : \, \eta'(z)=1\}$, so that
$\hat{\eta}^{\mathcal{A}'}_0=\eta'$, where
  $\eta'$ is independent from $Z(z,n), z \in \mathbb{Z}, n \in \mathbb{N}$ and the law of $\eta'$ is the renewal measure $\Psi$ of Definition \ref{def_renew_measure}.
\begin{lemma} \label{lem:durr_or_perc}
	For any~$k \in \mathbb{N}$,
	\begin{equation}
\lim_{n \to \infty} P_{p,\Psi}\left(\, \hat{\eta}^{\mathcal{A}'}_n \equiv 0 \text{ on } [0,k] \, \right)=\nu_p(\, \eta \equiv 0 \text{ on } [0,k]  \, ).
\end{equation}
\end{lemma}
In other words, the law of $\hat{\eta}^{\mathcal{A}'}_n$ and the law of $\hat{\eta}^{\mathbb{Z}}_n$ converge weakly to the same limit $\nu_p$ as $n \to \infty$.
We omit the details of the proof, which can be derived from the fact that $\eta'$ has positive density using the results of~\cite{D84}.

We now introduce yet another process~$(\doublehat{\eta}_n)_{n \geq 0}$ by taking~$\doublehat{\eta}_0 \in \{0,1\}^\Z$ arbitrarily and setting
\begin{align*}
&\doublehat{\eta}_1(z) = Z(z,1)\cdot \max\left(Z(z,0)\cdot \doublehat{\eta}_0(z),\;Z(z-1,0) \cdot \doublehat{\eta}_0(z-1)\right),\\[.2cm]
&\doublehat{\eta}_{n+1}(z) = Z(z,n+1) \cdot \max\left(\doublehat{\eta}_n(z),\; \doublehat{\eta}_n(z-1)\right),\quad n \geq 1.
\end{align*}
That is,~$(\doublehat{\eta}_n)$ is defined in the same way as~$(\hat{\eta}_n)$, except that the direction of propagation is inverted. In particular,~$(\doublehat{\eta}_n)$ satisfies a similar property as~\eqref{eq:or_paths}, except that the percolation paths in the definition of~`$\rightsquigarrow$' should be replaced by paths that move to the right rather than to the left. With this observation at hand, one easily verifies the following duality relation.
\begin{claim}\label{cl:or_2}
	For any~$\bar{\eta},\bar{\bar{\eta}} \in \{0,1\}^\Z$ and~$n \geq 1$, we have
	$$P_p(\, \hat{\eta}_n \cap \bar{\bar{\eta}} = \varnothing \mid \hat{\eta}_0 = \bar{\eta}\, ) =
 P_p(\, \doublehat{\eta}_n \cap \bar{\eta} = \varnothing \mid \doublehat{\eta}_0 = \bar{\bar{\eta}} \, ).$$
\end{claim}

\begin{proof}[Proof of Lemma~\ref{lem:or_perc1}]
	For any~$k,n \in \mathbb{N}$ we have
	\begin{equation*}
P_{p,\Psi}\left( \eta_n^{[0,k]} \cap \eta'=\varnothing  \right)=
P_{p,\Psi}\left( \hat{\eta}_{n-1}^{[-1,k]} \cap \eta'=\varnothing  \right)=
P_{p,\Psi}\left( \doublehat{\eta}_{n-1}^{\mathcal{A}'} \cap [-1,k]=\varnothing  \right),
\end{equation*}
	where the first equality follows from Claim~\ref{cl:or_1} and the second from Claim~\ref{cl:or_2}. Next, by translation invariance and symmetry considerations, the right-hand side above is equal to $P_{p,\Psi}\left( \hat{\eta}_{n-1}^{\mathcal{A}'} \cap [0,k+1]=\varnothing  \right)$.
		Taking~$n \to \infty$, Lemma~\ref{lem:or_perc1}  follows using Lemma~\ref{lem:durr_or_perc}.
	
\end{proof}

\medskip

 {\bf Acknowledgements:}
 The work of B.~R\'ath is partially supported by
Postdoctoral Fellowship NKFI-PD-121165 and grant NKFI-FK-123962 of NKFI
(National Research, Development and Innovation Office), the Bolyai Research
Scholarship of the Hungarian Academy of Sciences and the \'UNKP-19-4-BME-85
New National Excellence Program of the Ministry for Innovation and Technology.  The authors would like to thank Thomas Beekenkamp and Markus Heydenreich for helpful discussions, Jan Swart for suggesting the reference~\cite{SV86} and Stein Andreas Bethuelsen for pointing out the argument presented in Remark \ref{remark_stein_ideas}.

{\footnotesize


\begin{thebibliography}{99}


\bibitem[BGP12]{bgp12}
Benjamini, I., Gurel-Gurevich, O., and Peled, R.
\textit{On K-wise Independent Distributions and Boolean Functions}.
\emph{arXiv:1201.3261}

\bibitem[BK89]{bk89} Burton, R. M., Keane, M. \textit{Density and uniqueness in percolation}. Communications in mathematical physics, 121(3), 501-505 (1989).

\bibitem[BNP11]{bnp11}
	Benjamini, I., Nachmias, A., Peres, Y. \textit{Is the critical percolation probability local?}. Probability Theory and Related Fields 149(1) (2011):261-9.

\bibitem[vB20]{vB20} van Belle, T.  \textit{Uniqueness of the infinite open cluster on the
	stationary distributions of interacting particle systems.} Master's Thesis in Mathematics, University of Groningen (2020). Available at http://fse.studenttheses.ub.rug.nl/23642/

\bibitem[vdBHK06]{vdBHK06} van den Berg, J., H\"aggstr\"om, O., Kahn, J. \textit{Some conditional correlation inequalities for percolation and related processes}. Random Structures and Algorithms. 9(4):417-35 (2006).

\bibitem[vdB11]{rob_sharp} J.\ van den Berg,
\textit{Sharpness of the percolation transition in the two-dimensional contact process}.
Annals of Applied Probability, Vol. 21, No. 1, 374-395 (2011).

\bibitem[vdBBH15]{markus_jakob_rob_sharp}
J.\ van den Berg, J. E. Bj\"ornberg, and M. Heydenreich,
\textit{Sharpness versus robustness of the percolation transition in 2d contact processes}.
Stochastic Processes and Applications,
Vol. 125, no. 2, 513-537 (2015).


\bibitem[vdBB18]{rob_stein}
J.\ van den Berg, S.\ A.\ Bethuelsen, \textit{Stochastic domination in space-time for the contact process}. Random Structures and Algorithms, 53(2), 221-237 (2018).


\bibitem[BG90]{BG90} C. \ Bezuidenhout, G. \ Grimmett, \textit{The critical contact process dies out}. The Annals of Probability, Vol. 18(4), pp.1462-1482 (1990).


\bibitem[BDS89]{bds89}
M.\ Bramson, R.\ Durrett, G.\ Swindle,
\textit{Statistical mechanics of crabgrass}.
The Annals of Probability, Vol. 17, No 2, 444-481 (1989).

\bibitem[CR85]{CR85}
M.\ Campanino, L.\ Russo,
\textit{An upper bound on the critical percolation probability for the three-dimensional cubic lattice}. The Annals of Probability, 13(2), 478-491 (1985).

\bibitem[DP96]{DP96} Deuschel, J-D.,  Pisztora, A. \textit{Surface order large deviations for high-density percolation}. Probability Theory and Related Fields 104, no. 4 (1996): 467-482.


\bibitem[D84]{D84} Durrett, R., \textit{Oriented percolation in two dimensions}. The Annals of Probability, Vol~12(4), pp.999-1040 (1984).

\bibitem[Har74]{Har74} T.\ Harris, \textit{Contact interactions on a lattice}. The Annals of Probability, pp.969-988 (1974).

\bibitem[H82]{H82} Y.\ Higuchi, \textit{Coexistence of the infinite (*) clusters: - A remark on the square lattice site percolation}.
 Zeitschrift f\"ur Wahrscheinlichkeitstheorie und Verwandte Gebiete, 61(1), 75-81  (1982).


\bibitem[vdH09]{remco} van der Hofstad, R. \textit{Random graphs and complex networks}. Available on http://www.win. tue.nl/rhofstad/NotesRGCN.pdf (2009).

\bibitem[Gr99]{Gr99} Grimmett, G. \textit{Percolation}. Springer-Verlag Berlin (Second edition) (1999).


\bibitem[K82]{Kesten82}
H. Kesten,  {\em Percolation Theory for Mathematicians}. Birkh\"auser, Boston (1982).

\bibitem[LL10]{LL10}
G. F. Lawler, V. Limic., \textit{Random walk: a modern introduction.} Vol. 123. Cambridge University Press (2010).

\bibitem[Li85]{Li85} Liggett,  T. \textit{Interacting particle systems}. Grundlehren der mathematischen Wissenschaften 276, Springer (1985).

\bibitem[Li95]{Li95} Liggett, T. \textit{Survival of discrete time growth models, with applications to oriented percolation}. The Annals of Applied Probability 613-636 (1995).

\bibitem[Li99]{Li99} Liggett, T. \textit{Stochastic interacting systems: contact, voter and exclusion processes}. Vol. 324. Springer Science \& Business Media, 1999.

\bibitem[LSS97]{LSS97}
T.\ Liggett, R.\ H.\ Schonmann, A.\ M.\ Stacey. \textit{Domination by product measures.} The Annals of Probability 25(1), 71-95 (1997).



\bibitem[LS06]{LS06}
T.\ Liggett, J.\ E.\ Steif,
\textit{Stochastic Domination: The Contact Process, Ising models and FKG Measures}.
Annales Institut Henri Poincare, Probabilites et Statistiques,
 42,  223-243 (2006).

\bibitem[MT17]{MT17} Martineau S., Tassion, V. \textit{Locality of percolation for abelian Cayley graphs}. The Annals of Probability 45(2), 1247-1277 (2017).


\bibitem[Pei36]{Pei36} R.\ Peierls, \textit{On Ising's model of ferromagnetism}.
Mathematical Proceedings of the Cambridge Philosophical Society,
  32(3), 477-481  (1936).

\bibitem[Pi96]{Pi96} Pisztora, A. \textit{Surface order large deviations for Ising, Potts and percolation models}. Probability Theory and Related Fields 104, no. 4 (1996): 427-466.



\bibitem[Ra15]{Ra15} R\'ath, B. \textit{A short proof of the phase transition for the vacant set of random interlacements}. Electronic Communications in Probability 20 (2015).

\bibitem[RV15]{RV15} R\'ath, B., and Valesin, D. \textit{Percolation on the stationary distributions of the voter model}. Annals of Probability  45 (3), 1899-1951 (2017).

\bibitem[RV17]{RV17}  R\'ath, B., and Valesin, D. \textit{On the threshold of spread-out voter model percolation}. Electronic Communications in Probability 22 (2017).




\bibitem[SXZ14]{sxz14} Song, H., Xiang, K.N., Zhu, S.C.H. \textit{Locality of percolation critical probabilities: uniformly nonamenable case}. arXiv preprint arXiv:1410.2453.

\bibitem[SV86]{SV86} Schonmann, R.H., Vares, M.E., \textit{The survival of the large dimensional basic contact process}. Probability Theory and Related Fields, 72(3), pp.387-393 (1986).

\bibitem[Sz12]{sznitman_decoupling}  Sznitman, A.-S.
\textit{Decoupling inequalities and interlacement percolation on $G \times \mathbb{Z}$}.
Inventiones mathematicae, 187, 3, 645-706  (2012).





\end{thebibliography}
\end{document}